\documentclass[12pt, reqno]{amsart}
\makeatletter
\DeclareMathAlphabet{\mathpzc}{OT1}{pzc}{m}{it}
\usepackage[margin=2.4cm]{geometry}
\usepackage{amsmath, amssymb, amsfonts, amstext, verbatim, amsthm, mathrsfs}
\usepackage[mathcal]{eucal}
\usepackage{microtype}

\usepackage[all,arc,knot,poly]{xy}
\usepackage{enumerate}
\usepackage{mathrsfs}
\usepackage{mathtools}
\usepackage{amsxtra}
\usepackage{needspace}
\usepackage{graphicx}
\usepackage{setspace}

\usepackage[usenames]{color}
\usepackage{aliascnt} 
\usepackage[colorlinks=true,linkcolor=blue,citecolor=blue,urlcolor=blue,citebordercolor={0 0 1},urlbordercolor={0 0 1},linkbordercolor={0 0 1}]{hyperref} 
\usepackage{enumerate}
\usepackage{xspace}

\usepackage{tikz}
\usetikzlibrary{decorations.pathreplacing,cd}
\usetikzlibrary{matrix,arrows}

\theoremstyle{plain}

\newcommand{\refnewtheoremn}[4]{
\newaliascnt{#1}{#2}
\newtheorem{#1}[#1]{#3}
\aliascntresetthe{#1}
\expandafter\providecommand\csname #1autorefname\endcsname{#4}}

\newcommand{\refnewtheorem}[3]{\refnewtheoremn{#1}{#2}{#3}{#3}}

\def\makeCal#1{
\expandafter\newcommand\csname c#1\endcsname{\mathcal{#1}}}
\def\makeBB#1{
\expandafter\newcommand\csname b#1\endcsname{\mathbb{#1}}}
\def\makeFrak#1{
\expandafter\newcommand\csname f#1\endcsname{\mathfrak{#1}}}

\count@=0
\loop
\advance\count@ 1
\edef\y{\@Alph\count@}
\expandafter\makeCal\y
\expandafter\makeBB\y
\expandafter\makeFrak\y
\ifnum\count@<26
\repeat

\newtheorem{thm}{Theorem}[section]
\newtheorem{theorem}{Theorem}[section]

\refnewtheorem{lemma}{thm}{Lemma}
\refnewtheorem{assumptions}{thm}{Assumptions}
\refnewtheorem{examples}{thm}{Examples}
\refnewtheorem{claim}{thm}{Claim}
\refnewtheorem{cor}{thm}{Corollary}
\refnewtheorem{conj}{thm}{Conjecture}
\refnewtheorem{prop}{thm}{Proposition}
\refnewtheorem{proposition}{thm}{Proposition}
\refnewtheorem{quest}{thm}{Question}

\refnewtheorem{assumption}{thm}{Assumption}
\refnewtheorem{problem}{thm}{Problem}

\theoremstyle{definition}
\refnewtheorem{remark}{thm}{Remark}
\refnewtheorem{remarks}{thm}{Remarks}
\refnewtheorem{defn}{thm}{Definition}
\refnewtheorem{definition}{thm}{Definition}
\refnewtheorem{notn}{thm}{Notation}
\refnewtheorem{const}{thm}{Construction}
\refnewtheorem{example}{thm}{Example}
\refnewtheorem{fact}{thm}{Fact}

\newcommand{\fg}{\mathfrak{g}}

\newcommand{\fh}{\mathfrak{h}}
\newcommand{\Coh}{\operatorname{Coh}}

\newcommand{\rr}{r}

\newcommand{\ch}{\operatorname{ch}}
\newcommand {\id}{\operatorname{id}}

\newcommand{\Stab}{\operatorname{Stab}}

\renewcommand{\Im}{\operatorname{Im}}
\renewcommand{\Re}{\operatorname{Re}}
\newcommand {\Hom}{\operatorname{Hom}}
\newcommand {\Aut}{\operatorname{Aut}}
\newcommand {\End}{\operatorname{End}}

\newcommand{\DT}{\operatorname{DT}}

\newcommand{\Ham}{\operatorname{Ham}}
\newcommand{\GL}{\operatorname{GL}}

\newcommand{\Li}{\operatorname{Li}}

\newcommand{\PGL}{\operatorname{PGL}}

\newcommand{\GV}{\operatorname{GV}}

\newcommand{\lra}{\longrightarrow}

\newcommand {\<}{\langle}
\renewcommand {\>}{\rangle}
\newcommand{\isom}{\cong}
\newcommand{\half}{\tfrac{1}{2}}
\newcommand{\tensor}{\otimes}
\renewcommand{\O}{\mathscr{O}}

\newcommand{\Vect}{\operatorname{vect}}

\newcommand{\gl}{\mathfrak{gl}}

\newcommand{\dual}{\vee}

\renewcommand{\b}{\,|\,}

\newcommand{\od}{{\rm od}}
\newcommand{\god}{{\fg}^{\rm od}}

\newcommand{\LC}{ \scriptscriptstyle{ \rm LC}}
\newcommand{\reg}{\operatorname{reg}}
\newcommand{\stokes}{\bS}
\renewcommand{\hbar}{\epsilon}

\begin{document}

\title{Geometry from Donaldson-Thomas invariants}
\author{Tom Bridgeland}

\date{}

\begin{abstract}
We introduce geometric structures on the space of stability conditions of a three-dimensional Calabi-Yau category which encode the Donaldson-Thomas invariants of the category. We explain in detail a  close analogy between these structures, which we call Joyce structures, and Frobenius structures. In the second half of the paper we give explicit calculations of Joyce structures in three  interesting classes of examples.\end{abstract}

\maketitle


\section{Introduction}

The aim of this paper is to use the Donaldson-Thomas (DT) invariants \cite{JS,KS} of a CY$_3$ triangulated category  to define a geometric structure on its space of stability conditions \cite{Stab,Ludmil}. We call the resulting structure a Joyce structure, since the most important ingredients already appear in the paper \cite{HolGen}. There  is a close analogy between the notion of a Joyce structure and that of a Frobenius structure \cite{D1,D2} which will be  our  main theme. Both structures involve isomonodromic pencils of flat connections on the  tangent bundle of a complex manifold $M$, although in the case of Joyce structures these connections are non-linear.

A key difference between Frobenius and Joyce structures  lies in their relationship to the corresponding  enumerative invariants.  The genus 0 Gromov-Witten invariants of a smooth, complex, projective variety $X$ can be encoded in a Frobenius structure on a small, possibly formal, ball in the vector space $H^*(X,\bC)$. The  pencil of flat connections is given directly by the derivatives of the prepotential, which in turn is  the generating function for the invariants. In this paper we  shall similarly encode the DT invariants of a CY$_3$ category in a Joyce structure on its space of stability conditions. But in this context, it is not  the flat connections themselves that are described  by the  invariants, but rather their generalised monodromy or Stokes data. 

\subsection{BPS structures}
Nothing in this paper requires any knowledge of the inner workings of DT theory.  The output of DT theory applied to a stability condition on a CY$_3$ triangulated category was axiomatised in \cite{RHDT} to give the  notion of a BPS structure. It consists of the following simple data:
\begin{itemize}
\item[(a)] a free abelian group $\Gamma\isom \bZ^{\oplus n}$ equipped with a skew-symmetric integral form $\<-,-\>$;

\item[(b)] a homomorphism of abelian groups
$Z\colon \Gamma\to \bC$ called the central charge;
\item[(c)] a map of sets
$\Omega\colon \Gamma\to \bQ$ encoding the DT invariants;
\end{itemize}
subject only to  the symmetry condition $\Omega(-\gamma)=\Omega(\gamma)$, and a weak finiteness condition known as the support property.
The most basic general property of DT invariants is the Kontsevich-Soibelman wall-crossing formula \cite{KS}, which describes how the invariants  change under variations of stability condition. This can  be   axiomatised  in the abstract notion of a variation of BPS structures over a manifold $M$. These concepts, which are special cases of structures introduced by Kontsevich and Soibelman, will be reviewed in Section \ref{next}.

To give a  variation of BPS structures should be viewed as being analogous to giving the Stokes factors of a semi-simple Frobenius manifold as a function of the canonical co-ordinates. Reconstructing the Frobenius structure from such data requires inverting the  Riemann-Hilbert (RH) correspondence for a class of meromorphic connections, which in turn,  involves solving  a RH boundary value problem \cite[Lecture 3]{D1}, \cite[Lecture 4]{D2} for maps from the complex plane to the group $\GL_n(\bC)$.  
The problem of reconstructing a Joyce structure from a variation of BPS structures is  entirely analogous, but  with the finite-dimensional group  $\GL_n(\bC)$ replaced by the infinite-dimensional group of Poisson automorphisms of an algebraic torus.  The corresponding RH boundary value problems were studied in detail in \cite{RHDT}.  

Although we have no general existence or uniqueness results for solutions to the relevant RH  problems, in the second part of the paper, Sections \ref{uncoupled}--\ref{sten}, we  make  some non-trivial calculations in several cases arising naturally in  DT theory. We hope that these results, which relate in an interesting way to  known structures in mirror symmetry, will provide adequate motivation for the further study of the analogy considered here. 

\subsection{Joyce structures}
\label{cont}

Let us now describe in a little more detail the essential features of a Joyce structure. The precise definition can be found   in Section \ref{boring} below.  The starting point is a complex manifold  $M$  equipped with a flat,  torsion-free  connection $\nabla$ on the  holomorphic tangent bundle $\cT_M$, and a covariantly constant skew-symmetric pairing
\[\eta\colon \cT^*_M\times \cT^*_M\to \bC.\]
The example we have in mind is the space of stability conditions $\Stab(\cD)$ on a CY$_3$ triangulated category $\cD$, in which case the central charges  of the objects of $\cD$ give flat co-ordinates for the connection $\nabla$,  and the skew-symmetric form $\eta(-,-)$ is induced by the Euler form on the Grothedieck group $K_0(\cD)$.

Let us choose a  local system of flat co-ordinates $(z_1, \cdots, z_n)$ on  $M$.  We  obtain induced linear co-ordinates $(\theta_1,\cdots,\theta_n)$ on each tangent space $\cT_{M,p}$ by writing a tangent vector $X$ at a point $p\in M$ in the form \[X=\sum_{i=1}^n \theta_i\cdot  \frac{\partial}{\partial z_i}.\] We can also introduce the constant skew-symmetric matrix $\eta_{ij}=\eta(dz_i,dz_j)$. 

A Joyce structure on  $M$ involves a pencil of  flat, meromorphic Ehresmann connections $\cA^{(\hbar)}$ on  the bundle $\cT_M$, whose horizontal subspaces are spanned by  vector fields of the form
\begin{equation}
\label{sd2}\frac{\partial}{\partial z_i} +\frac{1}{\hbar} \cdot \frac{\partial}{\partial \theta_i}+\sum_{j} \eta_{jk} \cdot \frac{\partial^2 J}{\partial \theta_i 
\partial \theta_j} \cdot  \frac{\partial}{\partial \theta_k},\end{equation}
for some meromorphic function  $J\colon \cT_M\to \bC$  on the total space of the tangent bundle, which we call the Joyce function.\footnote{In fact it is the second derivatives of the Joyce function appearing in \eqref{sd2} that are genuine  meromorphic functions on $\cT_M$. In general, the  function $J$ itself could acquire logarithmic singularities, and  is therefore  only well-defined and single-valued on a dense open subset. See the discussion in Section \ref{heur} below.}
Flatness of the connections \eqref{sd2} is implied by the partial differential equations
\[
\frac{\partial^2 J}{\partial \theta_i \partial z_j}-\frac{\partial^2 J}{\partial \theta_j \partial z_i }=\sum_{p,q} \eta_{pq} \cdot \frac{\partial^2 J}{\partial \theta_i \partial \theta_p} \cdot \frac{\partial^2 J}{\partial \theta_j \partial \theta_q}.\]

The pencil of flat connections \eqref{sd2} should be viewed as analogous to the deformed flat connection of a Frobenius manifold.  Just as in the Frobenius case, this pencil extends to a meromorphic connection on the pull-back of the bundle $\cT_M$ to the product $M\times \bP^1$, which we can then alternatively view as an isomonodromic family of meromorphic  connections on the trivial bundles over $\bP^1$ with fibres $\cT_{M,x}$. In the Joyce case these connections take the form
\begin{equation}
\label{point2} \frac{\partial}{\partial \hbar} -\frac{1}{\hbar^2} \cdot \sum_i z_i \cdot \frac{\partial }{\partial \theta_i}-\frac{1}{\hbar}\cdot  \sum_{i,j,k} z_i\cdot  \eta_{jk} \cdot\frac{\partial^2 J}{\partial \theta_i  \partial \theta_j} \cdot \frac{\partial}{\partial \theta_k}.\end{equation}
It is the Stokes data of the Ehresmann   connections \eqref{point2} at the irregular singularity $\hbar=0$ that is encoded by the associated BPS structures. The reconstruction problem discussed above amounts to finding the Joyce function $J=J(z_i,\theta_j)$ as a function  of this Stokes data.

\subsection{Associated linear data}

 Consider again a Joyce structure on a complex manifold $M$, and let us assume that the Joyce function $J$ is holomorphic in a neighbourhood of the zero-section of $\cT_M$, corresponding to the locus where all co-ordinates $\theta_i=0$.
Although the  connections $\cA^{(\hbar)}$ of the pencil \eqref{sd2} are inherently non-linear, we can nonetheless use such a Joyce structure  to induce linear structures on the tangent bundle $\cT_M$. 

We begin by defining a flat, torsion-free  connection on  $\cT_M$ by the formula
\[
\nabla^J_{\frac{\partial}{\partial z_i}}\Big(\frac{\partial}{\partial z_j}\Big)=  -\sum_{p,q}\eta_{pq}\cdot\frac{\partial^3 J}{\partial \theta_i \, \partial \theta_j \, \partial \theta_p}\Big|_{\theta=0} \cdot \frac{\partial}{\partial z_q}.\]
We call it the linear Joyce connection, since an equivalent definition appears in \cite[Section 6.2]{HolGen}.
We can also define
a symmetric bilinear form $g\colon \cT_M\times\cT_M\to \O_M$
\[g\Big(\frac{\partial}{\partial z_i},\frac{\partial}{\partial z_j}\Big)=  \sum_m z_m \cdot \frac{\partial^3 J}{\partial \theta_i \partial \theta_j \partial \theta_m }\Big|_{\theta=0},\]
which we call   the Joyce form.
 This form  is covariantly constant for the connection $\nabla^J$, but need not be non-degenerate in general. When  it is,  we can also  define a commutative operation $\diamond\colon \cT_M\times\cT_M\to \cT_M$,  which we call the diamond product, by the formula
\[g\bigg(\frac{\partial}{\partial z_i} \diamond \frac{\partial}{\partial z_j}, \frac{\partial}{\partial z_k}\bigg)=\frac{\partial^3 J}{\partial \theta_i \partial \theta_j \partial \theta_k }\Big|_{\theta=0}= g\bigg(\frac{\partial}{\partial z_i},\frac{\partial}{\partial z_j}\diamond\frac{\partial}{\partial z_k}\bigg).\]
This operation is not always associative, but as we shall see, for at least some of the examples arising naturally in DT theory  it is.
 
 There is one other interesting object associated to a Joyce structure, which can be defined only when both sides of the equation \eqref{point2}  identically vanish. This happens for example when   the  form $\eta=0$. In this situation    there exist locally-defined functions $\cF\colon M\to \bC$ satisfying
 \[\frac{\partial^3 \cF}{\partial z_i\, \partial z_j \, \partial z_k }=\frac{\partial^3 J}{\partial \theta_i\, \partial \theta_j \, \partial \theta_k}\Big|_{\theta=0}.\]
 We call such a function a prepotential; it is unique up to the addition of quadratic functions in the co-ordinates $z_i$.
 
One of the main aims of this paper is to compute  the above linear data in some interesting  examples. In the next three subsections  we will give a brief summary of our results in this direction.

\subsection{Finite uncoupled BPS structures}

A BPS structure $(\Gamma,Z,\Omega)$ is called  finite if there are only finitely many classes $\gamma\in \Gamma$ for which the BPS invariant $\Omega(\gamma)$ is nonzero, and integral if all  invariants $\Omega(\gamma)\in \bZ$ are integers. It is called uncoupled if 
\[\gamma_1,\gamma_2\in \Gamma, \quad \Omega(\gamma_i)\neq 0 \implies \<\gamma_1,\gamma_2\>=0.\]
 This last condition implies that the BPS automorphisms associated to different rays commute, and it follows that in a variation of such BPS structures   the invariants $\Omega(\gamma)$ are constant. 

The RH problems defined by BPS structures satisfying the above three conditions were studied by Barbieri \cite{Barb}, who gave an explicit solution involving products of gamma functions. We review her results in detail in Section \ref{uncoupled} below. Suppose now given a variation of such BPS structures over a complex manifold $M$, which is  also miniversal, in the sense that the locally-defined period map
\[\varpi\colon M\to \Hom_{\bZ}(\Gamma,\bC),\]
sending a point $p\in M$ to the corresponding central charge $Z\colon \Gamma\to \bC$, is a local isomorphism. The derivative of the  map  $\varpi$ then allows us to  identify tangent vectors to $M$  with group homomorphisms $\theta\colon \Gamma\to \bC$, and so the Joyce function  can  be viewed as depending  on the central charge $Z\colon 
\Gamma\to \bC$  together with a tangent vector $\theta\colon \Gamma\to \bC$.

In Section \ref{uncoupled} we  show that the  Joyce function is given explicitly by the simple formula\[J(Z,\theta)=\frac{1}{24\pi i }\cdot\sum_{\gamma\in \Gamma\setminus\{0\}}  \Omega(\gamma)\cdot \frac{\theta(\gamma)^3}{Z(\gamma)}.\]
The associated linear data is then easily computed. The Joyce form is
\[g(X_1,X_2)=\frac{1}{4\pi i}\cdot \sum_{\gamma\in \Gamma\setminus\{0\}} \Omega(\gamma)\cdot X_1(\gamma)X_2(\gamma).\]
When this is non-degenerate,  the diamond product  is defined by
\[g(X_1\diamond X_2,X_3)=\frac{1}{4\pi i}\cdot  \sum_{\gamma\in \Gamma\setminus\{0\}}  \Omega(\gamma)\cdot \frac{X_1(\gamma)X_2(\gamma)X_3(\gamma)}{Z(\gamma)}= g(X_1,X_2\diamond X_3).\]
Finally, there is a  locally-defined prepotential $\cF\colon M\to \bC$   given by the formula
\[\cF(Z)=\frac{1}{8\pi i }\cdot \sum_{\gamma\in \Gamma\setminus\{0\}} \Omega(\gamma) \cdot Z(\gamma)^2 \log Z(\gamma).\]

The condition that the diamond product $\diamond$ is associative is non-trivial, and in fact implies that the prepotential $\cF$  satisfies the WDVV equations  with respect to the metric $g$. In the special case when all invariants $\Omega(\gamma)\in \{0,1\}$  this is equivalent to the condition that the set of classes $\gamma\in \Gamma$ for which $\Omega(\gamma)=1$ form a $\vee$-system in the sense of Veselov \cite{veselov}. 

\subsection{Calabi-Yau threefolds without compact divisors}
In Section \ref{snine} we consider the BPS structures which result from DT theory applied to compactly-supported coherent sheaves on a Calabi-Yau threefold $X$ containing no compact divisors.  Since all such sheaves $E\in \Coh(X)$ are then supported on curves, the Chern character $(\ch_2(E),\ch_3(E))$ is an element of the lattice
\[\Gamma=H_2(X,\bZ)\oplus \bZ.\]
The Euler form $\<-,-\>$ on this lattice vanishes, since curves do not generically intersect on a threefold, so 
the resulting BPS structures are uncoupled, although not  finite. 

We assume the well known conjecture that the BPS curve-counting invariants coincide with the genus 0 Gopakumar-Vafa invariants:
\[\Omega(\beta,n)=\GV(0,\beta).\]
We also assume that there are only finitely many homology classes $\beta\in H_2(X,\bZ)$ for which these invariants are nonzero. We can then solve the corresponding RH problem using the methods of \cite{Con}.

The resulting Joyce function  depends on central charge co-ordinates $(v,w)\in H^2(X,\bC)\oplus\bC$, and fibre co-ordinates $(\theta,\phi)\in H^2(X,\bC)\oplus\bC$. It is given by
\[J(v,w,\theta,\phi)=\frac{1}{6w^4} \cdot \sum_{\beta\in H_2(X,\bZ)}\GV(0,\beta) \cdot \left(v(\beta) \phi-w\theta(\beta)\right )^3\cdot \big(1-e^{-2\pi i v(\beta)/w}\big)^{-1}.\]
In this case the Joyce form vanishes, which is perhaps to be expected, since under our assumptions the intersection form on $H^*(X,\bC)$ also vanishes.  The diamond product is therefore undefined. On the other hand, computing the prepotential gives
\[F(v,w)=-\frac{w^2}{(2\pi i)^3}\cdot \sum_{\beta\in H_2(X,\bZ)} GV(0,\beta)\cdot \Li_3\big(e^{2\pi i v(\beta)/w}\big).\]
Up to a trivial factor this coincides with the genus 0  Gromov-Witten generating function.

These results should be compared with the main result of \cite{Con} which, in the case when $X$ is the resolved conifold,  used the solution of the RH problem at $(\theta,\phi)=(0,0)$ to define an analytic function $\tau(v,w,\hbar)$ whose asymptotic expansion  at $\hbar=0$ reproduced the  genus $\geq 2$  terms in the Gromov-Witten generating function of $X$. Taken together, we view these results as strong evidence that the RH problems we are studying are relevant for a global understanding of topological string theory.

\subsection{The A$_2$ quiver and Painlev{\'e} I}
 In Section \ref{sten} we focus on the variation of BPS structures defined by the DT theory of the A$_2$ quiver. This is the simplest example of a variation which is not uncoupled, and hence where the wall-crossing formula is non-trivial.  The base of the structure is the complex manifold
\[M=\big\{(a,b)\in \bC^2: 4a^3+27b^2\neq 0\big\}.\]
It follows from the results of \cite{BQS} that $M$ can be viewed as the space of stability conditions on the bounded derived category of the three-dimensional Ginzburg algebra associated to the A$_2$ quiver, quotiented by the subgroup of the group of auto-equivalences generated by spherical twists in the two vertex simples.

Each point of $M$  determines a meromorphic  quadratic differential on $\bP^1$
\begin{equation*}
Q_0(x) \, dx^{\tensor 2} = (x^3+ax+b) \, dx^{\tensor 2},\end{equation*}
with a single pole of order seven at $x=\infty$. The central charge co-ordinates are multi-valued on the quotient space $M$, and are given by the period integrals
\[z_i=\int_{\gamma_i} \sqrt{Q_0(x)} \, dx,\]
where $(\gamma_1,\gamma_2)$ is a basis for the first homology group of the affine elliptic curve $y^2=Q_0(x)$. 

Following the lead  of Gaiotto, Moore and Neitzke  \cite{GMN2}, it is proved in the paper \cite{A2}  that the asssociated  RH problems can be solved using the monodromy map for a family of differential equations of the form
\begin{equation}\label{de5} y''(x)=Q(x,\hbar) \cdot y(x), \qquad Q(x,\hbar)=\hbar^{-2}\cdot {Q_0(x)}+\hbar^{-1}\cdot Q_1(x) +  Q_2(x).\end{equation}
The particular equations appearing   are known as deformed cubic oscillators, and have been studied for many years in relation to the first Painlev{\'e} equation, which describes their isomonodromic deformations. It is explained in \cite{A2} that an extended version of the isomonodromy connection gives the pencil of non-linear connections  of the associated  Joyce structure.

Writing down the isomonodromy connection explicitly give rise to a formula for the Joyce function  $J\colon \cT_M\to \bC$ which we write out explicitly in  Theorem \ref{one}. The 
Joyce form can be calculated from this  and turns out to be 
\[g=\frac{2\pi i }{5}\cdot (da\tensor db+db\tensor da).\]
In particular, $g$ is non-degenerate, and we recover the natural co-ordinates $(a,b)$, which are not at all obvious from the point of view of the space of stability conditions, as the  flat co-ordinates for the linear Joyce connection. Furthermore,  the diamond product   is associative, and coincides, up to a constant factor,  with the product of the dual almost Frobenius structure \cite[Section 5.2]{D3} 
associated to the  A$_2$ root system. 

This A$_2$ example fits into a much more general class of examples,  which in physical terms correspond to theories of class $S$ with gauge group $SU(2)$. A result of the author  with Smith \cite{BS} shows that in these cases the space of stability conditions  modulo autoequivalences coincides with the moduli space of pairs $(S,\phi)$, consisting of a Riemann surface $S$ equipped with a meromorphic quadratic differential $\phi$ having poles of fixed orders and simple zeroes. The BPS invariants are obtained by counting finite-length trajectories of  the differential $\phi$. We expect that  in this more general setting the relevant Joyce structures can also be described using the isomonodromy connection for a  family of opers  of the form \eqref{de5}.

 \subsection{Relation to previous work}

As well as the theory of generalised DT invariants constructed by Joyce, Kontsevich, Soibelman and others, this paper builds in an essential way on several other works. One of the key underlying ideas is that the wall-crossing formula for DT invariants should be interpreted as  the  isomonodromy condition for a family of irregular connections. This  point-of-view was explained in a slightly different context by the author and Toledano Laredo  \cite{VTL}, and also played a prominent role in the work of Gaiotto, Moore and Neitzke \cite{GMN1}. The main results  of \cite{VTL} were themselves inspired by Joyce's remarkable paper  \cite{HolGen}, which attempts to use weighted sums of DT invariants to define holomorphic generating functions on the space of stability conditions.  
 
The material we present here is   also closely related to the work of Barbieri, Filipini, Garcia-Fernandez and Stoppa \cite{BSt,FGS}. These authors use DT invariants to construct formal Frobenius-type structures on the space of stability conditions, involving connections taking values in an infinite-dimensional Lie algebra of functions on an algebraic torus.  Their approach is close in spirit to that of Joyce \cite{HolGen}, and the resulting structures can be viewed as  formal versions of the same Joyce structures we consider here. 

In the approach of Joyce, and of Stoppa and his collaborators, and also in the approach we advocate  here, the non-trivial (and for the most part unsolved) problem is to invert a certain irregular RH map. For Joyce and Stoppa {\it et al} this is done  using a formal power-series expansion; the difficult issue is then to prove convergence of this series. In this paper, we instead take an analytic approach using RH problems, and the challenge is then to prove the existence of solutions to these problems. 

The RH problems which allow us to pass from BPS structures to Joyce structures  are closely related to  those considered by Gaiotto, Moore and Neitzke \cite{GMN1,GMN2} in relation to quantum field theories with $d=4$, $\cN=2$ supersymmetry. These authors already remarked on the close analogy between their constructions    and the $tt^*$ equations studied by Cecotti and Vafa \cite{CV1,CV2} in the context of theories with $d=2$, $\cN=2$  supersymmetry. In the conformal limit  \cite{G}, this  analogy becomes the analogy between Joyce structures and Frobenius structures which is the main topic of this paper. 

The RH problems considered here are also closely related to the Themodynamic Bethe Ansatz (TBA) equations, although the author is unfortunately not qualified to comment in detail on this. The basic point is that our RH problems can be reformulated as  integral equations, which after suitable reprocessing becomes the TBA equations. We refer the reader to   \cite{G}, \cite[Appendix E]{GMN1} and \cite{IMS} for more details on this connection.

\subsection{Plan of the paper}
The first two sections contain an exposition of the relevant parts of the theory of Frobenius manifolds. Everything here can be found in some form in Dubrovin's original lecture notes \cite{D1,D2}. In Section 2 we give the basic definitions, and explain how a Frobenius structure on a complex manifold can be encoded in an isomonodromic family of irregular connections. In Section 3 we then explain how this family of connections can in  turn be encoded by its Stokes data.

In the next two sections we give an analogous treatment of the theory of Joyce structures. The definition of a Joyce structure is introduced in Section 4, guided closely by the analogy with Frobenius structures. We also introduce the associated isomonodromic family of meromorphic Ehresmann connections. In Section 5 we recall from \cite{RHDT} the basic definitions  concerning BPS  structures, which encode the Stokes data of these connections.

In Section 6 we discuss the general problem of reconstructing pencils of connections from their Stokes data. This amounts to inverting an irregular RH correspondence. We consider several approaches to this problem, both in the finite-dimensional Frobenius case, and the infinite-dimensional Joyce case.

Section 7 is concerned with the linear structures on the tangent bundle induced by a Joyce structure. We define the linear Joyce connection, the Joyce form,  the diamond product and the prepotential discussed above. We also define a notion of compatibility for Frobenius and Joyce structures living on the same underlying manifold.

The bulk of the second half of the paper, comprising Sections 8--10, is devoted to explicit computations in   particular examples.  Section 8 deals with finite uncoupled BPS structures, following work of Barbieri \cite{Barb}. Section 9 is concerned with the DT theory of coherent sheaves on non-compact Calabi-Yau threefolds, and relies on the computations of \cite{Con}. Section 10 focuses on the case of the A$_2$ quiver, and is essentially a digest of the paper  \cite{A2}. 

\subsection*{Acknowledgements}
The author gratefully acknowledges useful conversations with  Dylan Allegretti, Anna Barbieri, Pierrick Bousseau, Kohei Iwaki, Dima Korotkin, Davide Masoero, Sven Meinhardt, Andy Neitzke, Nicolas Orantin, Ivan Smith, J{\"o}rg Teschner  and Valerio Toledano Laredo.


\section{Frobenius structures and pencils of connections}
\label{stwo}

In this section and the next we recall some of the basic theory of Frobenius manifolds, following Dubrovin \cite{D1,D2}. Other useful references are \cite{Hertling,Hitchin,Man,Sabbah}. This material will only be used in what follows as a motivating analogy for the definition of Joyce structures in Section \ref{joyce}, and we tailor our treatment to this purpose.  The main point for us is that a Frobenius structure can be encoded in  an isomonodromic family of meromorphic connections on $\bP^1$ of a particular kind.

\subsection{Definition of a Frobenius structure}
We begin with the definition of a Frobenius structure. The  holomorphic tangent bundle of a complex manifold $M$ will be denoted by $\cT_M$.

\begin{defn}\label{frobenius}A \emph{Frobenius structure} on a complex manifold $M$ consists of the following data
\begin{itemize}
\item[(i)]  the \emph{metric}: a  holomorphic map
$g\colon \cT_M \tensor \cT_M \to \O_M,$
inducing a symmetric and non-degenerate  bilinear form on each tangent space $\cT_{M,p}$;\smallskip

\item[(ii)]  the \emph{product}: a holomorphic map
$*\colon \cT_M \tensor \cT_M \to \cT_M,$
inducing on each tangent space $\cT_{M,p}$  the structure of a commutative and  associative  algebra;\smallskip

\item[(iii)] the   \emph{Euler vector field}: a holomorphic section
$E\colon \O_M\to \cT_M$;
\end{itemize}

subject to the following four axioms:\begin{itemize}
\item[(F1)] the Levi-Civita connection $\nabla^{\LC}$ on $\cT_M$ defined by the metric $g$ is flat;%

\item[(F2)] there exists a covariantly constant vector field $e\colon \O_M\to \cT_M$ which is  a  unit for the product;

\item[(F3)] in local flat co-ordinates $(t_1,\cdots, t_n)$ for the metric $g$,  there exists a locally-defined function $F(t_1,\cdots,t_n)$ satisfying
\begin{equation}\label{compcom} g\bigg (\frac{\partial}{\partial t_i}*\frac{\partial}{\partial t_j},\frac{\partial}{\partial t_k}\bigg )= \frac{\partial^3 F}{\partial t_i \partial t_j \partial t_k}= g\bigg (\frac{\partial}{\partial t_i},\frac{\partial}{\partial t_j}*\frac{\partial}{\partial t_k}\bigg );\end{equation}

\item[(F4)] the vector field $E$ is linear: $\nabla^{\LC}(\nabla^{\LC}(E))=0$, and satisfies\[\mathcal{L}ie_E(g) = (2-d)\cdot g, \qquad \mathcal{L}ie_E (*) =  *,\]
where   $\mathcal{L}ie_E$ denotes the Lie derivative, and $d\in \bC$ is some fixed complex number. \end{itemize}
\end{defn}

By a Frobenius manifold we of course mean a manifold equipped with a Frobenius structure. We commit the usual abuse of labelling a Frobenius manifold by its underlying manifold leaving the metric, product and Euler vector field implicit. The vector field $e$ of (F2) is clearly unique: it is called the \emph{identity vector field}. The number $d\in \bC$ appearing in (F4) is called the \emph{conformal dimension} or \emph{charge} of the Frobenius manifold. The complex dimension of the  underlying manifold $M$ will usually be denoted $n$.  
The function $F$ in condition (F3) is called a  \emph{prepotential}: on any given patch it is well-defined  up to the addition  of quadratic polynomials in the flat co-ordinates $t_i$.
The statement that the product $*$ is associative implies that $F$  satisfies the  WDVV equations.

\subsection{Tame points and canonical co-ordinates}
Let $M$ be a Frobenius manifold. The operation  of multiplication by the Euler vector field defines an endomorphism $U(X)=E*X$ of the tangent bundle. A point $p\in M$ is said to be \emph{tame} if $U$ has one-dimensional eigenspaces on the fibre $\cT_{M,p}$.  The subset of tame points is an open submanifold of $M$, and  we say that $M$ itself is tame if this open subset is the whole of $M$.   

\begin{lemma}
\label{lem2}
If a point $p\in M$ is tame then the  eigenvalues $(u_1,\cdots,u_n)$ of the operator $U$   form a 
system of co-ordinates in a neighbourhood of $p\in M$.  These satisfy
\[\frac{\partial}{\partial u_i}*\frac{\partial}{\partial u_j}=\delta_{ij} \, \frac{\partial}{\partial u_i}.\]
\end{lemma}

\begin{proof}
See \cite[Theorem 3.1]{D2}.
\end{proof}

Let $p\in M$ be a tame point. It follows from Lemma \ref{lem2} that the algebra $(\cT_{M,p}, *)$ is semisimple, and that there are
expressions
\begin{equation}
\label{cold}e=\sum_i \frac{\partial}{\partial u_i},\qquad E=\sum_i u_i \frac{\partial}{\partial u_i}.\end{equation}
The co-ordinates $(u_1,\cdots,u_n)$ of Lemma \ref{lem2} are 
called  \emph{canonical co-ordinates}. The compatibility \eqref{compcom} between metric and product ensures that the  tangent vectors  $\partial/\partial u_i$ define an  orthogonal basis of $\cT_{M,p}$. We denote by
$(f_1,\cdots, f_n)$ a corresponding orthonormal basis
\begin{equation}
\label{bloob}f_i = \frac{1}{d_i} \frac{\partial}{\partial u_i}, \qquad d_i^2 =g\Big(\frac{\partial}{\partial u_i}, \frac{\partial}{\partial u_i}\Big).\end{equation}
This basis is uniquely well-defined up to permutations and multiplication by $\pm1$.

\subsection{Lie algebra formalities}
\label{form2}

Let $M$ be a Frobenius manifold with a tame point $p\in M$. In this section we make some formal   remarks concerning  the Lie algebra $\gl(\cT_{M,p})$, which will be important for the analogy with Joyce structures explained in Section \ref{joyce} below. 

Sending a tangent vector $X\in \cT_{M,p}$ to the operation of multiplication by $X$ defines a linear map $m\colon \cT_{M,p} \to \gl (\cT_{M,p})$.
Since the algebra $\cT_{M,p}$ is unital, the map $m$ is injective, and since it   is commutative, and the elements $m(X)$ are semi-simple, the image  coincides with a  Cartan subalgebra $\fh\subset \gl (\cT_{M,p})$. Thus we have a diagram  \begin{equation}
\label{m}m\colon \cT_{M,p} \stackrel{\isom}{\lra}  \fh\subset \gl (\cT_{M,p}).\end{equation}
The Cartan subalgebra  $\fh\subset \gl (\cT_{M,p})$  determines a  canonical root decomposition
\begin{equation}\label{roots}\gl (\cT_{M,p})=\fh\oplus \bigoplus_{\alpha\in \Phi} \fg_\alpha.\end{equation}
Given a vector $X\in \cT_{M,p}$ we shall write $X(\alpha)=\langle m(X),\alpha\rangle\in \bC$
for the pairing of the vector $m(X)\in \fh$ with a root $\alpha\in \fh^*$. The condition that the point $p\in M$ is tame implies that   $U=m_E\in \fh$  lies in the complement $\fh^{\reg}$ of the root hyperplanes.

The metric on $\cT_{M,p}$ allows us to define an involution
$I(\Psi)=\Psi^{-*}$ of the group $\GL(\cT_{M,p})$
which sends an automorphism to its inverse adjoint. This induces an involution
\begin{equation}\label{ioio}
\iota\colon \gl(\cT_{M,p})\to \gl(\cT_{M,p}), \qquad \iota(\psi)=-\psi^{*}\end{equation}
at the level of Lie algebras, which maps an endomorphism to its negative adjoint. 
Note that the involution $\iota$ acts by $-1$ on the subalgebra $\fh\subset \fg$, and exchanges the root spaces $\fg_{\pm \alpha}$.   

In more down-to-earth terms, the basis $(f_1,\cdots,f_n)$ gives an identification
\begin{equation}
\label{hoho}\gl(\cT_{M,p})\isom \gl_n(\bC).\end{equation}
 Under this identification, the distinguished Cartan subalgebra $\fh\subset \fg$  corresponds to the diagonal matrices, and the root decomposition to the standard decomposition of $\gl_n(\bC)$ into multiples of the elementary matrices $E_{ij}$. The expression $X(\alpha)$ is just the difference between two eigenvalues of the operator of multiplication by $X$, and the subset $\fh^{\reg}\subset \fh$ is the set of diagonal matrices with distinct eigenvalues. Since the basis $(f_1,\cdots,f_n)$ is orthonormal, the involution $\iota$   corresponds under \eqref{hoho} to the usual negative transpose of matrices.

\subsection{First structure connection}

Let $M$ be a Frobenius manifold, and consider the projection $p\colon M\times \bP^1\to M$.
 Introduce the endomorphism of $\cT_M$
 \begin{equation}
 \label{defv}V(X)=\nabla^{\LC}_X(E)+\tfrac{1}{2}\cdot (d-2)\cdot   X.\end{equation}
The Frobenius structure on  $M$ can be encoded in a certain meromorphic connection $\cA$ on the bundle $p^*(\cT_M)$ known as the  \emph{first structure connection} or \emph{deformed flat connection}. It is   defined by the formulae
\begin{equation}\label{xmas}{\cA}_X(Y)=\nabla^{\LC}_X(Y)+\frac{1}{\hbar}\cdot X*Y,\qquad {\cA}_{\frac{\partial}{\partial \hbar}} (Y)=\frac{\partial Y}{\partial \hbar} -\bigg(\frac{U}{\hbar^2}+\frac{V}{\hbar}\bigg) Y,\end{equation}
where $\hbar$ is a co-ordinate on $\bP^1$, and we abuse notation by identifying a vector field on $M$ with its  lift to a section of $p^*(\cT_M)$ which is constant in the $\hbar$ direction.

\begin{remark}
\label{swim}
 It is easy to check that the operators $U$ and $V$ are self-adjoint and  skew-adjoint respectively with respect to the metric $g$.
Thus in terms of the involution $\iota$ of Section \ref{form2} we have
 $\iota(U)=-U$ and $ \iota(V)=V$.
  It follows that the connection on the frame bundle of  $p^*(\cT_M)$ induced by the first structure connection  is invariant under the composite operation of applying the involution $I$ and changing the sign of $\hbar$.
 \end{remark}

 In the literature one usually finds the formulae \eqref{xmas} written in terms of the co-ordinate $z=\hbar^{-1}$, but the above formulation will be more convenient for us.

\begin{lemma}
The first structure connection is flat.
\end{lemma}

\begin{proof}
See for example \cite[Proposition 2.1]{D2} or  \cite[Theorem I.1.5]{Man}.
\end{proof}

There are two complementary points-of-view on the  first structure connection which it is useful to keep in mind. 
On the one hand, if we forget the derivatives in the $\hbar$ direction, the first structure connection gives a pencil of 
flat, torsion-free connections \[\cA^{\hbar}_X(Y) = \nabla^{\LC}_X(Y)+\frac{1}{\hbar}\cdot X*Y,\]
on the tangent bundle $\cT_M$, parameterized by $\hbar\in \bC^*$.
On the other hand, for each point 
  $p\in M$, the connection $\cA$ induces a meromorphic connection $\cA_p$ on the trivial 
 vector bundle $\cT_{M,p} \tensor_{\bC} \O_{\bP^1}$ over the complex projective line $\bP^1$, given by the formula
\begin{equation}
\label{conn}\cA_{p}=d-\bigg(\frac{U}{\hbar^2}+\frac{V}{\hbar} \bigg)d\hbar.\end{equation}
From this second point-of-view, the manifold $M$ is parameterising a family of meromorphic connections on a trivial bundle over $\bP^1$. 

\subsection{Canonical connection}
\label{trd}

Consider a tame Frobenius manifold $M$. There is a unique  connection $\nabla$ on the tangent bundle $\cT_M$ in which the local orthonormal bases $(f_1,\cdots,f_n)$ defined above  are covariantly constant. We refer to $\nabla$  as the \emph{canonical connection}. In general $\nabla$ is flat, but not usually torsion-free. 
There is a bundle of Lie algebras $\gl(\cT_M)$ over $M$ whose fibre over a point $p\in M$ is the Lie algebra $\gl(\cT_{M,p})$. The canonical connection $\nabla$ induces a flat connection in this bundle, with respect to which the root decomposition \eqref{roots} and the involution \eqref{ioio}  are  covariantly constant. 

To relate the canonical connecrtion $\nabla$  to the Levi-Civita connection $\nabla^{\LC}$ we introduce the map of vector bundles
\begin{equation}\label{theta}\Theta\colon \cT_M\to \gl(\cT_M), \qquad \Theta_X= \sum_{\alpha\in \Phi} \frac{X(\alpha)}{U(\alpha)} \cdot V_\alpha,\end{equation}
where we used the  canonical decomposition \eqref{roots} at each point $p\in M$ to write \[V=\sum_{\alpha\in \Phi}V_\alpha,\qquad V_\alpha\in \fg_\alpha.\]

\begin{lemma}
There is an identity
\[\nabla_X^{\LC}
 =\nabla_X + \Theta_X.\]
\end{lemma}

\begin{proof}
This is a consequence of \cite[Lemma 3.2]{D2}.
\end{proof}

The flatness of the Levi-Civita connection then implies the following differential equation for the elements $V_{\alpha}$ as a function of the point $U\in \fh^{\reg}$:
\begin{equation}
\label{diffeq}d V_\alpha=\sum_{\beta+\gamma=\alpha}  [V_\beta,V_\gamma]\cdot d\log U(\gamma).\end{equation}
It follows from general results on isomonodromic deformations due to Malgrange \cite{Mal} (see also the discussion in \cite[Section II.2.1]{Man}), that any local solution to this equation extends to a meromorphic function on the universal cover  of the configuration space
\[\fh^{\reg}=\{U\in \fh: U(\alpha)\neq 0\text{ for all }\alpha\in \Phi\}.\]
The problem of reconstructing of a Frobenius manifold from such a solution is addressed in \cite[Proposition 3.5]{D1}, \cite[Lemma 3.3]{D2} and \cite[Section 4.3]{Hitchin}.

The fact that the function $V=V(U)$ has poles when extended to the universal cover of $\fh^{\reg}$  is related  to the existence of non-trivial vector bundles on $\bP^1$. In fact the connection \eqref{conn} on the trivial bundle over $\bP^1$ always extends uniquely to a meromorphic connection on a vector bundle $\cV$ over $\bP^1\times \widetilde{\fh^{\reg}}$. In general however, the restrictions of the bundle $\cV$ to the slices $\bP^1\times \{U\}$ will  not all be trivial, and this leads to poles in the analytic continuation of $V$. 

\subsection{First structure connection in canonical co-ordinates}
\label{firststructure}
Consider again a tame Frobenius manifold $M$.   In terms of the canonical co-ordinates $(u_1,\cdots, u_n)$ the  Frobenius structure can be completely encoded in the pair of  maps
\[m\colon \cT_M\to \gl(\cT_M),\qquad \Theta\colon \cT_M\to \gl(\cT_M),\]
defined in \eqref{m} and \eqref{theta}. Their images lie in the $+1$ and $-1$ eigenspaces of $\iota$ respectively. 
  Abusing notation as before, the first structure connection  takes the form
\[{\cA}_X = \nabla_X + \Theta_X+ \frac{1}{\hbar} \cdot m_X,\qquad {\cA}_{ \frac{\partial}{\partial \hbar}} = \frac{\partial}{\partial \hbar}-\frac{1}{\hbar}\cdot {\Theta_E} -\frac{1}{\hbar^2} \cdot m_E.\]
The fact that the first structure connection is flat is equivalent to the relations
\begin{equation}\label{nfirst}[m_X,m_Y]=0, \qquad [m_X,\Theta_Y]=[m_Y,\Theta_X], \end{equation}
\begin{equation}
\label{nthird}
m_{[X,Y]}=\nabla_X(m_Y)-\nabla_Y(m_X), \qquad \nabla_X(m_E)=m_X,
\end{equation}
\begin{equation}
\label{nsecond} \Theta_{[X,Y]}-[\Theta_X,\Theta_Y]=\nabla_X(\Theta_Y)-\nabla_Y(\Theta_X), \qquad \nabla_X(\Theta_E) = [\Theta_E,\Theta_X]. \end{equation}

It is possible to check these equations directly. The first equation of \eqref{nfirst} holds because the image of the map $m$ is contained in the bundle of Cartan subalgebras. The second equation  is then immediate from the definition  of $\Theta$. It implies that we can write 
$\Theta_X=[m_X,\Phi]$
and thus encode  $\Theta$ in the section
\[\Phi\colon M \to \gl_n(\cT_M), \qquad \Phi=\sum_{\alpha\in \Phi} \frac{V_\alpha}{U(\alpha)}. \]

The first equation of \eqref{nthird} is easily checked  by taking $X=\partial/\partial u_i$ and $Y=\partial/\partial u_j$ to be idempotents, and applying both sides of the equation to a $\nabla$ constant section $f_k$. The second equation follows similarly by taking $X=\partial/\partial u_i$ and using the expression \eqref{cold} for the Euler vector field. Finally, the equations \eqref{nsecond}
are equivalent to the differential equation \eqref{diffeq}: we leave the details of the proof of this last step to the reader.


\section{Stokes data and isomonodromy}
\label{sthree}

In the previous section we explained how a Frobenius structure can be encoded in a family of meromorphic connections \eqref{conn} on a trivial bundle over the complex projective line $\bP^1$. In this section we explain how  this family can in turn be encoded by its Stokes data.  As with the contents of the previous section, this material will only be used in what follows as a motivating analogy, and we hope the reader will therefore excuse some of the idiosyncrasies of our exposition. We refer the reader to Dubrovin's original notes \cite[Section 3]{D1} and \cite[Section 4]{D2}. Other useful references on this material are \cite{BJL, Boalch, VTL, Sabbah}.

\subsection{Canonical solutions}
\label{canic}

Let us begin by considering an abstract connection on the trivial $\GL_n(\bC)$ bundle over $\bP^1$ of the form
\begin{equation}
\label{connie}\nabla=d-\bigg(\frac{U}{\hbar^2}+\frac{V}{\hbar} \bigg)d\hbar,\end{equation}
where $U\in \fh^{\reg}\subset \fh\subset \gl_n(\bC)$ is a diagonal matrix with one-dimensional eigenspaces, and 
\[V\in \god=\bigoplus_{\alpha\in \Phi} \fg_\alpha\subset \gl_n(\bC)\]
is an arbitrary off-diagonal matrix. This connection has an irregular singularity at $\hbar=0$ and a regular singularity at $\hbar=\infty$.
 
 The \emph{Stokes rays}  of the equation \eqref{connie} at the irregular singularity $\hbar=0$ are defined to be the rays $\bR_{>0}\cdot U(\alpha)\subset \bC^*$, where, as explained in Section \ref{form2}, the points $U(\alpha)\in \bC^*$ are the differences $u_i-u_j$ of the eigenvalues of the matrix $U$.
Given an arbitrary ray in the complex plane of the form $r=\bR_{>0}\cdot z\subset \bC^*$, we denote by\[
\bH_r=\{z=uv:u\in r, \, \Re(v)>0\}\subset \bC^*
\]
the half-plane centered on it.

\begin{thm}
\label{bjl}
\begin{itemize}
\item[(i)] For any non-Stokes ray $r\subset \bC^*$, there is a  holomorphic map
\begin{equation}
\label{canskeg}\Psi_r\colon \bH_r\to \GL_n(\bC),\end{equation}
which is a flat section of the connection \eqref{connie}  and satisfies
\begin{equation}
\label{limit}\Psi_r(\hbar)\cdot \exp(U/\hbar)\to \id\end{equation}
as $\hbar \to 0$ in the half-plane $\bH_r$.
\smallskip

\item[(ii)]
Suppose that $\Delta\subset \bC^*$ is a closed, convex sector, bounded in clockwise order by non-Stokes rays $r_1,r_2$. Then for any flat sections $\Psi_{r_i}$ as in (i), there is an element
\[S(\Delta)\in \exp\Big(\bigoplus_{U(\alpha)\in \Delta} \fg_\alpha\Big)\subset \GL_n(\bC),\]
such that for all $\hbar\in \bH_{r_1}\cap \bH_{r_2}$ one has $\Psi_{r_2}(\hbar)=\Psi_{r_1}(\hbar)\cdot S(\Delta)$. \qed
\end{itemize}
\end{thm}

Part (i) can be found in  \cite{BJL}, although it goes back earlier, and in some form to Birkhoff. Part (ii) is completely elementary: any two flat sections of \eqref{conn} differ by a constant element $S(\Delta)\in \GL_n(\bC)$ and the asymptotics \eqref{limit} ensure that
\[\exp(-U/\hbar)\cdot S(\Delta)\cdot \exp(U/\hbar)\to \id\]
as $\hbar\to 0$ in the overlap $\bH_{r_1}\cap \bH_{r_2}$. This easily implies the given form of $S$.

\begin{remark}
\label{skeg}
Part (ii) of Theorem \ref{bjl} immediately gives the following two statements:
\begin{itemize}
\item[(i)] For each non-Stokes ray $r\subset \bC^*$, the flat section $\Psi_r\colon \bH_r\to \GL_n(\bC)$ of part (i) is unique with the property \eqref{limit}. We call it the \emph{canonical solution} to the connection \eqref{connie}  on the half-plane $\bH_r$.\smallskip

\item[(ii)] Suppose that two non-Stokes rays $r_1$ and $r_2$ are the boundary rays of a convex sector which contains no Stokes rays. Then the corresponding canonical solutions  $\Psi_{r_1}$ and $\Psi_{r_2}$ are analytic continuations of one another, in the sense that they glue to give a single analytic function on $\bH_{r_1}\cup \bH_{r_2}$.
\end{itemize}
\end{remark}

\subsection{Stokes data}
\label{stomap}

Let us again consider the meromorphic connection \eqref{connie}. To each Stokes ray
\[\bR_{>0} \cdot U(\alpha)=\bR_{>0}\cdot 
(u_i-u_j)\subset \bC\]
is associated a \emph{Stokes factor}  \begin{equation}
\label{skegn}\stokes(\ell)=\exp\bigg(\sum_{\alpha\in \Phi:U(\alpha)\in \ell}D(\alpha)\cdot  E_\alpha\bigg)\in \exp\bigg(\bigoplus_{\alpha\in \Phi:U(\alpha)\in \ell} \fg_\alpha\bigg)\subset \GL_n(\bC).\end{equation}
obtained by taking $\stokes(\ell)=\stokes(\Delta)$ as defined in Theorem \ref{bjl}(ii), where  $\Delta\subset \bC^*$ is a closed, convex sector which contains the ray $\ell$  in its interior and  no other Stokes rays.
Note that for a generic matrix $U\in \fh^{\reg}$  the Stokes rays $\bR_{>0} \cdot U(\alpha)$ are all distinct, and each of the subgroups appearing on the right of \eqref{skegn} is one-dimensional. But for special choices of $U$ several Stokes rays may line up, and the corresponding subgroup is then larger.

The sets $\{\alpha\in\Phi:U(\alpha)\in\ell\}$  partition the set of roots $\Phi\subset \fh^*$ as $\ell$
ranges over the Stokes rays of $\nabla$, so we may assemble the elements
$D(\alpha)\cdot E_\alpha\in \fg_\alpha$ corresponding to the different Stokes rays and form the sum
\begin{equation}\label{eq:g od}
R=\sum_{\alpha\in \Phi}D(\alpha)\cdot E_\alpha\in\god.
\end{equation}

\begin{theorem}
\label{boal}
For each element 
 $U\in\fh^{\reg}$ the corresponding map
 \[\cS\colon \god\to\god, \qquad \cS(V)=R,\]
 sending the equation \eqref{connie} to its Stokes data, is a local isomorphism of complex manifolds.
 \end{theorem}
 
 \begin{proof}
 The fact that the map $\cS$ is holomorphic follows from a   result of Sibuya which proves that the canonical solutions of Theorem \ref{bjl} vary holomorphically with parameters (see e.g. \cite[Proposition 3.2] {JMU}). It remains to show that the derivative of $\cS$ 
 is injective. This  follows from the description of iso-Stokes  deformations in Theorem \ref{emmy} below (see e.g. \cite[Theorem 3.3]{JMU}).
 \end{proof}
 
 The map $\cS$ of Theorem \ref{boal} is usually called the \emph{Stokes map}. Although it is a local isomorphism, it is not usually bijective. We shall return to this point in Section \ref{top} below.

\subsection{Iso-Stokes deformations}

Let us again return to the meromorphic connection \eqref{connie}, and consider now varying the matrices $U$ and $V$, always maintaining the  conditions $U\in \fh^{\reg}$ and $V\in\god$. Note that as the matrix $U\in \fh^{\reg}$ varies, the Stokes rays $\bR_{>0}\cdot U(\alpha)$ may collide and separate. The resulting family of connections is called \emph{iso-Stokes} if the following condition is satisfied: for any convex sector $\Delta\subset \bC^*$, the clockwise product of matrices
\begin{equation}
\label{footy}\stokes(\Delta)=\prod_{\ell\in \Delta} \stokes(\ell) \in \GL_n(\bC),\end{equation}
remains constant, unless and until the boundary rays of $\Delta$  become Stokes rays.

\begin{theorem}
\label{emmy}
Consider a family of  meromorphic connections of the form \eqref{conn} with $U\in \fh^{\reg}$ and $V$ varying continuously as a function of $U$. Then the family is iso-Stokes precisely if $V=V(U)$ satisfies the differential equation
\begin{equation}
\label{fz}d V_\alpha=\sum_{\beta+\gamma=\alpha}  [V_\beta,V_\gamma]\cdot d\log U(\gamma),\end{equation}
where we decompose $V=\sum_{\alpha\in \Phi} V_\alpha$ with $V_\alpha\in \fg_\alpha$. 
\end{theorem}

\begin{proof}
This appears in a much more general context in \cite[Theorem 3.3]{JMU}. The explicit form \eqref{fz} appears for example as \cite[Equation 3.107]{D1}.
\end{proof}

As discussed in Section \ref{trd},  any local solution to the  equation \eqref{fz} extends uniquely to a meromorphic function $V=V(U)$ on the universal cover of $\fh^{\reg}$. On the Stokes side, the analogue of this statement is that given a collection of numbers $D(\alpha)\in \bC$ as in \eqref{eq:g od},  we can uniquely extend the $D(\alpha)$ to (non-continuous) functions of $U$ on the universal cover of $\fh^{\reg}$, so that the resulting Stokes factors \eqref{skegn} satisfy the  iso-Stokes property.
 
 \begin{remark}
 \label{shorts}
When the iso-Stokes property of Theorem \ref{emmy} holds, it follows from \cite[Theorem 3.1]{JMU}  that   the canonical solutions $\Psi_r\colon \bH_r\to \GL_n(\bC)$ of Theorem \ref{bjl} for different $U\in \fh^{\reg}$ assemble to form holomorphic functions $\Psi_r(U,\hbar)$, which for fixed $\hbar\in \bC^*$  satisfy a partial differential equation describing their variation with $U$. The argument of \cite[page 70]{D2} shows that this equation is
\[\frac{\partial}{\partial u_i} \Psi_r(U,\hbar)+\bigg(\sum _{\alpha\in \Phi} \frac{\alpha_i}{U(\alpha)}\cdot V_{\alpha} +\frac{1}{\hbar} \cdot E_{ii}\bigg)\Psi_r(U,\hbar)=0,\]
where $E_{ii}$ denotes the elementary matrix at position $(i,i)$.  Comparing with the formulae in Section \ref{firststructure} shows that in the Frobenius setting  the functions $\Psi_r(U,\hbar)$ are flat sections of the first structure connection.
\end{remark}
 
 \subsection{FS structures}
 \label{fs}
 The  analysis of the previous three subsections applies in particular to the connection \eqref{conn} defined by a tame point  of a Frobenius manifold. 
For the purposes of comparison with the notion of a BPS structure in Section \ref{next} we will now introduce the notion of a Frobenius-Stokes (FS) structure. This consists essentially of the data of the leading term $U\in \fh^{\reg}$ of the connection \eqref{conn}, together with the element $R\in \god$ defined by \eqref{eq:g od} which encodes the Stokes factors. 
 
 \begin{remark}In the connections \eqref{conn} arising from Frobenius structures, the operator $V$ is skew-adjoint. This results in a symmetry property of the  associated Stokes factors. Indeed, according to Remark \ref{swim}, given a canonical solution \eqref{canskeg} on the half-plane $\bH_r$, we can define a canonical solution on the opposite half-plane $\bH_{-r}$ by setting
\[\Psi_{-r}(\hbar)=\Psi_r(-\hbar)^{-T}.\] It follows that the Stokes factors satsify
$\stokes(-\ell)=\stokes(\ell)^{-T}$, or equivalently, in the notation of \eqref{skegn}, that $D(-\alpha)=D(\alpha)$ for all $\alpha\in \Phi$.
\end{remark}

 Given a finite-dimensional complex vector space $T$ equipped with an unordered basis $\{f_1,\cdots,f_n\}$, we will denote by $\fh\subset \gl(T)$ the  distinguished Cartan subalgebra  consisting of endomorphisms which are diagonal in the given basis, and $\fh^{\reg}\subset \fh$ the open subset of  endomorphisms with one-dimensional eigenspaces. The corresponding root system $\Phi\subset \fh^*$ consists of the elements $f_i^*-f_j^*$  for  $i\neq j$. 

 \begin{defn}
\label{sira}
An \emph{FS structure}  $(T,U,D)$ consists of \begin{itemize}
\item[(a)] a complex vector space $T$ equipped with an unordered basis $\{f_1,\cdots,f_n\}$;\smallskip

\item[(b)] an endomorphism $U\in \fh^{\reg}\subset \fh\subset \gl(T)$; \smallskip

\item[(c)] a map of sets $D \colon \Phi\to \bC$ satisfying 
$D(-\alpha)=D(\alpha)$ for all $\alpha\in \Phi$.
\end{itemize}
\end{defn}

The Stokes rays of an FS structure are those of the form $\bR_{>0}\cdot U(\alpha)$, and the associated Stokes factors are defined by
\begin{equation}
\label{skegn2}\stokes_\ell=\exp\Big(\sum_{\alpha\in \Phi:U(\alpha)\in \ell}D(\alpha)\cdot  E_\alpha\Big)\in \exp\bigg(\bigoplus_{\alpha\in \Phi:U(\alpha)\in \ell} \fg_\alpha\bigg)\subset \GL(T),\end{equation}
where for each root $\alpha=f_i^*-f_j^*$ we denote by $E_\alpha=E_{ij}$ the distinguished generator  of the root space $\fg_\alpha$ defined by the relation $E_{ij}(f_m)=\delta_{im} f_j$.

\begin{remark}
\label{em}
Given a Frobenius manifold $M$ and a tame  point $p\in M$,  there is an associated FS structure obtained by taking the vector space $T=\cT_{M,p}$, the operator $U$ of multiplication by the Euler vector field,   and the basis $\{f_1,\cdots,f_n\}$ consisting of the normalised idempotents \eqref{bloob}.  The elements $D(\alpha)$ are determined by the Stokes data of the  connection \eqref{conn}. Note however that the definition of the  basis elements $f_i$ requires a choice of signs, or in other words a choice of section of a $\{\pm 1\}^n$ torsor over $M$. We make this remark now, because a very similar issue will arise in Section \ref{qr} in the context of Joyce structures. 
\end{remark}

\subsection{Variations of FS structures}
\label{vfs}

The following definition is again introduced for the 
purposes of comparison with the material of Section \ref{next}.  It is an abstraction of the iso-Stokes condition explained above.

\begin{defn}
\label{embryonic}
A \emph{variation of FS structures} over a complex manifold $M$ is a collection  of FS structures $(T_p,U_p,D_p)$ indexed by the points $p\in M$ such that
\begin{itemize}
\item[(a)] the vector spaces  $T_p$ form the fibres of a local system over $M$, and the basis $\{f_1,\cdots,f_n\}$ is covariantly constant;\smallskip
\item[(b)] the endomorphisms $U_p\in \gl(T_p)$ vary holomorphically;\smallskip
\item[(c)] for any convex sector $\Delta\subset \bC^*$, the clockwise composition
\[\stokes_p(\Delta)=\prod_{\ell\in \Delta} \stokes_p(\ell) \in \GL(T_p)\]
is covariantly constant as the point $p\in M$ varies, providing that  the boundary rays of $\Delta$  are never Stokes rays.
\end{itemize}
\end{defn}

What we mean by (a) more precisely is that there is a holomorphic vector bundle $\pi\colon T\to M$ equipped with a flat holomorphic connection $\nabla$ whose fibres can be identified with the vector spaces $T_p$, and that the unordered bases $\{f_1,\cdots,f_n\}\subset T_p$  are induced by  locally-defined covariantly constant sections of $T$. Condition (b) is then the statement that the operators $U_p$ are induced by a holomorphic endomorphism $U\in \End_M(T)$.
By covariantly constant in (c) we of course mean with respect to the connection induced by $\nabla$. 

Replacing $M$ by a contractible open subset or cover we can trivialise the  local system of vector spaces $T_p$ and obtain a holomorphic 
period map
\[\pi\colon M\to \fh\subset \gl_n(\bC), \qquad p\mapsto U_p.\] We say that a variation of BPS structures is \emph{miniversal} if the derivative of this map is everywhere non-vanishing, so that $\pi$ is a local isomorphism.

\begin{remark}
Let $M$ be a tame Frobenius manifold. Then the FS structures of Remark \ref{em} fit together to form a variation of FS structures over  $M$. The local system is given by the canonical connection $\nabla$ on the tangent bundle $\cT_M$. The only thing to check is the iso-Stokes condition, but this follows from 
Theorem \ref{emmy} together with the differential equation \eqref{diffeq}. This variation is miniversal, because the eigenvalues of the operator $U$ are local co-ordinates on $M$.
\end{remark}


\section{Joyce structures}
\label{joyce}

In this section we introduce the notion of a Joyce structure. This is  the geometric structure we expect to find on the space of stability conditions of a CY$_3$ triangulated category. The definition is  motivated by a strong analogy with  Frobenius structures which is summarised in Table \ref{table}. The basic idea is to replace the finite-dimensional Lie algebra $\gl_n(\bC)$ of Section \ref{form2} with the infinite-dimensional Lie algebra $\Vect_{\eta}(\bC^*)^n$ of   vector fields on the algebraic torus $(\bC^*)^n$ whose flows preserve a translation-invariant Poisson structure $\eta$. The original impetus for this  is the form of the wall-crossing formula in DT theory, which is an iso-Stokes condition of precisely the same form as \eqref{footy}, but taking values in the group of Poisson automorphisms of the torus $(\bC^*)^n$.

\begin{table}
\begin{center}
\begin{tabular}{|c|c|}
\hline
 $\gl_n(\bC)$ & $\Vect_\eta(\bC^*)^n$ \\
 \hline
 Frobenius structure & Joyce structure  \\ 
 Canonical co-ordinates & Central charge co-ordinates \\
 Levi-Civita connection & Joyce connection \\
 Product & Translation \\
 Deformed flat connection & Deformed Joyce connection \\
  Inverse transpose & Conjugation with inverse \\
Euler vector field & Central charge \\
Rotation coefficients $V_{ij}$ & Fourier coefficients $H_\alpha$ \\
Stokes factors & BPS automorphisms \\
FS structure & BPS structure \\\hline
\end{tabular}\end{center}\bigskip\caption{The analogy between Frobenius and Joyce structures. \label{table}}
\end{table}

\subsection{Lie algebra of vector fields}
\label{la}

We begin by discussing the abstract properties of the Lie algebra  which will replace the Lie algebra $\gl_n(\bC)$ in the Frobenius manifold story.  
The material here should be compared with that of Section \ref{form2}.

Consider a lattice $\Gamma\isom \bZ^{\oplus n}$ equipped with a skew-symmetric integral form
\[\eta\colon \Gamma\times \Gamma\to \bZ.\]
 Introduce
the algebraic torus
\[\bT_+=\Hom_\bZ(\Gamma,\bC^*)\isom (\bC^*)^n\]
and its co-ordinate ring, which is also the group ring of the lattice $\Gamma$
\[\bC[\bT_+]=\bC[\Gamma]\isom \bC[y_1^{\pm 1}, \cdots, y_n^{\pm 1}].\]
The  character of $\bT_+$ corresponding to an element $\gamma\in \Gamma$  will be denoted $y_\gamma$. The skew-symmetric form $\eta$ induces a translation invariant Poisson structure $\{-,-\}$  on the torus $\bT_+$, which is given on characters by
\begin{equation}
\label{poisson}\{y_{\alpha}, y_{\beta}\}= \eta(\alpha,\beta)\cdot y_{\alpha}\cdot y_{\beta}.\end{equation}

Consider the Lie algebra of  non-constant algebraic functions on $\bT_+$
\begin{equation}\label{wa} \god=\bigoplus_{\alpha\in \Gamma\setminus\{0\}} \fg_\alpha = \bigoplus_{\alpha\in \Gamma\setminus\{0\}} \bC\cdot y_\alpha\end{equation}
 equipped with the Lie bracket induced by the Poisson bracket $\{-,-\}$. The abelian Lie algebra \[\fh= \Hom_\bZ(\Gamma,\bC),\]
of translation-invariant vector fields on $\bT_+$ acts by derivations on $\god$ in the obvious way, and we can form the corresponding semi-direct product. This is the vector space direct sum $\fg=\fh\oplus \god$ equipped with the bracket defined on generators by
\[ \Big[(h_1,y_{\alpha_1}), (h_2,y_{\alpha_2})\Big]=\big(0,h_1(\alpha_2) y_{\alpha_2}-h_2(\alpha_1)y_{\alpha_1}+\eta(\alpha_1,\alpha_2)\cdot  y_{\alpha_1+\alpha_2}\big). \]
When the form $\eta$ is non-degenerate, the subalgebra $\fh\subset \fg$ is a Cartan subalgebra, and the decomposition \eqref{wa}  can be viewed as a root decomposition, with the roots being precisely the nonzero elements of $\Gamma$. 

There is an obvious  homomorphism of Lie algebras
\begin{equation}
\label{rho}\rho\colon \fg\to \operatorname{vect}_{\eta}(\bT_+),\end{equation}
to the Lie algebra of algebraic vector fields on the torus $\bT_+$ whose flows preserve the Poisson structure $\{-,-\}$. It acts by the identity on $\fh$ and sends a function $f\in\god$ to the corresponding Hamiltonian vector field $\Ham_f$. In particular, if we choose a basis $(\gamma_1,\cdots,\gamma_n)\subset \Gamma$ and write $y_i=y_{\gamma_i}$, the homomorphism $\rho$ is defined on generators by
\[h\in \fh\mapsto \sum_i h(\gamma_i) \cdot y_i \frac{\partial}{\partial y_i}, \qquad y_\alpha\in \god\mapsto \sum_{j} \eta(\alpha,\gamma_j)\cdot y_\alpha\cdot y_j \frac{\partial}{\partial y_j}.\]
In general the map $\rho$ has a large kernel: it is an isomorphism precisely when the form $\eta$ is non-degenerate.
We do not want to make this assumption in general, since it fails in some interesting examples.
There is an involution  $\iota\colon \fg\to \fg$  induced by the inverse map of the torus $\bT_+$. 
It acts on the generators considered above via
\[\iota(h,y_\alpha)=(-h,y_{-\alpha}).\]
In particular, the abelian subalgebra $\fh\subset \fg$ is contained in the $-1$ eigenspace of $\iota$. 

\subsection{Structures on the tangent bundle}
\label{geom}

Consider a complex manifold $M$ equipped with a flat, torsion-free connection $\nabla$ on the tangent bundle $\cT_M$. Suppose also that there is a covariantly constant lattice
$\Gamma_M\subset \cT^*_{M}$,  and a covariantly constant skew-symmetric form \[\eta\colon \cT^*_M\times\cT^*_M\to \cO_M\]
which takes integral values on $\Gamma_M$. 
Given a  holomorphic vector field $X$ on $M$ there are then two lifts  of $X$ to vector fields on the total space of the bundle $\pi\colon \cT_M\to M$, both of which will be important in what follows. 

For the first lift of $X$, note that the vertical tangent vectors to the space $\cT_M$ at a point $x\in \cT_M$  are in natural correspondence with the elements of the vector space $\cT_{M,\pi(x)}$. Thus there is a vertical vector field $m_X$ on $\cT_M$, obtained by mapping a  point $x\in \cT_M$  to the vertical tangent vector corresponding to $X_{\pi(x)}\in \cT_{M,\pi(x)}$. Clearly the restriction of $m_X$ to each fibre $\pi^{-1}(p)=\cT_{M,p}$ is invariant under  translations. 

The second lift of the vector field $X$, which we denote by $\cH_X$, is induced by  $\nabla$ viewed   as an Ehresmann connection on the bundle $\pi\colon \cT_M\to M$. From this point-of-view $\nabla$ corresponds to a choice of sub-bundle $\cH\subset \cT_{\cT_M}$ of the tangent bundle to the total space of $\cT_M$, which is everywhere complementary to the sub-bundle $\cV\subset \cT_{\cT_M}$ of vertical tangent vectors for the map $\pi$.  Thus the derivative of $\pi$ then induces an isomorphism $\pi_*\colon \cH_x\to \cT_{M,\pi(x)}$ at each point $x\in \cT_M$. The lift $\cH_X$ is then defined to be the unique holomorphic vector field on $\cT_M$ which takes values in the horizontal sub-bundle $\cH\subset \cT_{\cT_M}$ and satisfies  $\pi_*(\cH_X)=X$. 

To write all this in co-ordinates,  take a basis $(\gamma_1,\cdots,\gamma_n)\subset \Gamma_{M,p}$ at some point $p\in M$, and extend to a covariantly constant basis of the bundle of lattices $\Gamma_M$ using the connection $\nabla$. In this basis the form $\eta$ is given by a constant skew-symmetric integral matrix
\[\eta(\gamma_i,\gamma_j)=\eta_{ij}\in \bZ.\]

Since the connection $\nabla$ is torsion-free,  there are local co-ordinates $(z_1,\cdots,z_n)$ on the manifold $M$ satisfying $\gamma_i=dz_i$.  We can then define co-ordinates $(\theta_1,\cdots,\theta_n)$ on the fibre $\cT_{M,p}$  over a point $p\in M$  by writing a vector $X\in \cT_{M,p}$ in the form\[X=\sum_{i=1}^n \theta_i \cdot \frac{\partial}{\partial z_i}.\]
In the resulting local co-ordinate system $(z_1,\cdots,z_n, \theta_1,\cdots,\theta_n)$ on the manifold $\cT_{M}$ we have
\[m_{\frac{\partial}{\partial z_i}}=\frac{\partial}{\partial \theta_i}, \qquad \cH_{\frac{\partial}{\partial z_i}}= \frac{\partial}{\partial z_i}.\]

For each point $p\in M$, the skew-symmetric form $\eta_p$ induces a translation-invariant Poisson structure on the fibre $\cT_{M,p}$. These combine to give a Poisson structure on the space $\cT_M$, which we call the \emph{vertical Poisson structure} induced by the form $\eta$. In co-ordinates it is given by
\[\{f,g\}=\sum_{i,j} \eta_{ij} \cdot \frac{\partial f}{\partial \theta_i}\cdot \frac{\partial g}{\partial \theta_j}.\]

At each point $p\in M$ we can apply the construction of the previous subsection to the lattice $\Gamma_p\subset \cT_{M,p}$ and  the skew-symmetric form $\eta_p$. The corresponding tori are
\begin{equation}\label{t}\bT_{M,p}=\Hom_{\bZ}(\Gamma_{M,p},\bC^*)=\cT_{M,p}/\Gamma_{M,p}^*.\end{equation} This results in a family of Lie algebras 
$\fg_{M,p}$ which fit together to form a bundle of Lie algebras  $\fg_M$  over $M$.   The abelian subalgebras $\fh_{M,p}\subset \fg_{M,p}$ form a sub-bundle $\fh_M\subset \fg_M$, and the map $m$ considered above  defines a bundle isomorphism
\begin{equation}
\label{seagulls}m\colon \cT_M\to \fh_M.\end{equation}
Since the connection $\nabla$  preserves the form $\eta$ and the bundle of lattices $\Gamma_M\subset \cT_M^*$, it induces a connection on the bundle of Lie algebras $\fg_M$, which we also denote by $\nabla$.

\subsection{Heuristics for the definition}
\label{heur}
In this subsection we explain the heuristics behind the definition of a Joyce structure which will be given in the next subsection. We continue with the notation of  the previous two subsections. The basic idea is to follow the lead of Section  \ref{firststructure} and look for maps
\[m\colon \cT_M\to \fg_M,\qquad \Theta\colon \cT_M\to \fg_M,\]
whose images lie in the $+1$ and $-1$ eigenspaces of $\iota$ respectively, and satisfying the same relations
\begin{equation}\label{nfirsts}[m_X,m_Y]=0, \qquad [m_X,\Theta_Y]=[m_Y,\Theta_X], \end{equation}
\begin{equation}
\label{nthirds}
m_{[X,Y]}=\nabla_X(m_Y)-\nabla_Y(m_X), \qquad \nabla_X(m_E)=m_X,
\end{equation}
\begin{equation}
\label{nseconds} \Theta_{[X,Y]}-[\Theta_X,\Theta_Y]=\nabla_X(\Theta_Y)-\nabla_Y(\Theta_X), \qquad \nabla_X(\Theta_E) = [\Theta_E,\Theta_X], \end{equation}
for some holomorphic vector field $E\colon \O_M\to \cT_M$.

The first equation of \eqref{nfirsts} shows that   the image of the map $m$ should be a bundle of  abelian subalgebras of $\fg_M$, and   there is then an obvious choice, namely the map  \eqref{seagulls}. The image of the map $\Theta$ necessarily lies in the bundle of subalgebras $\fg^{\od}_M\subset \fg_M$. Thus it assigns to each tangent vector $X\in \cT_{M,p}$ a non-constant algebraic function $\Theta_{X,p}$ on the  torus \eqref{t}. Pulling back via the quotient map $\cT_{M,p}\to \bT_{M,p}$, we can also view $\Theta_{X,p}$ as a holomorphic function on the vector space $\cT_{M,p}$. Thus if  $X$ is a holomorphic vector field on $M$ we can  view $\Theta_X$ as a holomorphic function on the total space $\cT_M$.
 
As in Section  \ref{firststructure}, the  second relation of \eqref{nfirsts} allows us to encode the map $\Theta$ more efficiently in the form $\Theta_X=m_X(J)$, where 
\[J\colon \cT_M\to \bC\]
is a holomorphic function. Note that if we interpret the map $\Theta$   literally as taking values in the bundle of Lie algebras $\fg_M$, then the restriction of the function $J$  to each fibre $\cT_{M,p}$ should descend to an algebraic function on the  torus \eqref{t}. 

The connection $\nabla$ on the bundle of Lie algebras $\fg_M$ induces a connection on the sub-bundle $\fh_M\subset \fg_M$, which under the isomorphism \eqref{seagulls}  reduces to  the original connection $\nabla$ on the bundle $\cT_M$. 
The first relation of \eqref{nthirds} is then the condition that the connection $\nabla$ is torsion-free, and the second  becomes the statement that $\nabla_X(E)=X$ for all vector fields $X$ on $M$.

The first relation of \eqref{nseconds}  implies a non-linear partial differential equation for the  function $J$, namely
\[\cH_X(m_Y(J))-\cH_Y(m_X(J))=m_{[X,Y]}(J)-\{m_X(J),m_Y(J)\}. \]
This is written out in co-ordinates in \eqref{fl} below. The second relation of \eqref{nseconds} follows from this if we also impose the homogeneity relation $\cH_E(J)=-J$. 

 In what follows, this heuristic picture will need to be modified in several ways.
Firstly we must  allow the function $J\colon \cT_M\to \bC$ to have poles, and   drop the condition that the restrictions of $J$  to the fibres $\cT_{M,p}$ are induced by  algebraic functions on the  tori \eqref{t}. Secondly, it is really the second derivatives of $J$ appearing in the connection \eqref{matrix} below that are the most relevant quantities: the local existence of the function $J$ is just a convenient way to encode the  symmetries of its higher derivatives. In particular, the function  $J$ should only be considered as being  well-defined up to  transformations of the form
\begin{equation}\label{indet}J(z_i,\theta_j)\mapsto J(z_i,\theta_j)+ L(z_i,\theta_j),\end{equation}
where $L(z_i,\theta_j)$ is  a polynomial in the $\theta_j$ co-ordinates of degree at most 1. Moreover, since poles in the  second derivatives  of $J$ can lead to multi-valued logarithmic factors, we should only insist that the function $J$ itself is  well-defined  on a dense open subset of $\cT_M$. 

\subsection{Definition of a Joyce structure}
\label{boring}

After the heuristic discussion of the last subsection we can now proceed to a precise  definition of a Joyce structure. In the next subsection we will give a description in terms of co-ordinates, which of course may be more helpful.

\begin{definition}
\label{joycestructure}
A \emph{Joyce structure} on a complex manifold $M$ consists of the following data:
\begin{itemize}
\item[(i)] a flat, torsion-free connection $\nabla$ on the  tangent bundle $\cT_M$, together with a covariantly constant full rank sublattice
$\Gamma_M\subset \cT^*_{M}$;\smallskip
\item[(ii)] a covariantly constant skew-symmetric form $\eta\colon \cT^*_M\times\cT^*_M\to \cO_M$  taking integer values on the sublattice $\Gamma_M\subset \cT_M^*$;\smallskip
\item[(iii)]  the \emph{Joyce function}: a holomorphic function $J\colon U\to \bC$ defined on a dense open subset $U\subset \cT_M$;\smallskip
\item[(iv)] the   \emph{Euler vector field}: a holomorphic section
$E\colon \O_M\to \cT_M$;
\end{itemize}
satisfying the following conditions:
\begin{itemize}
\item[(J1)] for any two vector fields $X_1,X_2$ on $M$, the function $m_{X_1} m_{X_2}(J)$ extends to a meromorphic function on $\cT_M$, whose restriction to any fibre $\cT_{M,p}$ is an odd function;
\smallskip
\item[(J2)] for any three vector fields $X_1,X_2,X_3$ on $M$, the function $m_{X_1} m_{X_2} m_{X_3} (J)$ extends to a meromorphic function on $\cT_M$ which is invariant under translations by elements of the  sublattice $(2\pi i)\cdot \Gamma_{M,p}^*\subset \cT_{M,p}$;
\smallskip
\item[(J3)] for any two vector fields $X,Y$ on $M$, there is a relation
\[\cH_X(m_Y(J))-\cH_Y(m_X(J))=m_{[X,Y]}(J)-\{m_X(J),m_Y(J)\};\]
\item[(J4)] the vector field $E$ satisfies the relations
\[\nabla(E)=\id, \qquad  \cH_E(J)=-J.\]
\end{itemize}
\end{definition}

Here, as in Section \ref{geom}, given a vector field $X$ on $M$, we denote by $m_X$ the corresponding translation-invariant vertical vector field on $\cT_M$, and by $\cH_X$ the $\nabla$-horizontal lift of $X$ to  $\cT_M$. The  bracket $\{-,-\}$  appearing in (J3)  is the vertical Poisson bracket  on $\cT_M$   induced  by the form $\eta$. By an odd function on a vector space in axiom (J1) we mean one that satisfies $f(-v)=-f(v)$.

\subsection{Co-ordinate description}

Suppose given a Joyce structure as in Definition \ref{joycestructure}, and take a co-ordinate system $(z_1,\cdots,z_n,\theta_1,\cdots,\theta_n)$ on the tangent bundle $\cT_M$ as in Section \ref{geom}. In particular the lattice $\Gamma_M\subset \cT_M^*$ is spanned by the forms $dz_i$.  Axiom (J4) shows that after applying a transformation $z_i\mapsto z_i+c_i$,  the Euler vector field  $E$ takes the form
\[E=\sum_i z_i \cdot \frac{\partial}{\partial z_i}.\]
The conditions on the Joyce function $J(z_1,\cdots, z_n,\theta_1,\cdots, \theta_n)$ are then: 
\begin{itemize}
\item[(i)]  $J$  has the homogeneity properties
\begin{equation}
\label{oddd}
J(z_1,\cdots,z_n,-\theta_1,\cdots,-\theta_n)=-J(z_1,\cdots, z_n,\theta_1,\cdots,\theta_n),\end{equation}\begin{equation}
\label{homo}
J(\lambda\cdot z_1,\cdots,\lambda\cdot z_n,\theta_1,\cdots,\theta_n)=\lambda^{-1}\cdot J(z_1,\cdots, z_n,\theta_1,\cdots,\theta_n);\end{equation}
\item[(ii)] the second partial derivatives  of $J$ with respect to the $\theta_j$ co-ordinates 
extend to single-valued meromorphic functions on $\cT_M$, and the third partial derivatives  are  invariant under $\theta_j\mapsto \theta_j+2\pi i$;\smallskip
\item[(iii)] $J$ satisfies the differential equation 
\begin{equation}
\label{fl}
\frac{\partial^2 J}{\partial \theta_i \partial z_j}-\frac{\partial^2 J}{\partial \theta_j \partial z_i}=\sum_{p,q} \eta_{pq} \cdot \frac{\partial^2 J}{\partial \theta_i \partial \theta_p} \cdot \frac{\partial^2 J}{\partial \theta_j \partial \theta_q}.
\end{equation}\end{itemize}

\begin{remark}
When the form $\eta$ is non-degenerate, it is natural to try to verify the  equation   \eqref{fl} by checking that the  Hamiltonian vector fields on $\cT_M$ defined by the two sides are equal. This gives rise to a situation where   \emph{a priori} we only know that the equality \eqref{fl}   holds up to the addition of  a function independent of the variables $\theta_i$. In fact, it turns out that  this function can always be eliminated by applying particular transformations of the form  \eqref{indet} to the function $J$. Indeed, if we  replace $J$ by the expression
\[J(z_i,\theta_j)-\sum_k \theta_k\cdot \frac{\partial J}{\partial \theta_k}(z_i,0),\]
then each of the terms on the left-hand side of the expression \eqref{fl} vanishes along the locus where all $\theta_i=0$. But the right-hand side also vanishes along this locus, because by part (i) the function $J$ is odd in the  co-ordinates $\theta_j$.  So if the difference of the two sides of \eqref{fl} is independent of the  co-ordinates $\theta_j$, then after the above modification it must  vanish identically.
\end{remark}

\subsection{Deformed Joyce connection}

Suppose given a  Joyce structure on a complex manifold $M$. The formula 
\begin{equation}
\label{shrike}{\cA}^{(\hbar)}_X = \nabla_X + \Theta_X+ \frac{1}{\hbar} \cdot m_X\end{equation}defines a pencil of  connections with values in the bundle of Lie algebras $\fg_M$.
Applying the homomorphism  of Lie algebras $\rho\colon \fg\to \operatorname{vect}_\eta(\bT_+)$ of Section \ref{la} fibrewise then gives a pencil of flat Ehresmann connections on the tangent bundle $\pi\colon \cT_M\to M$, whose horizontal sections are spanned by the vector fields
\begin{equation}
\label{strike}\cA^{(\hbar)}_X= \cH_X +\Ham_{m_X(J)} +\frac{1}{\hbar} \cdot m_X.\end{equation}
In terms of a local co-ordinate system $(z_1,\cdots,z_n,\theta_1,\cdots,\theta_n)$ as in Section \ref{geom} we have
\begin{equation}
\label{matrix}\cA^{(\hbar)}_{\frac{\partial}{\partial z_i}}= \frac{\partial}{\partial z_i} +\sum_{j,k} \eta_{jk}\cdot \frac{\partial^2 J}{\partial \theta_i\partial \theta_j} \cdot \frac{\partial}{\partial \theta_k} +\frac{1}{\hbar} \cdot \frac{\partial}{\partial \theta_i}.\end{equation}

As in the Frobenius case we actually get more, namely  an Ehresmann  connection on the bundle $p^*(\cT_M)$ over the space $M\times \bP^1$, where $p\colon M\times \bP^1 \to M$ is the projection map.
 The space of  horizontal sections of this connection is spanned by the vector fields 
\begin{equation}
\label{extjoyce}{\cA}_X= \cH_X +\Ham_{m_X(J)} +\frac{1}{\hbar} \cdot m_X, \qquad {\cA}_{\frac{\partial}{\partial \hbar}}=\frac{\partial}{\partial \hbar} - \frac{1}{\hbar}\cdot \Ham_{m_E(J)} -\frac{1}{\hbar^2} \cdot m_E,\end{equation}
where we abuse notation by identifying  vector fields on the spaces $M$ and $\cT_M$ with vector fields on $M\times \bP^1$ and $\cT_M\times \bP^1$ respectively, which are constant in the $\bP^1$ direction.
We call \eqref{extjoyce}  the \emph{deformed Joyce connection}.

\begin{prop}
The deformed Joyce connection  is flat.
\end{prop}

\begin{proof}
The deformed Joyce connection \eqref{extjoyce} is obtained by applying the homomorphism of Lie algebras $\rho\colon \fg\to \operatorname{vect}_\eta(\bT_+)$ of Section \ref{la} fibrewise to the  connection
\[{\cA}_X = \nabla_X + \Theta_X+ \frac{1}{\hbar} \cdot m_X,\qquad {\cA}_{ \frac{\partial}{\partial \hbar}} = \frac{\partial}{\partial \hbar}-\frac{1}{\hbar}\cdot {\Theta_E} -\frac{1}{\hbar^2} \cdot m_E.\]
Exactly as in Section \ref{firststructure}, the flatness of this connection is equivalent to the relations 
\eqref{nfirst}, \eqref{nthird} and \eqref{nsecond}. As explained in Section \ref{heur}, these follow directly from the axioms of a Joyce structure.
\end{proof}

Setting $\hbar=\infty$ in \eqref{strike} gives a particular flat Ehresmann connection on $\cT_M$ which we call the \emph{Joyce connection}. It is the analogue of the Levi-Civita connection on a  Frobenius manifold.
In co-ordinates the horizontal sections of the Joyce connection are spanned by the vector fields
\begin{equation}
\label{nons}\cA_{\frac{\partial}{\partial z_i}}= \frac{\partial}{\partial z_i} + \sum_{j,k} \eta_{jk}\cdot \frac{\partial^2 J}{\partial \theta_i\partial \theta_j} \cdot \frac{\partial}{\partial \theta_k}.\end{equation}

\begin{remark}\label{andrea}Note that if the form $\eta$ fails to be non-degenerate the homomorphism $\rho$ of \eqref{rho} has a kernel, and  the passage from \eqref{shrike} to \eqref{strike} therefore loses information. In the extreme case  $\eta=0$, all Hamiltonian vector fields vanish, and the Joyce connection \eqref{extjoyce} coincides with the linear connection $\nabla$. We shall return to this point in Section \ref{sdouble}.
\end{remark} 

\subsection{Isomonodromic family and Stokes data}
As with the deformed flat connection on a Frobenius manifold, there are two complementary points-of-view on the deformed Joyce connection \eqref{extjoyce}. In the first, we forget the derivatives in the $\hbar$-direction, and view \eqref{strike} as defining a pencil of flat Ehresmann connections on $\cT_M$ which deform the Joyce connection \eqref{nons}. In the second, we consider the family of meromorphic connections on $\bP^1$ of the form
\begin{equation}
\label{biggy}\nabla=d-\bigg(\frac{Z}{\hbar^2}+\frac{H}{\hbar} \bigg)d\hbar,\end{equation}
parameterised by the points $p\in M$, and taking values in the Lie algebras $\fg=\fg_{M,p}$. Here $Z=m_E$ is a translation-invariant vector field on the fibre $\cT=\cT_{M,p}$, and at least heuristically,  $H=\Theta_E=m_E(J)$ is a holomorphic function on $\cT$ which descends to an algebraic function on the torus $\bT=\cT/\Gamma^*$. In co-ordinates we have
\begin{equation}
\label{gola}Z=\sum_i z_i\cdot \frac{\partial}{\partial \theta_i}\in\fh, \qquad H=\sum_i z_i \cdot \frac{\partial J}{\partial \theta_i}\in \fg^{\od}.\end{equation}
As we discussed above, in practice, the function $J$   need not restrict to give  an algebraic function on the tori \eqref{t}, and the same remark applies to $H$. Nonetheless, it is sometimes useful for heuristic purposes to identify $H$ with its Fourier expansion
\begin{equation}
\label{flu}H=\sum_{\gamma\in \Gamma\setminus\{0\}} H(\gamma)\cdot  x^\gamma.\end{equation}
Note that the assumption that $J$ restricts to an odd function on $\cT$ means that $H$ is even, and hence $H(-\gamma)=H(\gamma)$. This is the condition that $H
\in \fg$ lies in the $+1$ eigenspace of the involution $\iota$ of Section \ref{la}.

Recall that the roots of the Lie algebra $\fg$ of Section \ref{la} are  the non-zero elements of the lattice $\Gamma=\Gamma_{M,p}$. Following the ideas of Section \ref{sthree}, but replacing the Lie algebra $\gl(\cT_{M,p})$ of Section \ref{form2} with the Lie algebra $\fg_{M,p}$ of Section \ref{la}, we should  expect the Stokes data of the connection \eqref{biggy} at the singularity $\hbar=0$ to consist of  Stokes rays of the form $\bR_{>0}\cdot Z(\gamma)\subset \bC^*$ for nonzero elements $\gamma\in \Gamma$, and associated Stokes factors
\[\stokes(\ell)=\exp\Big(-\hspace{-.8em}\sum_{\gamma\in \Gamma: Z(\gamma)\in \ell} \DT(\gamma) \cdot x_\gamma\Big),\]
encoded by elements $\DT(\gamma)\in \bC$ exactly as in \eqref{skegn2}, and satisfying the symmetry condition $\DT(-\gamma)=\DT(\gamma)$.  In the context of stability conditions on CY$_3$ triangulated categories, it is these elements  which it is natural to identify with the DT invariants, and we will therefore restrict attention to the case when all $\DT(\gamma)\in \bQ$. The Stokes data is then encoded by the notion of a BPS  structure which is the subject of the next section.


\section{BPS structures}
\label{next}

In this section we recall the basic definitions concerning BPS structures from \cite{RHDT}. A BPS structure  axiomatises the output of unrefined Donaldson-Thomas theory applied to a CY$_3$ triangulated category. A variation of BPS structures describes the way the invariants change under variation of stability condition, the main ingredient being the Kontsevich-Soibelman wall-crossing formula.  In the analogy described in this paper, BPS structures correspond to the FS structures of Section \ref{fs}: they encode the Stokes data of the connection \eqref{biggy}.

\subsection{BPS structures}

The following definition  is a special case of the notion of stability data on a graded Lie algebra \cite[Section 2.1]{KS}. It was also studied by Stoppa and his collaborators \cite[Section 3]{BSt}, \cite[Section 2]{FGS}.

\begin{defn}
\label{bpss}
A \emph{BPS structure}  consists of \begin{itemize}
\item[(a)] a finite-rank free abelian group $\Gamma\isom \bZ^{\oplus n}$, equipped with a skew-symmetric form \[\<-,-\>\colon \Gamma \times \Gamma \to \bZ;\]

\item[(b)] a homomorphism of abelian groups
$Z\colon \Gamma\to \bC$;\smallskip

\item[(c)] a map of sets
$\Omega\colon \Gamma\to \bQ;$
\end{itemize}

satisfying the following properties:

\begin{itemize}
\item[(i)] symmetry: $\Omega(-\gamma)=\Omega(\gamma)$ for all $\gamma\in \Gamma$;\smallskip
\item[(ii)] support property: fixing a norm  $\|\cdot\|$ on the finite-dimensional vector space $\Gamma\tensor_\bZ \bR$, there is a constant $C>0$ such that 
\begin{equation}
\label{support}\Omega(\gamma)\neq 0 \implies |Z(\gamma)|> C\cdot \|\gamma\|.\end{equation}
\end{itemize}
\end{defn}

The lattice $\Gamma$ will be  called the \emph{charge lattice}, and the form $\<-,-\>$ is the \emph{intersection form}. The group homomorphism $Z$ is  called the \emph{central charge}. The rational numbers $\Omega(\gamma)$ are called \emph{BPS invariants}. A class $\gamma\in \Gamma$ will be called \emph{active} if $\Omega(\gamma)\neq 0$. The \emph{Donaldson-Thomas (DT) invariants}  are defined by the expression
\begin{equation}
\label{bps}\DT(\gamma)=\sum_{ \gamma=m\alpha}   \frac{1}{m^2} \, \Omega(\alpha)\in \bQ,\end{equation}
where the sum is over integers $m>0$ such that $\gamma$ is divisible by $m$ in the lattice $\Gamma$.

It will be useful to  introduce some  terminology for describing BPS structures of various special kinds.

\begin{defn}
\label{blahblah}
We say that a BPS structure $(\Gamma,Z,\Omega)$ is
\begin{itemize}

\item[(a)] \emph{finite} if there are only finitely many nonzero BPS invariants $\Omega(\gamma)$;\smallskip

\item[(b)]  \emph{ray-finite}, if for any ray $\ell\subset \bC^*$ there are only finitely many active classes $\gamma\in \Gamma$ for which $Z(\gamma)\in \ell$;\smallskip

\item[(c)] \emph{convergent}, if for some $R>0$
\[\big.\sum_{\gamma\in \Gamma} |\Omega(\gamma)|\cdot e^{-R|Z(\gamma)|}<\infty;\]

\item[(d)] \emph{uncoupled}, if  for any two active classes $\gamma_1,\gamma_2\in \Gamma$  one has
 $\<\gamma_1,\gamma_2\>=0$;\smallskip

\item[(e)] \emph{generic}, if for any two active classes $\gamma_1,\gamma_2\in \Gamma$ one has 
\[\bR_{>0} \cdot Z(\gamma_1)=\bR_{>0} \cdot Z(\gamma_2) \implies \< \gamma_1, \gamma_2\>=0;\]

\item[(f)] \emph{integral}, if the BPS invariants $\Omega(\gamma)\in \bZ$ are all integers.
\end{itemize}
\end{defn}

\subsection{Twisted torus}
\label{twto}

As well as the algebraic torus $\bT_+$, we will also consider an associated torsor
\[\bT_-= \big\{g\colon \Gamma \to \bC^*: g(\gamma_1+\gamma_2)=(-1)^{\<\gamma_1,\gamma_2\>} g(\gamma_1)\cdot g(\gamma_2)\big\},\]
which we call the \emph{twisted torus}.
The torus $\bT_+$ acts freely and transitively on the twisted torus $\bT_-$  via
\[(f\cdot g)(\gamma)=f(\gamma)\cdot g(\gamma)\in \bC^*, \quad f\in \bT_+, \quad  g\in \bT_-.\]
Choosing a base-point $\sigma\in \bT_-$ therefore gives a bijection \begin{equation}
\label{bloxy}\rho_{\sigma}\colon \bT_+\to \bT_-, \qquad f\mapsto f\cdot \sigma.\end{equation}

We can use the identification $\rho_{\sigma}$ to give $\bT_-$ the structure of an algebraic variety. The result is independent of the choice of $g_0$, since the translation maps on $\bT_+$ are algebraic. 
The co-ordinate ring of $\bT_-$ is spanned as a vector space by the functions
\[x_\gamma\colon \bT_-\to \bC^*, \qquad x_\gamma(g)=g(\gamma)\in \bC^*,\]
which we  refer to as \emph{twisted characters}.  Thus
\[\bC[\bT_-]=\bigoplus_{\gamma\in \Gamma} \bC\cdot x_\gamma, \qquad x_{\gamma_1}\cdot x_{\gamma_2}=(-1)^{\<\gamma_1,\gamma_2\>}\cdot x_{\gamma_1+\gamma_2}.\]
Similarly, the Poisson structure \eqref{poisson} on $\bT_+$  is invariant under translations, and hence can be transferred to $\bT_-$ via the map \eqref{bloxy}. It is given on twisted characters by 
 \[\{x_{\alpha}, x_{\beta}\}= \<\alpha,\beta\>\cdot x_{\alpha}\cdot x_{\beta}.\]

There is an involution of $\bT_-$, which we call the \emph{twisted inverse map}, which acts on twisted characters via $x_{+\gamma}\leftrightarrow x_{-\gamma}$. The identification  $\rho_\sigma$ intertwines this twisted inverse map on $\bT_-$ with the standard inverse map on $\bT_+$ precisely if the base-point 
 $\sigma$ lies in the subset
\[\Big\{\sigma\colon \Gamma\to \{\pm 1\} \b \sigma(\gamma_1+\gamma_2)=(-1)^{\langle\gamma_1,\gamma_2\rangle}\cdot \sigma(\gamma_1)\cdot \sigma(\gamma_2)\Big\}\subset \bT_-,\]
whose points are called \emph{quadratic refinements} of the form $\<-,-\>$.

\subsection{BPS automorphisms}
 
Associated to any ray $\ell\subset \bC^*$ is a formal sum of twisted characters
\[\DT(\ell)=-\hspace{-.8em}\sum_{\gamma\in \Gamma: Z(\gamma)\in \ell} \DT(\gamma) \cdot x_\gamma.\]
We call the ray active if this sum is nonzero. In general there will be countably many active rays.
Naively, we would like to view  $\DT(\ell)$  as a well-defined holomorphic function on the twisted torus $\bT_-$, and consider the associated time 1 Hamiltonian flow as a Poisson automorphism
\[\stokes(\ell)\in \Aut(\bT_-).\]
We refer to the resulting automorphism $\bS(\ell)$  as the \emph{BPS automorphism} associated to the ray $\ell$. In \cite{RHDT} we considered several ways of making  sense of $\stokes(\ell)$  which we  briefly summarise here:

\begin{itemize}
\item[(i)] {\it Formal approach.}  If we are only interested in the elements $\stokes(\ell)$ for rays $\ell\subset \bC^*$ lying in a fixed  acute sector $\Delta\subset \bC^*$, then we can work with a variant of the algebra $\bC[\bT_-]$ consisting of formal sums of the form \[\big.\sum_{Z(\gamma)\in \Delta} a_\gamma\cdot x_\gamma,\quad a_\gamma\in \bC,\] such that for any $H>0$ there are only finitely many terms with $|Z(\gamma)|<H$. See \cite[Appendix B]{RHDT}. It has the advantage of not requiring any extra assumptions, and gives a rigorous definition of variations of BPS structures (see below).\smallskip

\item[(ii)] {\it Analytic approach.} Let us suppose that the BPS structure is convergent. In \cite[Appendix B]{RHDT},   we  associate to each convex sector $\Delta\subset \bC^*$, and each real number $R>0$, a non-empty analytic open subset $U_\Delta(R)\subset \bT_-$ defined to be the interior of the subset 
\[\big\{g\in \bT_-: Z(\gamma)\in \Delta \text{ and } \Omega(\gamma)\neq 0\implies |g(\gamma)|<\exp(-R|Z(\gamma)|)\big\}\subset \bT_-.\]
 We then show that  for sufficiently large $R>0$, and any  active ray $\ell\subset \Delta$, the formal series $\DT(\ell)$ is absolutely convergent on   $U_\Delta(R)\subset \bT_-$, and that the time 1 Hamiltonian flow of the resulting  function defines a holomorphic  embedding \[\stokes(\ell) \colon U_\Delta(R)\rightarrow \bT_-.\]
 We can then view this map as being a partially-defined automorphism of $\bT_-$.\smallskip
 
 \item[(iii)] {\it Birational approach.} In  the case of a generic, integral and ray-finite  BPS structure,  the partially-defined automorphisms $\stokes(\ell)$  discussed in (ii) extend to  birational transformations of $\bT_-$ whose  pullback  of twisted characters is expressed by the formula
\begin{equation}\label{bir}\stokes(\ell)^*(x_\beta)=x_\beta\cdot \prod_{Z(\gamma)\in \ell}(1-x_\gamma)^{\,\Omega(\gamma)\<\gamma,\beta\>}.\end{equation}
\end{itemize}

In fact these different approaches will not play a very important role in this paper, and for heuristic purposes we shall continue to regard the elements $\stokes(\ell)$ as being genuine automorphisms of $\bT_-$. It is worth noting though that approaches (ii) and (iii) involve applying the Lie algebra homomorphism \eqref{rho}, and hence lose information when the form $\<-,-\>$ has a kernel. We will return to this point in Section \ref{sdouble}.

\subsection{Variations of BPS structure}
\label{vbpss}

The wall-crossing behaviour of DT invariants under changes in stability parameters is controlled by the Kontsevich-Soibelman wall-crossing formula. This forms the main ingredient in the following  definition of a variation of BPS structures, which is a special case of the  notion of a continuous family of stability structures  from \cite[Section 2.3]{KS}.
Full details can be found in \cite[Appendix B]{RHDT}: here we just give the rough idea.

\begin{defn}
\label{co}
A \emph{variation of BPS structures} over a complex manifold $M$ consists of a collection of BPS structures $(\Gamma_p,Z_p,\Omega_p)$ indexed by the points  $p\in M$, such that

\begin{itemize}
\item[(a)] the charge lattices  $\Gamma_p$ form a local system of abelian groups, and the intersection forms $\<-,-\>_p$ are covariantly  constant;\smallskip

\item[(b)] given a covariantly constant family of elements $\gamma_p\in \Gamma_p$, the central charges $Z_p(\gamma_p)\in \bC$ vary holomorphically;\smallskip

\item[(c)] the constant in the support property \eqref{support} can be chosen uniformly on compact subsets;\smallskip

\item[(d)] for each acute sector $\Delta\subset \bC^*$, the anti-clockwise product over active rays in $\Delta$\[ \stokes_p(\Delta)=\prod_{\ell\subset \Delta} \stokes_p(\ell) \in \Aut(\bT_{p,-}),\]
is covariantly constant as $p\in M$ varies, providing the boundary rays of $\Delta$ are never active.
\end{itemize}
\end{defn}

It requires some work to make rigorous sense of the wall-crossing formula, condition (d). This is done in \cite[Appendix B]{RHDT} following  the ideas of \cite{KS}. This needs no convergence assumptions and completely describes the behaviour of the BPS invariants as the point $p\in M$ varies: once one knows all the  invariants $\Omega(\gamma)$ at some point of $M$, they are determined at all other points.

\begin{remark}
There is a small discrepancy between Definition \ref{embryonic} of a variation of FS structures, and the above definition of a variation of BPS structures. Namely, the products of Stokes factors over rays are taken in different directions (clockwise or anti-clockwise) in the two cases. The explanation for this is that, for various reasons, it is convenient to take the underlying group $G$ in the case of BPS structures to be the \emph{opposite} of the group of automorphisms of $\bT_-$ (at least at the level of heuristics). Thus $G$ is related to the group of automorphisms of the ring of functions $\bC[\bT_-]$, and has a  Lie algebra given by the Poisson bracket on $\bC[\bT_-]$. Compare \cite[Remark 4.6 (i)]{RHDT}, and see Section \ref{indianrock} below for more details.
\end{remark}

Replacing $M$ by a contractible open subset or cover we can use trivialise the local system of charge lattices and obtain a holomorphic 
period map
\begin{equation}
\label{per}\varpi\colon M\to \Hom_\bZ(\Gamma,\bC) , \qquad p\mapsto Z_p.\end{equation}We shall say that a variation of BPS structures is \emph{miniversal} if the derivative of this map is everywhere non-vanishing, so that it is a local isomorphism.

\subsection{Quadratic refinements}
\label{qr}
In the definition of BPS automorphisms  we consider Hamiltonian vector fields on the twisted torus $\bT_-$ rather than the  torus $\bT_+$ itself. This is essential if the notion of a variation of BPS structures is to reproduce the wall-crossing behaviour of DT invariants. On the other hand, in the notion of a Joyce structure we consider vector fields on the family of tori $\bT_+$, rather than $\bT_-$, and this is also important, since for example, in Section \ref{linear} we will define the linear Joyce connection  by examining the  behaviour of these vector fields near the  identity of $\bT_+$.

To combine  these  requirements we use the identification \begin{equation}
\label{bloxy2}\rho_{\sigma}\colon \bT_+\to \bT_-, \qquad f\mapsto f\cdot \sigma\end{equation}
explained in Section \ref{twto}, corresponding to a particular choice of element $\sigma\in \bT_-$. Under this identification the twisted and untwisted characters are related  via $x_\gamma=\sigma(\gamma)\cdot y_\gamma$. If the symmetry property of the BPS invariants from Definition \ref{bpss} is to correspond to the oddness \eqref{oddd} of the Joyce function $J$, the map \eqref{bloxy2} needs to intertwine the inverse map on $\bT_+$ and the twisted inverse map on $\bT_-$, and as explained in Section \ref{twto},  this implies that  $\sigma$ must be chosen to lie in the finite subset of $\bT_-$ consisting of quadratic refinements of the form $\<-,-\>$.

Suppose given a variation of BPS structures over a manifold $M$. Since the lattices $\Gamma_p$ form a local system over $M$, and the form $\<-,-\>_p$ is covariantly constant, once a quadratic refinement is chosen at one point $p\in M$, there is at most one way to extend it in a covariantly constant way to the whole of $M$. But if the answer is to be well-defined we must choose a quadratic refinement which is preserved by parallel transport around loops in $M$, and in general this may only be possible after passing to a cover of $M$.  Conceptually we can say that we must choose a section of a $\{\pm 1\}$-torsor over $M$. It is interesting to note the analogy between this torsor and the one appearing  in Remark \ref{em}.


\section{Inverting the Stokes map}
\label{top}
In Section \ref{stwo}  we explained how a tame Frobenius structure on  a manifold $M$ gives rise to an isomonodromic family of meromorphic connections \eqref{conn} on a family of trivial vector bundles over $\bP^1$. Up to some minor technicalities, the data of this family of connections is enough to reconstruct the Frobenius structure. In Section \ref{sthree} we explained how the Stokes data of this family of connections  in turn defines what we called a variation of FS structures. In this section we begin by discussing the inverse problem of reconstructing the family of connections \eqref{conn} from a variation of FS structures. We  then move on to the analogous but more difficult problem of reconstructing a Joyce structure from a variation of BPS structures.  

\subsection{Properties of the Stokes map}

Let us again consider  connections  of the form
\begin{equation}
\label{ozfire}
\nabla=d-\bigg(\frac{U}{\hbar^2}+\frac{V}{\hbar}\bigg)d\hbar,
\end{equation}
as studied in Section \ref{sthree}. The Lie algebra relevant to this context is $\fg=\gl_n(\bC)$. Recall from Section \ref{stomap} the definition of the Stokes map
\begin{equation}
\label{stok}\cS\colon \god\to\god, \qquad \cS(V)=R,\end{equation}
which for fixed $U\in \fh^{\reg}$ sends the connection
\eqref{ozfire} to its Stokes data
\begin{equation*}
R=
\sum_{\alpha\in \Phi}D(\alpha)\cdot E_\alpha\in\god\end{equation*}
at the singularity $\hbar=0$.
To understand the properties of this map it is useful to compare it to  a closely-related map, namely the irregular RH map. 

Rather than studying connections
of the form \eqref{ozfire}, one can consider instead meromorphic
connections on the trivial vector bundle over the unit
disc $\Delta\subset \bC$ of the form
\begin{equation}
\label{nabb}
\nabla=d-\bigg(\frac{U}{\hbar^2}+\frac{V(\hbar)}{\hbar}\bigg)d\hbar,
\end{equation}
where  $V\colon \Delta \to \god$ is an arbitrary  holomorphic function. Such connections are naturally considered up to the action of holomorphic gauge transformations $P\colon \Delta\to \GL_n(\bC)$ satisfying the condition $P(0)=1$.
One can define the Stokes data at $\hbar=0$ for  such connections in exactly the same way as before, and it is immediate from the definitions that gauge-equivalent connections have the same Stokes data.

For fixed $U\in\fh^{\reg}$, the irregular RH map sends the  gauge equivalence class of the connection \eqref{nabb}
to its Stokes data. The results of \cite{BJL} show that this map is  an isomorphism. In contrast,  although the Stokes map \eqref{stok} is a local isomorphism of complex manifolds, it is neither injective
nor surjective in general. The reason for this is that  not
every connection of the form \eqref{nabb} can be put into the constant coefficient
form \eqref{ozfire} by a gauge transformation, and on the other hand, distinct elements  $V$ in \eqref{ozfire} can lead to gauge equivalent connections.  For more details on this see,  for example, \cite[Remark 12]{Boalch} and the reference given there.

\subsection{Stokes map in the FS structure case}
\label{fSm}
Power series expansions  for the Stokes map $\cS$, and its inverse, in a neighbourhood of the origin in the vector space $\god$ were derived in \cite{VTL}. They depend on a fixed element $U\in \fh^{\reg}$. The expansion of $\cS^{-1}$  takes the form
\begin{equation}
\label{main2}
V_\alpha=
\sum_{n\geq 1}
\sum_{\stackrel{\alpha_i\in\Phi}{\alpha_1+\cdots+\alpha_n=\alpha}}
F_n(U(\alpha_1),\ldots, U(\alpha_n))\cdot\prod_{i=1}^n D(\alpha_i)\cdot 
E_{\alpha_1} \cdot E_{\alpha_2}\cdots E_{\alpha_n},
\end{equation}
where the piecewise-holomorphic functions $F_n\colon (\bC^*)^n\to \bC$  are     closely related to  multi-logarithms. Each term in the sum over $n$ on the right of \eqref{main2} is a finite sum of elements in the universal enveloping algebra $U(\fg)$, but  the symmetry properties of the functions $F_n$ ensure that  these  sums  in fact lie in  $\fg\subset U(\fg)$.

It was shown in \cite{VTL} that for sufficiently small $D(\alpha_i)\in \bC$ the sum over $n\geq 1$ in \eqref{main2} is absolutely convergent, and the resulting connection \eqref{ozfire} has
Stokes factors given by the relation \eqref{skegn2}.  In this case therefore, the sum \eqref{main2} successfully inverts the Stokes map, and therefore gives a potential method for associating a connection \eqref{conn} to a FS structure. But since, as we explained above, the Stokes map is not expected to be injective or surjective in general, one should not expect the series \eqref
{main2} to be convergent for arbitrary choices of $D(\alpha_i)\in \bC$. And since, for a given $U\in \fh^{\reg}$, nothing is known about how small the $D(\alpha_i)$ need to be to ensure that \eqref{main2} converges, this approach appears to be  only of interest in a formal setting. 

Note that when the sum \eqref{main2} is absolutely convergent, the resulting connections have constant Stokes data \eqref{skegn2} as the leading term $U\in\fh^{\reg}$ varies, and therefore the expressions $V_\alpha(U)$ must satisfy the differential equation \eqref{diffeq}. This is easily seen to be equivalent to the statement that the functions $F_n$  satisfy the differential equation
\begin{equation}
\label{rem}dF_n(u_1,\cdots, u_n)=\sum_{1\leq i < n} F_i(u_1,\cdots,u_i)\cdot F_{n-i} (u_{i+1},\cdots,u_n)\cdot d\log\left(\frac{u_{i+1}+\cdots+u_n}{u_1+\cdots+u_i}\right)\end{equation}
on the dense open subsets of $(\bC^*)^n$ on which they are holomorphic. 

\subsection{Riemann-Hilbert problem in the FS structure case}

Consider again the problem of reconstructing a connection of the form \eqref{ozfire} from its associated FS structure. Rather than working with the generally divergent series \eqref{main2}, we can try to 
construct the fundamental solutions $\Psi_r(\hbar)$ directly by their defining properties, and then obtain the associated connection $\nabla$ by differentiation.

Recall from Section \ref{canic} the definition of the open half-plane $\bH_r\subset \bC^*$ centered on a ray $r\subset \bC^*$. Consider an FS structure $(T,U,D)$, with associated Stokes factors $\stokes_\ell\in \GL(T)$ as defined in \eqref{skegn2}. The associated  RH problem is as follows.

\begin{problem}\label{FrobRH}For each non-Stokes ray $r\subset \bC^*$, find a holomorphic function
\[\Psi_r\colon \bH_r\to \GL(T)\]
satisfying the following conditions:
\begin{itemize}
\item[(i)] if $r_1$ and $r_2$ are the boundary rays, taken in clockwise order, of a convex sector $\Delta\subset \bC^*$ containing a single Stokes ray $\ell$, then
\[\Psi_{r_2} (\hbar)=\Psi_{r_1} (\hbar)\cdot \stokes_\ell,\]
for all $\hbar\in \bH_{r_-}\cap \bH_{r_+}$;\smallskip

\item[(ii)] as $\hbar\to 0$ in the half-plane $\bH_r$ we have $\Psi_r(\hbar) \cdot \exp(U/\hbar)\to \id$;
\smallskip

\item[(iii)]  as $\hbar\to \infty$ in the half-plane $\bH_r$ the norm $\|\Psi_r(\hbar)\|$ has at most polynomial growth.
\end{itemize}
\end{problem}

Given a connection of the form \eqref{ozfire}, the canonical solutions of Theorem \ref{bjl} satisfy condition (ii) by  their definition, and condition (i) is the definition of the Stokes factors.
Condition (iii) holds because the connection \eqref{ozfire} has a regular singularity at $\hbar=\infty$, which implies that any fundamental solution has moderate growth.

Given a solution to Problem \ref{FrobRH}  we can construct a  connection \eqref{conn} having the Stokes data \eqref{skegn2} by writing
\[V= \hbar \cdot \frac{d\Psi_r}{d\hbar}\cdot \Psi_r^{-1}-\frac{U}{\hbar},\]
which is independent of the choice of ray $r\subset \bC^*$ by property (i). As in Section \ref{fSm} however, since the Stokes map is not an isomorphism, we cannot expect existence of solutions to the Problem \ref{FrobRH} for arbitrary choices of FS structure. General results show that when the $D(\alpha)\in \bC$ are sufficiently small, so that the  Stokes matrices $\stokes(\ell)$ are close to the identity, Problem \ref{FrobRH} can be solved using an integral equation \cite[Proposition 3.13]{D1}. But again, for a given choice of $U\in\fh^{\reg}$, there is no explicit description of what `sufficiently small' means.

When Problem \ref{FrobRH} has a solution, it need not be unique, since we can obtain another one by  pre-multiplying by a holomorphic function $P\colon \bC\to \GL_n(T)$ satisfying $P(0)=1$ and such that $\|P(\hbar)\|$ has moderate growth as $\hbar\to \infty$. 
To eliminate this indeterminacy one can introduce a more complicated RH problem which also takes into account the existence of canonical solutions to the connection \eqref{ozfire} near $\hbar=\infty$. Relating these to the canonical solutions near $\hbar=0$ leads to new monodromy invariants called central connection matrices. The resulting RH problems are discussed in detail in \cite[Lecture 4]{D2}.

\subsection{Stokes map in the BPS structure case}

Let us now consider the, presumably much more difficult, problem of reconstructing a connection from its Stokes data in the case of the infinite-dimensional Lie algebra $\fg$ of Section \ref{joyce}. More precisely, we would like to be able to construct Joyce structures for which the associated connection \[\nabla=d-\bigg(\frac{Z}{\hbar^2}+\frac{H}{\hbar} \bigg)d\hbar,\]
 has Stokes data described by a given BPS structure. One possibility is to use the expansion \eqref{main2}. The analogous expression for the Lie algebra $\fg$  is \begin{equation}
\label{jojo}
H_\gamma=
\sum_{n\geq 1}
\sum_{\stackrel{\gamma_i\in\Gamma\setminus\{0\}}{\gamma_1+\cdots+\gamma_n=\gamma}}
 F_n(U(\gamma_1),\ldots, U(\gamma_n))\cdot (-1)^{n}\cdot \prod_{i=1}^{n} \DT(\gamma_i)\cdot 
x_{\gamma_1} *x_{\gamma_2}*\cdots * x_{\gamma_n}.
\end{equation}
As with \eqref{main2}, this expression is viewed as living in the universal enveloping algebra $U(\fg)$, but due to the symmetry properties of the functions $F_n$, it should actually output elements of the Lie algebra $\fg$. If the sum in \eqref{jojo} is absolutely convergent we can  define
\begin{equation}
\label{flu2}H= \sum_{\gamma\in \Gamma} H_\gamma \cdot e^{\theta(\gamma)}, \qquad  J=\sum_{\gamma\in \Gamma} \frac{H_\gamma \cdot e^{\theta(\gamma)}}{Z(\gamma)},\end{equation}
and so obtain the Joyce function. 

The formula \eqref{jojo}  was first written down by Joyce \cite{HolGen}, and was also investigated by Stoppa and collaborators  \cite{BS,FGS}. Given the discussion of Section \ref{fSm} however, we cannot expect  that the sum \eqref{jojo} is convergent in general, and in fact the  convergence problems are  even more acute in the setting of the infinite-dimensional Lie algebra $\fg$, since even for a fixed value of $n$ the sum in \eqref{jojo} is potentially infinite. 
One possiblity is to try to introduce formal parameters as in \cite{BS}, but since this depends on a non-canonical choice of positive cone $\Gamma^+\subset \Gamma$,  given a variation of BPS structure over a base $M$, the resulting quantities $H_\gamma$ will not in general be holomorphic functions on $M$.

\begin{remark}
A  remarkable feature of the paper \cite{HolGen} is that Joyce wrote down the expansion \eqref{jojo} without any input from the theory of irregular connections or Stokes data. His motivation was simply to write down a generating function involving the invariants $\DT(\alpha)$  which would give holomorphically varying quantities $H_\alpha$ on the space of stability conditions. Postulating the form \eqref{jojo}, and making some other natural assumptions, he was able to show that the functions $F_n$ were uniquely determined and satisfied the differential equation \eqref{rem}, and hence that the resulting quantities $H_\alpha$ satisfied the equation \eqref{diffeq}.
The observation that this equation describes  isomonodromic families of connections of the form \eqref{biggy}   led directly to the paper \cite{VTL}, which shows that Joyce's expansion can be interpreted as the formal inverse to the Taylor expansion of the Stokes map.
\end{remark}

\subsection{Riemann-Hilbert problem in the BPS structure case}
\label{rh_intro}

In this paper we will attempt to use the RH approach to find Joyce structures whose Stokes data is described by a given BPS structure. 
Let $(\Gamma, Z, \Omega)$ be a convergent BPS structure, and denote by $\bT_-$ the  twisted torus. The associated RH problem depends also on an element $\xi\in \bT_-$ which we call the constant term.  It involves meromorphic maps
$X_{\rr}\colon \bH_\rr\to \bT_-$ depending on a choice of non-active ray $r\subset \bC^*$. Composing   with the twisted characters of $\bT_-$ we can equivalently consider functions
\[X_{\rr,\gamma}\colon \bH_\rr\to \bC^*,\qquad X_{\rr,\gamma}(\hbar)=x_\gamma(X_\rr(\hbar)).\]
The RH problem reads as follows; we will explain the precise analogy with Problem \ref{FrobRH} in the next subsection.

\begin{problem}
\label{dtsect}
For each non-active ray $\rr\subset \bC^*$ we  seek a meromorphic function
\[X_\rr\colon \bH_\rr\to \bT_-\]
such that the following three conditions are satisfied:
\begin{itemize}

\item[(RH1)] if the non-active rays $\rr_1,\rr_2\subset \bC^*$ are the  boundary rays of a convex sector $\Delta\subset \bC^*$ taken in clockwise order, then
\[X_{\rr_2}(\hbar)= \stokes(\Delta)( X_{\rr_1}(\hbar)), \]
for all $\hbar\in \bH_{\rr_1}\cap \bH_{\rr_2}$ with  $ 0<|\hbar|\ll 1$;
\smallskip

\item[(RH2)] 
for each non-active ray $\rr\subset \bC^*$ and each class $\gamma\in \Gamma$ we have
\[\exp(Z(\gamma)/\hbar)\cdot X_{\rr,\gamma}(\hbar) \to \xi(\gamma)\]
as $\hbar\to 0$ in the half-plane $\bH_\rr$;\smallskip

\item[(RH3)]  for any class $\gamma\in \Gamma$ and any non-active ray $\rr\subset \bC^*$, there exists $k>0$ such that
\[|\hbar|^{-k} < |X_{\rr,\gamma}
(t)|<|\hbar|^k,\]
for  $\hbar\in \bH_\rr$ satisfying $|\hbar|\gg 0$.
\end{itemize}
\end{problem}

To make rigorous sense of the condition (RH1) we need to use the results of \cite{RHDT} which guarantee that the clockwise-product of Stokes factors
\[\ \stokes(\Delta)=\prod_{\ell\subset \Delta} \stokes(\ell)\]
 has a well-defined  meaning as a partially-defined automorphism of the twisted torus $\bT_-$. See the discussion following \cite[Problem 4.3]{RHDT} for more details.

\subsection{Riemann-Hilbert problem discussion}
\label{indianrock}

At first sight Problem \ref{dtsect} looks somewhat different to Problem \ref{FrobRH}. In this section we explain the precise analogy between the two problems. In this discussion we will ignore issues to do with lack of convergence  and pretend that the Stokes factors \[\stokes(\ell)=\exp\Big(-\hspace{-.8em}\sum_{\gamma\in \Gamma: Z(\gamma)\in \ell} \DT(\gamma) \cdot x_\gamma\Big)\]
are well-defined elements of the corresponding group $G$ of Poisson automorphisms of the algebra $\bC[\bT_-]$. The obvious analogue of Problem \ref{FrobRH} then involves maps
\begin{equation}
\label{wawa}\Psi_r\colon \bH_r\to G=\Aut_{\{-,-\}} (\bC[\bT_-] ).\end{equation}

The first point to note is that instead of the group $G$, it will be convenient to work with the opposite group $G^{\rm{op}}$, which we can view heuristically as the group  of Poisson automorphisms of the twisted torus $\bT_-$. There is an obvious anti-homomorphism of groups $G\to G^{\rm{op}}$ which is the identity on the underlying sets. Composing \eqref{wawa} with this gives a map
\begin{equation}
\label{wawa3}\Psi_r\colon \bH_r\to G^{\rm {op}} =\Aut_{\{-,-\}}(\bT_-),\end{equation}
which satisfies the analogous conditions to those of Problem \ref{FrobRH}, but with the order of all products reversed. In particular,  Stokes automorphisms will act on the left rather than the right.

A second issue is that, rather than considering the infinite-dimensional group $G^{\rm{op}}$, we prefer to fix an element $\xi\in \bT_-$,  and  consider the composite of \eqref{wawa3} with the evaluation map
\[\operatorname{ev}_\xi\colon G^{\rm {op}} =\Aut_{\{-,-\}}(\bT_-)\to \bT_-, \qquad \operatorname{ev}_\xi(\Psi)= \Psi(\xi).\]
The maps $X_r(\hbar)$ considered in Problem \ref{dtsect} are then given by
\begin{equation}
\label{kar}X_r(\hbar)=\Psi_r(\hbar)(\xi).\end{equation}

Finally, although we would initially expect the resulting maps \eqref{kar} to be holomorphic functions of $
\hbar\in \bH_r$, computations in examples  show that in general we must allow them to be meromorphic. This has the unfortunate effect of making the solutions of Problem \ref{dtsect} highly non-unique. It would be very interesting to understand whether there are some further natural  conditions which can be imposed on the  solutions to control this lack of uniqueness.

\begin{remark}
One important feature of the heuristic picture involving maps \eqref{wawa3} has been lost in the translation to the more concrete Problem \ref{dtsect}, namely that the automorphism  $\Psi_r(\hbar)$ should preserve the Poisson structure on the torus $\bT_-$. As we explain in the next section, if we choose a basis $(\gamma_1,\cdots,\gamma_n)$ for the lattice $\Gamma$ we can encode the solutions to Problem \ref{dtsect} in terms of functions \[x_j(z_i,\theta_i,\hbar)=\log X_{r,\gamma_j}(\hbar), \qquad Z(\gamma_i)=z_i, \qquad \xi(\gamma_i)=\sigma(\gamma_i)\cdot \exp(\theta_i),\]
for some quadratic refinement $\sigma\colon \Gamma\to \{\pm 1\}$ of the form $\<-,-\>$. The missing condition is then that for fixed  $z_i\in \bC$ and $\hbar\in \bC^*$, the map $\theta_i\to x_j$ should be Poisson, so that
\[\sum_{p,q} \eta_{pq} \cdot \frac{\partial x_i}{\partial \theta_p}\cdot \frac{\partial x_j}{\partial \theta_q}=\eta_{ij}.\]
Clearly, imposing this condition will reduce the indeterminacy in the solutions to Problem \ref{dtsect}  as functions of the constant term $\xi\in \bT_-$.
\end{remark}

\begin{remark}In what follows we will not have  much to say about the behaviour of the flat sections of the connection \eqref{biggy} near the point $\hbar=\infty$ where it has a logarithmic singularity. This seems to be an interesting point which deserves further study. The corresponding picture for  Frobenius manifolds  is explained in detail by Dubrovin \cite[Section 2]{D2}, and leads to  the full set of monodromy data for a Frobenius manifold: as well as the Stokes factors at $\hbar=0$, there is monodromy data at $\hbar=\infty$, and a central connection matrix which intertwines them. It would be interesting to understand  the analogue of this extra data in the context of Joyce structures.\end{remark}

\subsection{Relating BPS structures and  Joyce structures}
\label{recon}

In Section \ref{vfs} we explained that a tame Frobenius structure on a complex manifold $M$ determines a variation of FS structures on $M$. We similarly expect that a Joyce structure on $M$, possibly satisfying some additional assumptions, should determine a variation of  BPS structures over $M$, as defined in 
Section \ref{vbpss}. 
 To prove such a result one would need an  analogue in the context of Joyce structures of the fundamental existence result, Theorem \ref{bjl}, which gives canonical flat sections of the  connection \eqref{biggy} in suitable half-planes.
 
What we shall do instead in what follows is to start with a variation of BPS structures on a complex manifold $M$, and attempt to construct a  Joyce structure by solving the RH problem of Section \ref{rh_intro}. Given the lack of uniqueness for solutions discussed above, this will necessarily involve a certain amount of guesswork. A more rigorous mathematical result can be obtained by  reading the following calculations backwards, and viewing them as proving instances of the missing analogue of Theorem \ref{bjl}. In this approach we suppose given the Joyce structure, and  view the solutions to the RH problem as giving the required existence of canonical sections of the connection \eqref{biggy}.

Let us describe in more detail how to get from a miniversal variation of BPS structures $(\Gamma_p,Z_p,\Omega_p)$ on a complex manifold $M$  to a Joyce structure, assuming the existence of suitably varying solutions to Problem \ref{dtsect}. Note first that the miniversal condition ensures that the derivative of the period map \eqref{per} at a point $p\in M$ is an isomorphism
\begin{equation}\label{bodmach}(D\varpi)_p\colon \cT_p M \to \Hom_\bZ(\Gamma_p,\bC).\end{equation}
Using these identifications, the local system of lattices $\Gamma_p$ induces a flat, torsion-free connection on the bundle $\cT_M$, with a covariantly constant bundle of full rank sublattices $\Gamma_M\subset \cT_M^*$.  Moreover, the skew-symmetric form $\<-,-\>$ on the lattice $\Gamma_p$  induces a skew-symmetric form $\eta(-,-)$  on the bundle $\cT_M^*$. We define the Euler vector field $E\colon \O_M\to \cT_M$ by the requirement that $(D\varpi)_p(E_p)=Z_p$ at each point $p\in M$. 

It remains to define the  Joyce function $J$.
Let us choose a covariantly constant family of quadratic refinements  $\sigma_p\colon \Gamma_p\to \{\pm 1\}$ for all $p\in M$. As discussed in Section \ref{qr} this may only be possible after passing to a cover or an open subset of $M$.  The associated maps \eqref{bloxy2} then allow us to identify the two tori $\bT_{p,\pm}$ at each point $p\in M$.  
Suppose we can solve the associated RH problems for generic choices of constant terms $\xi_p\in \bT_{p,-}$. We can then view the solution as a collection of maps
\eqref{wawa} which are the analogues of the canonical solutions of Theorem \ref{bjl}, and which should  therefore, as in Remark \ref{shorts},   define sections of the deformed Joyce connection \eqref{strike}. Differentiating then gives a formula for the Joyce function. There is one confusing sign which arises from the fact that, as explained in Section \ref{indianrock}, we have chosen to work with the opposite group $G^{\rm{op}}$.

To carry this out in detail, let us work locally on $M$, and take a covariantly constant basis $(\gamma_1,\cdots,\gamma_n)$ for the lattices $\Gamma_p$. We can then identify all the lattices $\Gamma_p$, tori $\bT_{p,\pm}$, quadratic refinements $\sigma_p$, {etc.}, and hence drop the point $p\in M$ from the notation. 
Problem \ref{dtsect}  depends on a point $p\in M$ specified locally by the corresponding central charges $z_i=Z(\gamma_i)$, and a constant term $\xi\in \bT_-$ specified by writing $\xi(\gamma_i)=\sigma(\gamma_i)\cdot \exp(\theta_i)$. For a given choice of non-active ray $r$, the solution $X_{r}\colon \bH_r\to \bT_-$  is then encoded in the functions \begin{equation}
\label{sool}x_j(z_i,\theta_i,\hbar)=\log X_{r, \gamma_j}(\hbar).\end{equation}
Note that if we use the quadratic refinement $\sigma$ to view the solution $X_{r}\colon \bH_r\to \bT_-$ as taking values in the torus $\bT_+$, instead of the twisted torus $\bT_-$, the corresponding logarithmic solutions $y_j(z_i,\theta_i,\hbar)$ will differ from $x_j(z_i,\theta_i,\hbar)$ only by constant multiples of $\pi i$, which will vanish upon differentiation and hence play no role in the calculation.

Substituting the solution \eqref{sool} into the equation for the extended Joyce connection \eqref{matrix} gives
\begin{equation}
\label{master}\frac{\partial x_j}{\partial z_i}+\frac{1}{\hbar} \cdot \frac{\partial x_j}{\partial \theta_i}+\sum_{p,q} \eta_{pq}   \cdot \frac{\partial^2 J}{\partial \theta_i\, \partial \theta_p}\cdot \frac{\partial x_j}{\partial \theta_q}=0,\end{equation}
from which we can extract the second derivatives of the Joyce function as meromorphic functions of $z_i$ and $\theta_i$. 
Note that  $J$ precisely measures the failure of the solution of the RH problem to be a function of the combined variables $z_i-\hbar \theta_{i}$.

\subsection{Doubling construction}
\label{sdouble}

Let us return to the setting of Section \ref{la} and the definition of the Lie algebra $\fg$ associated to a lattice $\Gamma\isom \bZ^{\oplus n}$ equipped with an integral skew-symmetric form $\eta$. When the form $\eta$ is degenerate, the Lie algebra homomorphism $\rho\colon \fg\to \operatorname{vect}_{\eta}(\bT_+)$ of equation \eqref{rho} has a nonzero kernel. As explained in Remark \ref{andrea}, this means that the geometric incarnation of the deformed Joyce connection \eqref{strike} contains less information than the abstract $\fg$-valued version \eqref{shrike}. Similar remarks apply to a BPS structure $(\Gamma,Z,\Omega)$ with a degenerate intersection form: viewing  the wall-crossing automorphisms $\bS(\ell)$  as (partially-defined) automorphisms of the torus $\bT_-$, rather than as elements of the abstract group associated to the Lie algebra $\fg$, loses information. In the extreme case when $\eta=0$, the connection \eqref{strike} is almost trivial, every wall-crossing automorphisms $\bS(\ell)$ is the identity, and Problem \ref{dtsect}   is also uninteresting.

In principle these  issues can be dealt with by working directly with the abstract Lie algebra $\fg$ and the associated Lie group. But since these are really only the correct objects at a heuristic level, this would be difficult to make precise. Instead it is easier to  reduce to the case when $\eta$ is non-degenerate by applying a standard doubling construction \cite[Section 2.8]{RHDT}, \cite[Section 2.6]{KS} which we explain here. The reader can also refer to Section \ref{thatone} below, where an explicit example is worked out in detail. 

On the monodromy side, suppose given 
a BPS structure $(\Gamma, Z, \Omega)$. The \emph{doubled BPS structure} then takes the form
\[(\Gamma\oplus \Gamma^\dual, Z\oplus Z^\vee, \Omega),\]
where $\Gamma^\dual=\Hom_\bZ(\Gamma,\bZ)$ is the dual lattice. We equip the doubled lattice $=\Gamma\oplus \Gamma^\dual$ with the non-degenerate skew-symmetric form
\begin{equation}
\label{job}\big\<(\gamma_1,\lambda_1), (\gamma_2,\lambda_2)\big\>=\<\gamma_1,\gamma_2\>+\lambda_1(\gamma_2)-\lambda_2(\gamma_1).\end{equation}
The central charge is defined by
$Z(\gamma,\lambda)=Z(\gamma)+Z^\vee(\lambda),$ for some arbitrary group homomorphism $Z^\vee\colon \Gamma^\vee\to \bC$,
and the BPS invariants by
\[\Omega(\gamma,\lambda)=\begin{cases} \Omega(\gamma) &\text{ if } \lambda= 0, \\ 0&\text{ otherwise}.\end{cases}\]
Given a miniversal variation of BPS structures over a manifold $M$, the derivative of the period map \eqref{per} at a point $p\in M$ gives an identification between the tangent space $\cT_{M,p}$ and the vector space $\Gamma_p^{\vee}\tensor \bC$. We leave it to the reader to prove that the  doubled BPS structures then fit together to give a miniversal variation of BPS structures over the total space of the cotangent bundle $P=\cT^*_{M}$.

On the Joyce structure side, consider a Joyce structure on a complex manifold $M$ given by the data of Definition \ref{joycestructure}. The base of the doubled structure will again be the  complex manifold $P$ which is the total space of the cotangent bundle $\cT^*_M$. The projection $q\colon P\to M$ induces a projection $q_*\colon \cT_{P}\to \cT_M$, and the Joyce function on  $\cT_{P}$ is just the pullback of the Joyce function $J\colon \cT_M\to \bC$ along this map. The tangent space at a point $p\in P$ is \[\cT_p P=\cT_m M\oplus \cT_m^* M, \qquad m=q(p).\]
 We take the lattice $\Gamma_m\oplus \Gamma^*_m\subset \cT_p^* P$, equipped with the form \eqref{job},  and as  connection on $\cT_P$ the sum of the pullbacks along $q$ of the connections induced by $\nabla$ on the bundles $\cT_M$ and  $\cT_M^*$. The Euler vector field at a point $p\in P$ is the sum  $E_m\oplus V_p$, where $E_m$ is the Euler vector  in the fibre $\cT_{M,m}$, and $V_p$ is the 
tautological element of $\cT_m^* M$ defined by the point $p\in P=\cT_{M}^*$. Once again, we leave the reader to check that these definitions give rise to a Joyce structure on the total space  $P=\cT_M$.


\section{Linear data from Joyce structures}
\label{linear}

The  Joyce connection \eqref{nons} defines an Ehresmann connection on the tangent bundle of a complex manifold $M$ equipped with a Joyce structure. Joyce pointed out \cite[Section 4]{HolGen} that this induces a linear connection on the tangent bundle, which we shall call the linear Joyce connection. We explain the definition of this connection here, together with some closely-related tensors: the Joyce form, and an operator $V$. In the special case when both sides of the equation \eqref{fl} individually vanish we also define a holomorphic function  $\cF\colon M\to \bC$ which we call the prepotential. We finish by defining a notion of compatibility for Joyce and Frobenius structures with the same underlying complex manifold.

\subsection{Linear Joyce connection}
\label{prepot}

Consider a Joyce structure as in Section \ref{joyce}, and assume that the Joyce function $J$ extends to a holomorphic function in a neighbourhood of the zero-section of $\cT_M$. The assumption that $J$ is an odd function of the fibre co-ordinates means that for any vector field $X$ on $M$, the Hamiltonian vector field $\Ham_{m_X(J)}$ on $\cT_M$ vanishes along the zero-section. The derivative of  this vector field therefore defines an endomorphism of the normal bundle of the zero-section, which can be identified with the tangent bundle to $M$ via the map $m$ of Section \ref{geom}. The result is an endomorphism-valued 1-form on the tangent bundle of $M$
\[C_X(Y)= D\Ham_{m_X(J)}(Y)\]
which we can use to define  a linear connection on $\cT_M$
\[\nabla^{J}_X(Y)=\nabla_X(Y)-D\Ham_{m_X(J)}(Y).\]
We call this the \emph{linear Joyce connection}.
 Taking co-ordinates as in Section \ref{geom} we have
\begin{equation}
\label{st}\nabla^J_{\frac{\partial}{\partial z_i}}\Big(\frac{\partial}{\partial z_j}\Big)=  -\sum_{l,m}\eta_{lm}\cdot\frac{\partial^3 J}{\partial \theta_i \, \partial \theta_j \, \partial \theta_l}\Big|_{\theta=0} \cdot \frac{\partial}{\partial z_m}.\end{equation}

\begin{lemma}
\label{bla}
The  linear Joyce connection is flat and torsion-free, and preserves the skew-symmetric form $\eta$ on the cotangent bundle $\cT_M^*$.
\end{lemma}

\begin{proof}
These properties can   be easily derived from the above geometric description of the linear Joyce connection, but we shall instead give a computational proof based on the formula  \eqref{st}. For flatness we must show that
\[0=\nabla^J_{\frac{\partial}{\partial z_i}}\circ \nabla^J_{\frac{\partial}{\partial z_j}}\left(\frac{\partial}{\partial z_k}\right)- \nabla^J_{\frac{\partial}{\partial z_j}}\circ \nabla^J_{\frac{\partial}{\partial z_i}}\left(\frac{\partial}{\partial z_k}\right)\]\[=\sum_{l,m}\eta_{lm}\cdot\bigg(-\frac{\partial^4 J}{\partial z_i\, \partial \theta_j \, \partial \theta_k \, \partial \theta_l}\Big|_{\theta=0} +\frac{\partial^4 J}{\partial z_j\, \partial \theta_i \, \partial \theta_k \, \partial \theta_l}\Big|_{\theta=0}\bigg) \cdot \frac{\partial}{\partial z_m} \]
\[+\sum_{l,m,p,q}\eta_{lm}\eta_{pq}\cdot \bigg(\frac{\partial^3 J}{\partial \theta_i \, \partial \theta_q \, \partial \theta_l}\Big|_{\theta=0} \cdot \frac{\partial^3 J}{\partial \theta_j \, \partial \theta_k \, \partial \theta_p}\Big|_{\theta=0}-\frac{\partial^3 J}{\partial \theta_j \, \partial \theta_q \, \partial \theta_l}\Big|_{\theta=0}\cdot \frac{\partial^3 J}{\partial \theta_i \, \partial \theta_k \, \partial \theta_p}\Big|_{\theta=0}\bigg)\cdot \frac{\partial}{\partial z_m}.\]
Since $J$ is an odd function of the fibre co-ordinates, this follows by differentiating \eqref{fl} with respect to $\theta_k$ and $\theta_l$, and setting all $\theta_i=0$.

The  linear Joyce connection induces a  connection on the cotangent bundle $\cT_M^*$ given by
\[\nabla^J_{\frac{\partial}{\partial z_i}}\left(dz_j\right)=\sum_{k,l} \eta_{lj} \cdot \frac{\partial^3 J}{\partial \theta_i \, \partial \theta_l \, \partial \theta_k}\Big|_{\theta=0}\cdot dz_k.\]
The preservation of the form $\eta$ is then a consequence of
\[\eta\Big( \nabla^J_{\frac{\partial}{\partial z_i}}\left(dz_j\right),d z_k\Big)=\sum_{l,m} \eta_{lj}\cdot \eta_{mk} \cdot \frac{\partial^3 J}{\partial \theta_i \, \partial \theta_l \, \partial \theta_m}\Big|_{\theta=0}=-\eta\Big( dz_j, \nabla^J_{\frac{\partial}{\partial z_i}}\left(dz_k\right)\Big).\]
The statement that the linear Joyce connection is torsion-free is immediate from  \eqref{st}. 
\end{proof}

Consider the special case when both sides of the equation \eqref{fl} vanish. This happens for example when the form $\eta=0$. Differentiating
the left-hand side  gives 
\[\frac{\partial^4 J}{\partial z_i\, \partial \theta_j \, \partial \theta_k \, \partial \theta_l}\Big|_{\theta=0} =\frac{\partial^4 J}{\partial z_j\, \partial \theta_i \, \partial \theta_k \, \partial \theta_l}\Big|_{\theta=0}. \]
This implies the existence of a locally-defined function $\cF\colon M\to \bC$, well-defined  up to the addition  of polynomials of degree $\leq 2$ in the variables $z_i$, satisfying
\[\frac{\partial^3 \cF}{\partial z_i\, \partial z_j \, \partial z_k }=\frac{\partial^3 J}{\partial \theta_i\, \partial \theta_j \, \partial \theta_k}\Big|_{\theta=0}.\]

In the notation of Section \ref{joyce}, the function $\cF$ satisfies the relations
\[X_1 X_2 X_3 (\cF)=m_{X_1} m_{X_2} m_{X_3}(J)|_{\theta=0},\]
where $X_1,X_2,X_3$ are arbitrary holomorphic vector fields on $M$.
We call a   locally-defined function $\cF\colon M\to \bC$ 
 satisfying these conditions a \emph{prepotential.} 

\subsection{Joyce form}

Consider again a Joyce structure on a complex manifold $M$ whose Joyce function extends holomorphically over the zero-section. Recall from \eqref{gola} the definition of the function $H=m_E(J)$. We can use this function, and the formula 
\begin{equation}
\label{joyceg}g\Big(\frac{\partial}{\partial z_j},\frac{\partial}{\partial z_k}\Big)= \frac{\partial^2 H}{\partial \theta_j \, \partial \theta_k}\Big|_{\theta=0}= \sum_i z_i \cdot \frac{\partial^3 J}{\partial \theta_i \partial \theta_j \partial \theta_k }\Big|_{\theta=0}\end{equation}
 to define a symmetric bilinear form  on the tangent bundle $\cT_M$ which we call  the \emph{Joyce form}.

\begin{lemma}
\label{blA2}
The Joyce form is preserved by the linear Joyce connection.
\end{lemma}

\begin{proof}
We must check that
\[\frac{\partial}{\partial z_i}\cdot g \bigg(\frac{\partial}{\partial z_j}, \frac{\partial}{\partial z_k}\bigg)=g\bigg(\nabla^J_{\frac{\partial}{\partial z_i}}\Big(\frac{\partial}{\partial z_j}\Big), \frac{\partial}{\partial z_k}\bigg)+g\bigg(\frac{\partial}{\partial z_j}, \nabla^J_{\frac{\partial}{\partial z_i}}\Big(\frac{\partial}{\partial z_k}\Big)\bigg).\]
This boils down to the identity
\[
 \frac{\partial^3 H}{\partial z_i \partial \theta_j \partial \theta_k }\Big|_{\theta=0} =-\sum_{p,q} \eta_{pq} \bigg(  \frac{\partial^3 J}{\partial \theta_i \partial \theta_j \partial \theta_p }\Big|_{\theta=0}\cdot \frac{\partial^2 H}{\partial \theta_k \partial \theta_q }\Big|_{\theta=0}+\frac{\partial^3 J}{\partial \theta_i \partial \theta_k \partial \theta_p }\Big|_{\theta=0}\cdot  \frac{\partial^2 H}{\partial \theta_j\partial \theta_q  }\Big|_{\theta=0} \bigg).\]
To establish this, we sum the relation \eqref{fl}, and use the assumption \eqref{gola} that $J$ is homogeneous of weight $-1$ in the $z_i$ variables to obtain
\[\frac{\partial H}{\partial z_i} = -\sum_{p,q} \eta_{pq} \cdot \frac{\partial^2 J}{\partial \theta_i \partial \theta_p }\cdot \frac{\partial H}{\partial \theta_q  }.\]
The identity  then follows by differentiating  with respect to  $\theta_j$ and $\theta_k$, setting all fibre co-ordinates $\theta_p=0$, and using the fact that $J$ is an odd function of these co-ordinates.
\end{proof}

When the Joyce form $g$ is non-degenerate, it defines a holomorphic metric on $M$, and Lemmas \ref{bla} and \ref{blA2} then imply that the linear Joyce connection $\nabla^J$ is the corresponding  Levi-Civita connection, and that $g$  is a flat metric.

\subsection{Gradient of the Euler vector field}
\label{grad}
Let us continue with the set-up of the last two sections.  
The formula
\begin{equation}
\label{defv2}V(X)=\nabla^J_X(E)-X\end{equation}
defines an endomorphism  of the tangent bundle $\cT_M$. 
Note the similarity with the definition  \eqref{defv}. Explicitly we have
\begin{equation}\label{joycev}V\Big(\frac{\partial}{\partial z_i}\Big)= -\sum_{j,p,q} z_j \cdot \eta_{pq} \cdot \frac{\partial^3 J}{\partial \theta_i \partial \theta_j \partial \theta_p}\Big|_{\theta=0}\cdot \frac{\partial}{\partial z_q}= -\sum_{p,q} \eta_{pq}\cdot \frac{\partial^2 H}{\partial \theta_i \, \partial \theta_p}\Big|_{\theta=0}  \cdot \frac{\partial}{\partial z_q}.\end{equation}
The following result shows that this operator intertwines  the Joyce form $g$ with the skew-symmetric form $\eta$.

\begin{lemma}
\label{quit}
The operator $V$ is covariantly constant with respect to the linear Joyce connection, and skew-adjoint with respect to the Joyce form. There is a commutative diagram
\begin{equation}
\label{diagr}\xymatrix{ T \ar[rr]^{-V}\ar[dr]_{X\mapsto g(X,-)} &&T \\ &T^* \ar[ur]_{\xi\mapsto \eta(\xi,-)}}\end{equation}
\end{lemma}

\begin{proof}
The commutativity of the  diagram \eqref{diagr} follows by comparing the expressions \eqref{joyceg} and \eqref{joycev}. The fact that $V$ is covariantly constant with respect to the linear Joyce connection 
is then immediate from  Lemmas \ref{bla} and  \ref{blA2}. For skew-adjointness note that
\[g\bigg( V\Big(\frac{\partial}{\partial z_i}\Big), \frac{\partial}{\partial z_j}\bigg)=-\sum_{p,q} \eta_{pq} \cdot \frac{\partial^2 H}{\partial \theta_i \, \partial \theta_p}\Big|_{\theta=0}\cdot \frac{\partial^2 H}{\partial \theta_j \, \partial \theta_q}\Big|_{\theta=0},\]
which  clearly changes sign when $i$ and $j$ are exchanged.
\end{proof}

\begin{remark}
\label{lakes}
Suppose given a miniversal variation of BPS structures on a complex manifold $M$, and consider the conjectural construction of the associated Joyce structure from Section \ref{recon}. In particular recall that the derivative of the period map \eqref{bodmach} gives an identification between the tangent space $\cT_p M$ and the vector space $\Hom_{\bZ}(\Gamma_p,\bC)$, and that the skew-symmetric form $\<-,-\>_p$ on the lattice $\Gamma_p$ then induces a skew-symmetric form $\eta_p$ on the tangent space $\cT_p M$.
In terms of the Fourier expansion \eqref{flu2} we can formally write the linear Joyce connection and the Joyce form as
\[\nabla^J_X(Y)=\nabla_X(Y)-\sum_{\gamma\in \Gamma\setminus\{0\}}H(\gamma) \cdot  \frac{ X(\gamma) Y(\gamma)}{Z(\gamma)} \cdot  \<\gamma,-\>,\]
\[g(X,Y)=\sum_{\gamma\in \Gamma\setminus\{0\}} H(\gamma)\cdot X(\gamma)  Y(\gamma).\]
These expressions  were first written down by Joyce, who also gave formal proofs of Lemma \ref{bla} and Lemma \ref{blA2} in these terms. 
\end{remark}

\subsection{Compatibility of Joyce and Frobenius structures}
\label{compatible}

There is clear evidence of non-trivial links between Frobenius structures and spaces of stability conditions \cite{Ludmil, TB, BQS, IQ,IQ2}.   The linear data derived from  Joyce connections appears to be relevant to this story, although it is perhaps too early to be sure about this. In this section we discuss a compatibility relation for Frobenius and Joyce structures having the same underlying complex manifold, which seems to capture a non-trivial feature of several interesting examples.

Consider again a Joyce structure on a complex manifold $M$ whose Joyce function extends holomorphically over the zero-section. Suppose also that the Joyce form $g$ is non-degenerate. 
We can then define a  bilinear operation $\diamond\colon \cT_M\times \cT_M\to \cT_M$,  which we call the \emph{diamond product},   by writing
\begin{equation}
\label{specu}
g\bigg(\frac{\partial}{\partial z_i} \diamond \frac{\partial}{\partial z_j}, \frac{\partial}{\partial z_k}\bigg)=\frac{\partial^3 J}{\partial \theta_i \partial \theta_j \partial \theta_k }\Big|_{\theta=0}= g\bigg(\frac{\partial}{\partial z_i},\frac{\partial}{\partial z_j}\diamond\frac{\partial}{\partial z_k}\bigg).
\end{equation}
 We say that the Joyce structure satisfies the \emph{Frobenius condition} if the diamond product   is associative. This is not the case in general (see Section \ref{gengen} below), but it does seem to hold in a number of interesting examples arising from DT theory. Note that since
\[g(E\diamond X_1,X_2)=m_{X_1} m_{X_2} m_E (J)|_{\theta=0}= m_{X_1} m_{X_2}(H)=g(X_1,X_2),\]
the Euler vector field $E$ defines a unit  for the operation $\diamond$.
Thus when the Frobenius condition is satisfied, the pair $(g,\diamond)$ induces the structure of a  commutative and unital Frobenius algebra on each tangent space $\cT_{M,p}$. Note that in the special case when we can define a prepotential $\cF\colon M\to \bC$ we can write
\[g(X_1\diamond X_2,X_3)=X_1X_2X_3(\cF)=g(X_1, X_2\diamond X_3).\]

Let  us call a Frobenius structure \emph{discriminant-free} if the operator $U(X)=E*X$  is everywhere  invertible. Given such a Frobenius structure we can define  a \emph{twisted multiplication} $\diamond\colon \cT_M\times \cT_M\to \cT_M$  by writing
\begin{equation}
\label{dim}E*(X \diamond Y) = X*Y.\end{equation}
It follows immediately from this definition that $\diamond$ is commutative and associative, and that the Euler vector field defines a  unit.

\begin{defn}
\label{blog}
Suppose that a complex manifold $M$ is equipped with both a discriminant-free Frobenius structure, and a Joyce structure  whose Joyce function extends holomorphically over   the zero-section. We call the two structures \emph{compatible} if the   following conditions hold:
\begin{itemize}
\item[(i)] the Joyce form coincides with the metric of the Frobenius structure up to scale:
\[g(X,Y)_{Joyce}=\lambda\cdot  g(X,Y)_{Frob}, \qquad \lambda\in \bC^*;\]
\item[(ii)] the diamond multiplication of the Joyce structure coincides with the twisted multiplication of the Frobenius structure up to scale:\begin{equation}
\label{mu}X\diamond_{Joyce} Y = \mu\cdot X\diamond_{Frob} Y, \qquad \mu\in \bC^*.\end{equation}
\end{itemize}
\end{defn}

We shall see some examples where these conditions  hold below.

\begin{remark}An intriguing question is whether there exists some natural construction which, given  a tame Frobenius structure, produces   a compatible Joyce structure with the same underlying manifold. \end{remark} 

\subsection{Consequences of compatibility}The following result lists  some simple consequences of the compatibility relation of Definition \ref{blog}. Part (e) involves the \emph{odd periods} of a Frobenius manifold \cite[Definition  8]{D3}, which are defined to be the flat co-ordinates for the second structure connection of \cite[Section 9.2]{Hertling} with parameters $s=-\tfrac{1}{2}$ and $z=0$.

\begin{prop}
\label{coco}
Suppose given a discriminant-free Frobenius structure of conformal dimension $d$ on a complex manifold $M$, and a compatible Joyce structure. Then $d\neq 2$, and the following conditions hold:
\begin{itemize}
\item[(a)] the scaling factor $\mu$ in \eqref{mu} is  $\mu=(2-d)/2$, \smallskip
\item[(b)] the Euler vector fields of the Joyce and Frobenius  structures coincide up to scale:
\[E_{Joyce}= \mu^{-1}\cdot E_{Frob};\]
\item[(c)] the operator $V$ defined in Section \ref{grad} coincides up to scale with the operator $V$ of the Frobenius manifold defined by \eqref{defv}:
\[V_{Joyce}= \mu^{-1}\cdot V_{Frob};\]
\item[(d)] the linear Joyce connection coincides with the  Levi-Civita connection of the Frobenius structure;\smallskip
\item[(e)] the flat co-ordinates for the connection $\nabla$ of the Joyce structure are the odd periods of the Frobenius structure.
\end{itemize}
\end{prop}

\begin{proof}
Since the Euler vector fields $E_{Joyce}$ and $E_{Frob}$ are the identities for the  products $\diamond_{Joyce}$ and  $\diamond_{Frob}$ respectively, it follows that $E_{Joyce}=\mu^{-1}\cdot E_{Frob}$, where $\mu$ is the constant of \eqref{mu}. Consider the relations
\[\mathcal{L}ie_{E_{Joyce}}(g_{Joyce})=2g_{Joyce}, \qquad \mathcal{L}ie_{E_{Frob}}(g_{Frob})=(2-d) \cdot g_{Frob}.\]
The first follows from the   condition \eqref{homo}, which shows that the function $H$ is homogeneous of degree $0$ in the variables $z_i$,   and the second is the definition of the conformal dimension $d$. 
Comparing these equations shows that $\mu=(2-d)/2$, which proves parts (a) and (b).

Definition \ref{blog}(i) ensures that the Joyce form is non-degenerate, so that the linear Joyce connection $\nabla^J$ is the associated Levi-Civita connection. This gives part (d). Part (c) then follows by comparing \eqref{defv} and \eqref{defv2}.
 Finally, note that using the definition \eqref{specu} of the diamond product, together with the diagram \eqref{diagr}, we can rewrite \eqref{st} as
\[\nabla^{J}_X(Y)=\nabla_X(Y)+V_{Joyce}(X\diamond_{Joyce} Y).\]
Part (e) then follows from the observation that this expression coincides with the difference between the second structure connection and the Levi-Civita connection given in \cite[Proposition 3.3]{D3} or \cite[Definition 9.3]{Hertling}.\end{proof}

\begin{remark}Given a compatible pair as in Definition \ref{blog}, we can attempt to reconstruct the Frobenius structure from the Joyce structure. The main missing ingredient is the identity vector field $e\colon \O_M\to \cT_M$. It follows from the axioms of a Frobenius manifold that $e$  is an eigenvector of the operator $V_{Frob}$ with eigenvalue $-d/2$, and is flat for the Levi-Civita connection. If we are given $e$, the formula \eqref{dim} shows that $U(X)= X\diamond_{Frob} e$, and we can then reconstruct the product of the Frobenius structure by writing  $X* Y = U(X\diamond_{Frob} Y)$. For more details see the discussions in \cite[Proposition 3.5]{D1}, \cite[Lemma 3.3]{D2} and \cite[Section 4.3]{Hitchin}.
\end{remark}

\begin{remark}
\label{exotic}
It will be convenient to extend  Definition \ref{blog} to include slight variations on the notion of a Frobenius structure known as \emph{almost Frobenius structures} \cite{D3}. These structures have the same basic ingredients (i)--(iii) appearing in Definition \ref{frobenius}, and satisfy all the axioms except (F2). Replacing this axiom there is a different condition on a vector field $e\colon \O_M\to \cT_M$ which will play no role here. See  \cite[Definition  9]{D3} for details.
 Since  an almost Frobenius structure has a metric and a product, we can extend Definition \ref{blog} in the obvious way to include this case. In fact there are some special features of almost Frobenius manifolds which simplify the results of Proposition \ref{coco}. Firstly, the Euler vector field  is the unit of the multiplication, so the non-degeneracy condition is automatic, and the diamond multiplication \eqref{dim} coincides with the product $*$. Secondly, the gradient of the Euler vector field is a multiple of the identity \cite[Equation (3.23)]{D3}, so the operator $V$ defined by \eqref{defv} vanishes, and it follows that the odd periods  coincide with  the flat co-ordinates.
\end{remark}


\section{The finite uncoupled case}
\label{uncoupled}

In this section we study the RH problems of Section \ref{rh_intro} associated to BPS structures $(\Gamma,Z,\Omega)$ which are  finite, uncoupled and integral. Recall that this means that all BPS invariants $\Omega(\gamma)$ are integers, with only finitely many being nonzero, and
\[\Omega(\gamma_1)\neq 0, \ \Omega(\gamma_2)\neq 0 \implies \<\gamma_1,\gamma_2\>=0.\] These RH problems were first solved by Barbieri \cite{Barb}, after a special case was studied in \cite{RHDT}. We shall follow Barbieri's treatment closely in this section.

\subsection{The A$_1$ BPS structure}
\label{sma}

Let us begin by considering the following basic example of a finite, uncoupled, integral BPS structure, determined by an element $z\in \bC^*$:

\begin{itemize}
\item[(i)] the lattice $\Gamma=\bZ\cdot \gamma$ has rank one, and  thus necessarily $\<-,-\>=0$;
\smallskip
\item[(ii)] the central charge $Z\colon \Gamma\to \bC$ is determined by $Z(\gamma)=z\in \bC^*$;
\smallskip
\item[(iii)] the only non-vanishing BPS invariants are $\Omega(\pm\gamma)=1$. 
\end{itemize}

This family of BPS structures arises  naturally from  the DT theory of the A$_1$ quiver, consisting of   one vertex and no arrows. It is easily checked that they fit together to form a miniversal variation of BPS structures over the manifold $\bC^*$. The local system of lattices is the trivial one $\Gamma\times \bC^*$.

We will always take the unique quadratic refinement $\sigma\colon \Gamma\to \{\pm 1\}$ satisfying $\sigma(\gamma)=-1$, and use the corresponding isomorphism $\rho_{\sigma}$ of \eqref{bloxy2} to relate the two tori $\bT_{\pm}$. Explicitly this means that given a point $\xi\in \bT_-$ we write
\begin{equation}
\label{mar}x_\gamma(\xi)=\exp(\vartheta),\qquad y_\gamma(\rho_{\sigma}^{-1}( \xi))=\exp(\theta), \qquad \theta=\vartheta+\pi i ,\end{equation}
where $y_\gamma\colon \bT_+\to \bC^*$ (respectively $x_\gamma\colon \bT_-\to \bC^*$) denotes the  character (respectively twisted character) corresponding to the element $\gamma\in \Gamma$.

We would like to construct a Joyce function whose Stokes data is given by this variation of BPS structures. As explained in Section \ref{recon}, most features of this structure are automatic. The underlying manifold is $\bC^*$, with the flat connection $\nabla$ on the tangent bundle $\cT_M$ being the one in which the standard co-ordinate $z\in \bC^*$ is flat. The bundle of lattices $\Gamma_M\subset \cT_M^*$ is spanned  by $dz$, and the form $\eta=0$.  The Euler vector field is $E=z\cdot \frac{\partial}{\partial z}$. What remains is to construct the Joyce function $J$ on the space  $\cT_M$ as a function of the co-ordinates $(z,\theta)$.

Consider the formula \eqref{jojo} in this case.
Since the Lie algebra $\god$ is abelian, and the expression on the right is a sum of commutators,  all terms with $n>1$  cancel. Consulting \cite{VTL} we find that
\[H(\gamma)=-F_1(Z(\gamma))\cdot \DT(\gamma)=\frac{-1}{2\pi i}\cdot \DT(\gamma),\]
for all $\gamma\in \Gamma$. We therefore obtain from the expression \eqref{flu}, and the formula \eqref{bps} for the DT invariants 
\begin{equation}
\label{pass}H(z,\theta)=\frac{-1}{2\pi i} \cdot \sum_{n\geq 1}   
\frac{1}{n^2}\left(e^{n \vartheta}+ e^{-n \vartheta}\right)=\frac{1}{2\pi i}\cdot \left(\frac{\theta^2}{2}+\frac{\pi^2}{6}\right),\end{equation}
where we used the Fourier series identity
\begin{equation*}
\sum_{n\geq 1}\frac{\cos(nt)}{n^2}=\frac{(t+\pi)^2}{4}-\frac{\pi^2}{12},\end{equation*}
valid for $-\pi<\Re(t+\pi)<\pi$.
The defining relation \eqref{gola} for the function $H$ then gives the Joyce function
\[J(z,\theta)=\frac{1}{2\pi i}\cdot \frac{\theta^3}{6z},\]
where we neglect functions linear in $\theta$.

Note that we made a slightly arbitrary choice  to sum the series \eqref{pass} under the assumption $-\pi<\Im(\theta)<\pi$. If we choose a different interval we will change the Joyce function by a  function quadratic in $\theta$. But the choice we made is in a sense the right one, since it results in a  Joyce function which is an odd function of  $\theta$.

\subsection{Riemann-Hilbert problem in the doubled A$_1$ case}
\label{thatone}
Let us now attempt to calculate the Joyce function for the same BPS structures using the RH approach.
In fact the RH problems associated to the A$_1$ BPS structures are trivial since the form $\<-,-\>$ vanishes and so all wall-crossing automorphisms are the identity.  Following the discussion in Section \ref{sdouble} we instead consider the RH problems associated to the doubles  of these structures \begin{equation}\label{cheedale}(\Gamma_D,Z\oplus Z^\vee,\Omega), \quad \Gamma_D=\Gamma\oplus\Gamma^\dual.\end{equation}
Let $\gamma^\dual\in \Gamma^\dual$ be the unique generator satisfying $\gamma^\dual(\gamma)=1$. Note that  the definition \eqref{job} of the symplectic form $\<-,-\>$ on the doubled lattice $\Gamma_D$ implies that $\<\gamma^\dual,\gamma\>=1$.  We set $z^\vee=Z^\vee(\gamma^\vee)\in \bC$. We also extend the quadratic refinement $\sigma\colon \Gamma\to \{\pm 1\}$ of Section \ref{sma} to the doubled lattice  $\Gamma_D$ by setting $\sigma(\gamma^\vee)=1$.

To define the RH problem for the doubled BPS structure we must first choose a constant term $\xi_D\in \bT_{D,-}$, where $\bT_{D,-}$ is the twisted torus defined by the lattice $\Gamma_D$ and the form $\<-,-\>$.
We set
\[\xi_D(\gamma)=\exp(\vartheta), \qquad \xi_D(\gamma^\vee)=\exp(\vartheta^\vee).\]
The only active rays  are $\ell_\pm=\pm \bR_{>0}\cdot z$.
 Since the functions $X_{\rr,\gamma}(\hbar)$ have no discontinuities it is natural to take   
\begin{equation}\label{cali}X_{\rr,\gamma}(\hbar)=X(\hbar)=\exp\Big(\vartheta-\frac{z}{\hbar}\Big),\end{equation}
for any non-active ray $\rr\subset \bC^*$. If we were looking for holomorphic solutions to the RH problem this would be the only possibility \cite[Lemma  4.8]{RHDT}, but in the meromorphic context   there are of course other choices. This is discussed further in  Section \ref{unique} below. Let us also write
\[X_{\rr,\gamma^\vee}(\hbar)=\exp\Big(\vartheta^\dual-\frac{z^\dual}{\hbar}\Big)\cdot R_{\rr}(\hbar).\]

The functions $R_{\rr}(\hbar)$ for non-active rays $\rr\subset \bC^*$ lying in the same component of $\bC^*\setminus (\ell_-\cup \ell_+)$ are necessarily analytic continuations of each other. Thus we obtain just two meromorphic functions
$R_\pm$,
 corresponding to the non-active rays lying in the half-planes $\pm \Im(z/\hbar)>0$. 
The RH problems for the doubled BPS structures can then be restated as follows.

\begin{problem}
\label{eps}
 For fixed $z\in \bC^*$ and $\vartheta\in \bC$, find meromorphic functions \[R_\pm\colon \bC^*\setminus i\ell_\pm \to \bC\]
 with the following properties:
\begin{itemize}
\item[(i)]
there are relations
\begin{equation}
\label{rabbit}R_+(\hbar)=\begin{cases} R_-(\hbar) \cdot \left(1-X(\hbar)\right)^{-1}&\mbox{ if }\hbar\in \bH_{\ell_+},\\ R_-(\hbar) \cdot \left(1-X(\hbar)^{-1}\right)^{-1}&\mbox{ if }\hbar\in \bH_{\ell_-},\end{cases}\end{equation}
where $x(\hbar)=\exp(\vartheta-z/\hbar)$ is as in \eqref{cali};\smallskip

\item[(ii)] as $\hbar\to 0$ in a closed subsector of $\bC^*\setminus i\ell_{\pm}$ we have $R_{\pm}(\hbar)\to 1$;\smallskip

\item[(iii)] there exists $k>0$ such that
\[|\hbar|^{-k} < |R_{\pm}(\hbar)|<|\hbar|^k\]
as $\hbar\to \infty$ in a closed subsector of $\bC^*\setminus i\ell_{\pm}$.
\end{itemize}
\end{problem}

To understand condition (i) note that if $\hbar\in \bH_{\ell_+}$, then $\hbar$ lies in the domains of definition $\bH_{\rr_i}$ of the functions $X_{\rr_i,\gamma^\vee}$ corresponding to sufficiently small deformations $\rr_i$ of the ray $\ell_+$. Thus \eqref{bir} applies to the ray $\ell_+$ and we obtain the first of  the relations \eqref{rabbit}. The second follows similarly from \eqref{bir} applied to the opposite ray $\ell_-$.

\subsection{Solution in the doubled A$_1$ case}

The formula
\[\Lambda(w,\eta)=\frac{e^{w} \cdot\Gamma(w+\eta)}{\sqrt{2\pi}\cdot w^{w+\eta-\half}}\]
defines a  meromorphic  function of $w\in \bC^*$ and $\eta\in \bC$, which is multi-valued due to the   factor
\begin{equation}
\label{har}w^{w+\eta-\half}=\exp({(w+\eta-\half)\log w}).\end{equation}
We specify it uniquely for $w\in \bC^*\setminus\bR_{<0}$ by taking the principal branch of $\log$.
Since $\Gamma(w)$ is meromorphic on $\bC$ with poles only at the non-positive integers, it follows that
$\Lambda(w,\eta)$ has poles only at the points $w+\eta\in \bZ_{\leq 0}$. 

The Stirling expansion \cite[Section 12.33]{WW} gives an asymptotic expansion
\begin{equation}\label{stirling}\log \Lambda(w,\eta)
\sim  \sum_{k=2}^{\infty} \frac{(-1)^{k}\cdot B_{k}(\eta)}{k(k-1)}w^{1-k},
\end{equation}
where $B_k(\eta)$ denotes the $k$th Bernoulli polynomial. This expansion is valid as $w\to \infty$ in the complement of any closed sector containing the ray  $\bR_{<0}$. It implies in particular that for fixed $\eta\in \bC$ we have $\Lambda(w,\eta)\to 1$.

\begin{prop}
\label{gs}
Problem \ref{eps} has a  solution, namely 
\[R_\pm (\hbar)= \Lambda\bigg( \frac{\pm z}{2\pi i\hbar},\frac{ \pi i \mp \theta}{2\pi i}\bigg)^{\pm 1},\]
where we set $\theta=\vartheta+\pi i $ as in \eqref{mar}.
\end{prop}

\begin{proof}
Note that the function $R_+(\hbar)$ is well-defined and meromorphic on the domain $\bC^*\setminus i\ell_+$  as required, because if we set $w=z/2\pi i\hbar$ then \[\hbar\in i\ell_+ \iff w\in \bR_{<0}.\]
Similarly $R_-(\hbar)$ is well-defined and meromorphic on $\bC^*\setminus i\ell_-$.
The
Euler reflection formula  states that
\[\Gamma(x)\cdot \Gamma(1-x) =\frac{\pi}{\sin(\pi x)},\qquad x\in \bC\setminus\bZ.\]
Combined with the formula $\Gamma(n)=(n-1)!$ valid for $n\in \bZ_{>0}$, it follows that  $\Gamma(x)$ is nowhere vanishing. The same is therefore true of  the function $\Lambda(w,\eta)$.
The reflection formula also implies that for $w\in \bC^*\setminus \bR_{<0}$
\[\Lambda(w,\eta)\cdot \Lambda(-w,1-\eta)=\frac{1}{2\pi}\cdot  \frac{\pi }{\sin(\pi (w+\eta))}\cdot e^{-( (w+\eta-\half)\log(w)+(-w-\eta+\half)\log(-w))}.\]
The principal branch of $\log(w)$ on the domain $\bC^*\setminus \bR_{<0}$ satisfies \[\log(w)=\log(-w)\pm \pi i\text{ when }\pm \Im(w)>0,\]
so we conclude that when $\pm \Im(w)>0$
\begin{equation}
\label{luminy}\Lambda(w,\eta)\cdot \Lambda(-w,1-\eta)=\frac{i\cdot e^{\mp \pi i(w+\eta-\half)}} {e^{\pi i(w+\eta)}-e^{-i\pi (w+\eta)}}=(1-e^{\pm 2\pi i(w+\eta)})^{-1}.\end{equation}

Let us again set $w=z/2\pi i \hbar$. This implies that\[ \hbar\in \bH_{\ell_\mp}\iff \pm \Im(w)>0.\]
Writing $\eta=(\pi i-\theta)/2\pi i=-\vartheta/2\pi i$, and recalling \eqref{cali},   the relation \eqref{luminy} becomes
\[R_+(\hbar)\cdot R_-(\hbar)^{-1}=\big(1-X(\hbar)^{\mp 1}\big)^{-1},\]
when $\hbar\in \bH_{\ell_{\mp}}$, which is equivalent to the relations \eqref{rabbit}.

Condition (ii) in the Problem \ref{eps} is immediate from the Stirling expansion \eqref{stirling}. Condition (iii) is easily checked: for fixed $\eta$ the term \eqref{har}  clearly contributes at worst polynomial growth as $w\to 0$, and since   $\Gamma(w)$ is meromorphic at $w=0$, the same is true of $\Lambda(w,\eta)$.
\end{proof}

In the special case $\vartheta=0$ it follows from the identity $\Lambda(w,0)=\Lambda(w,1)$ that the solution of Proposition \ref{gs} reproduces the one given in \cite[Section 5]{RHDT}.

\subsection{Computing the Joyce function}
\label{cjf}

To compute the Joyce function  from the solution of Proposition \ref{gs} we use the formula
\begin{equation}
\label{master2}\frac{\partial x_j}{\partial z_i}+\frac{1}{\hbar} \cdot \frac{\partial x_j}{\partial \theta_i}+\sum_{p,q} \eta_{pq}   \cdot \frac{\partial^2 J}{\partial \theta_i\, \partial \theta_p}\cdot \frac{\partial x_j}{\partial \theta_q}=0\end{equation}
derived in Section \ref{recon}. 
Plugging the trivial part of the solution
\[x_{\pm,\gamma}(\hbar)=\log X_{\pm,\gamma}(\hbar)=\vartheta-\frac{z}{\hbar},\]
 into this equation immediately gives
\[\frac{\partial^2 J}{\partial \theta^\vee \partial \theta^\vee}= 0=\frac{\partial^2 J}{\partial \theta \partial \theta^\vee}.\]
Since we 
neglect linear terms in the variables $\theta$, $\theta^\vee$, we can assume that $J$ is independent of $\theta^\vee$.

Note that this fits perfectly with discussion of Section \ref{sdouble}: we should interpret the solution of the RH problem for  the doubled BPS structure   \eqref{cheedale} as describing the double of the Joyce structure corresponding to the original BPS structure $(\Gamma, Z, \Omega)$. It is therefore to be expected that the Joyce function is independent of the doubling variables $z^\vee$ and $\theta^\vee$.

The non-trivial part of our solution to the RH problem is
\[x_{\pm,\gamma^\vee}(\hbar)=\log X_{\pm,\gamma^\vee}(\hbar) =\vartheta^\vee-\frac{z^\vee}{\hbar}\pm \log \Lambda\Big( \frac{\pm z}{2\pi i\hbar},\frac{\pi i \mp \theta}{2\pi i}\Big).\]
Plugging this into \eqref{master2} gives
\[\frac{\partial^2 J}{\partial \theta^2}=\pm \Big(\frac{\partial }{\partial z} + \frac{1}{\hbar} \cdot \frac{\partial }{\partial \theta}\Big) \log \Lambda\Big( \frac{\pm z}{2\pi i  \hbar}, \frac{\pi i \mp\theta}{2\pi i}\Big)= \frac{1}{2\pi i \hbar} \Big( \frac{\partial }{\partial w}-\frac{\partial }{\partial \eta} \Big)\log \Lambda(w,\eta),\]
where we set $w=\pm z/2\pi i \hbar$ and $\eta=(\pi i\mp \theta)/2\pi i$. From the definition of $\Lambda(w,\eta)$ we obtain
\[\frac{\partial^2 J}{\partial \theta^2}=\frac{1}{2\pi i \hbar}\Big(1-\frac{w+\eta-\frac{1}{2}}{w}\Big)=\frac{\theta}{2\pi i z}.\]
Since the Joyce function is an odd function of $\theta$,  we conclude that up to linear terms in $\theta$
\begin{equation}
\label{windows}J(z,\theta)=\frac{\theta^3}{12\pi i z},\end{equation}
in agreement with what we found in Section \ref{sma}.

\subsection{Remarks on uniqueness}
\label{unique}
The solution of Proposition \ref{gs} is not unique with the stated properties.
More precisely, it is unique only up to multiplication of both $R_\pm(\hbar)$ by a meromorphic function $f\colon \bP^1\to \bP^1$ satisfying $f(0)=1$. We can further constrain the solution by insisting that the  zeroes and poles of $R_\pm(\hbar)$  lie only at the points
\begin{equation}
\label{pn}p_n=\frac{z}{\vartheta+2\pi i n},\qquad n\in \bZ,\end{equation}
where the factors appearing in \eqref{rabbit} have zeroes or poles.
The solution of Proposition \ref{gs} is then unique with this property up to multiplication of both $R_\pm(\hbar)$ by factors of the form
\[\bigg(1-\frac{ \hbar(\vartheta+2\pi i n)}{z}\bigg)^{\pm 1}.\]
Note that inserting such factors adds only quadratic functions in $\theta$ to the Joyce function, since
\[\bigg( \frac{\partial }{\partial z}+ \frac{1}{\hbar}\cdot \frac{\partial }{\partial \theta}\Big)\log \Big(1-\frac{ \hbar(\vartheta+2\pi i n)}{z}\Big)=-\frac{1}{z}.\]

The particular solution of Proposition \ref{gs} is distinguished by the property that $R_+(\hbar)$ is holomorphic and nonzero at the point $p_n$ for all $n> 0$. But note that the labelling of the $p_n$ depends on a choice of $\vartheta=\log \xi_D(\gamma)$, and hence cannot be chosen consistently as the constant term $\xi_D\in \bT_{D,-}$ varies.

\begin{remark}
\label{thth}
The  RH problem of  \cite[Section 5]{RHDT} asks for functions $R_\pm(\hbar)$  exactly as in Problem \ref{eps}, but which are holomorphic and everywhere non-vanishing  on their domains. It is easy to see  \cite[Lemma 4.9]{RHDT} that if a solution to this more constrained problem exists then it is unique.
But since  the factors appearing in \eqref{rabbit} have zeroes or poles at the points \eqref{pn}, such  a solution   is only possible if all the points \eqref{pn} lie on the rays $i \ell_{\pm}$, which is only the case if $\vartheta\in i\bR$. Given this,  it is easy to see that the unique holomorphic, non-vanishing solution is the one given by Proposition \ref{gs}, with $\theta$ chosen to satisfy $-\pi< \Im(\theta)\leq \pi$.
\end{remark}

\subsection{Compatibility with the A$_1$ Frobenius manifold}

Let us now consider the linear data of  Section \ref{linear} in the case of the   Joyce structure of Section \ref{sma}. 
Since the form $\eta=0$,
 the Joyce connection coincides with the standard connection $\nabla$, and hence  $z\in \bC^*$ is a flat co-ordinate for this connection. The gradient of the Euler vector field $E=z\frac{\partial}{\partial z}$ is then the identity, which implies that the operator $V$ vanishes. The Joyce form is
\[g\Big(\frac{\partial}{\partial z},\frac{\partial}{\partial z}\Big)=\frac{\partial^2 H}{\partial \theta^2}\Big|_{\theta=0}=\frac{1}{2\pi i},\]
which in particular is non-degenerate.
The diamond multiplication \eqref{specu} is then
\[\frac{\partial}{\partial z} \diamond \frac{\partial}{\partial z}= \frac{1}{z} \cdot \frac{\partial}{\partial z} .\]
A possible choice of  prepotential is
\[\cF(z)=\frac{1}{4\pi i}\cdot z^2 \log(z).\]

Let us compare these structures with the  trivial one-dimensional Frobenius manifold \cite[Example I.1.5]{D1}. This has flat co-ordinate $t\in \bC$ and is defined by the data
\[g\Big(\frac{\partial}{\partial t},\frac{\partial}{\partial t}\Big)=1, \qquad \frac{\partial}{\partial t}*\frac{\partial}{\partial t}=\frac{\partial}{\partial t}, \qquad E=t\cdot \frac{\partial}{\partial t}.\]
Since $\mathcal{L}ie_E(g)=2g$, the  conformal dimension  is $d=0$. The discriminant-free condition  holds for $t\in \bC^*$, and the corresponding twisted multiplication is
 \[\frac{\partial}{\partial t}\diamond \frac{\partial}{\partial t}=\frac{1}{t}\cdot \frac{\partial}{\partial t}.\]
 The conditions of Definition \ref{blog} clearly hold, so this Frobenius structure is compatible with the  Joyce structure of Section \ref{sma}. Note that  in this case $V=0$, which implies  that the odd periods  coincide with the flat co-ordinates. 
  
 It is also worth considering the almost dual of the above trivial Frobenius structure. The underlying manifold is $\bC^*$, and the basic ingredients  \cite[Example 1]{D3} are
 \[g\Big(\frac{\partial}{\partial t},\frac{\partial}{\partial t}\Big)=\frac{1}{t}, \qquad \frac{\partial}{\partial t}*\frac{\partial}{\partial t}=\frac{1}{t}\cdot \frac{\partial}{\partial t}, \qquad E=t\cdot \frac{\partial}{\partial t}.\]
 In terms of the flat co-ordinate $x=2\sqrt{t}$ this takes the form
  \[g\Big(\frac{\partial}{\partial x},\frac{\partial}{\partial x}\Big)=1, \qquad \frac{\partial}{\partial x}*\frac{\partial}{\partial x}=\frac{2}{x} \cdot \frac{\partial}{\partial x}, \qquad E=\frac{x}{2}\cdot \frac{\partial}{\partial x}.\]
  As explained in Remark \ref{exotic}, the Euler vector field is the identity for the multiplication in this almost dual structure, so the diamond product   coincides with  the ordinary product. It follows that, in the extended sense of Remark \ref{exotic}, this almost Frobenius structure is also compatible with the Joyce structure of Section \ref{sma}.
  
\subsection{The general finite uncoupled case}
\label{gengen}

Consider a miniversal variation of  finite, integral and uncoupled BPS structures $(\Gamma_p,Z_p,\Omega_p)$ over a complex manifold $M$. Let us attempt to follow the conjectural construction of the associated Joyce structure from Section \ref{recon}. Working locally, we can replace $M$ by a contractible   open subset, and use the local system in the definition  of a variation of BPS structures to identify all the lattices $\Gamma_p$ with a fixed lattice $\Gamma$.  We can then  specify a point of $M$ by the corresponding central charge $Z\colon \Gamma\to \bC$. The miniversal assumption ensures that the derivative of the period map \eqref{bodmach}
is an isomorphism. Thus  tangent vectors  at a point $p\in M$ can be identified with  homomorphisms $\theta\colon \Gamma\to \bC$.

The uncoupled condition ensures that the BPS automorphisms associated to different rays commute, which means that we can solve the associated RH problems by superposing solutions to the A$_1$ RH problem. This is explained in detail in \cite[Section 5.3]{RHDT}. There is no canonical choice of quadratic refinement, but this is unimportant since different choices give the same  Joyce function up to functions which are at most quadratic in the fibre co-ordinates $\theta_i$.  Adding copies of \eqref{windows}, and noting that the A$_1$ case has two nonzero BPS invariants $\Omega(\pm \gamma)$, it follows that the  Joyce function  in the general case can be taken to be
\begin{equation}\label{cubic}J(Z,\theta)=\frac{1}{24\pi i }\cdot\sum_{\gamma\in \Gamma\setminus\{0\}}  \Omega(\gamma)\cdot \frac{\theta(\gamma)^3}{Z(\gamma)},\end{equation}
up to linear terms in $\theta$.
This gives the  Joyce form
\[g(X_1,X_2)=\frac{1}{4\pi i}\cdot \sum_{\gamma\in \Gamma\setminus\{0\}} \Omega(\gamma)\cdot X_1(\gamma)X_2(\gamma),\]
which is therefore constant.
The linear Joyce connection is 
\[\nabla^J_{X}(Y)=\nabla_X(Y)-\frac{1}{4\pi i}\cdot \sum_{\gamma\in \Gamma\setminus\{0\}}  \Omega(\gamma)\cdot  \frac{ X(\gamma)Y(\gamma)}{Z(\gamma)} \cdot \<\gamma,-\>,\]
where $\nabla$ denotes the flat connection  resulting from identifying all tangent spaces $\cT_p M$  with the fixed vector space $\Hom_{\bZ}(\Gamma,\bC)$. 
A possible prepotential $\cF\colon M\to \bC$   is \[\cF(Z)=\frac{1}{8\pi i }\cdot \sum_{\gamma\in \Gamma\setminus\{0\}} \Omega(\gamma) \cdot Z(\gamma)^2 \log Z(\gamma).\]

When the Joyce form is non-degenerate the diamond product \eqref{specu} is defined by
\[g(X_1\diamond X_2,X_3)=\frac{1}{4\pi i}\cdot \sum_{\gamma\in \Gamma\setminus\{0\}}  \Omega(\gamma)\cdot \frac{X_1(\gamma)X_2(\gamma)X_3(\gamma)}{Z(\gamma)}=g(X_1, X_2\diamond X_3).\] 
Recall that the Joyce structure is said to satisfy the Frobenius condition if   this operation is associative. Since
\[ g(X_1\diamond X_2,X_3)=X_1 X_2 X_3 (\cF)=g(X_1, X_2\diamond X_3)\]
this is equivalent to the condition that the function $\cF$ satisfies the WDVV equations with respect to the metric $g$. In the special case when all $\Omega(\gamma)\in \{0,1\}$ this happens precisely when the set of active classes $\gamma\in \Gamma$ form a $\vee$-system in the sense of Veselov (see \cite[Theorem 1]{veselov}). It would be very interesting to understand whether there is some natural class of  CY$_3$ triangulated categories whose associated DT theory always leads to Joyce structures satisfying the Frobenius condition.

\section{The case of Calabi-Yau threefolds without compact divisors}
\label{snine}

In this section we consider the  BPS structures arising from DT theory applied to compactly-supported coherent sheaves on quasi-projective Calabi-Yau threefolds which contain no compact divisors. The relevant sheaves are then supported in dimension $\leq 1$, and the resulting DT invariants are conjecturally determined by the genus 0 Gopakumar-Vafa invariants. Assuming this conjecture we can use the methods of \cite{Con} to solve the associated RH problems, and explicitly describe the corresponding Joyce structures.

\subsection{Conifold case}

In the finite, uncoupled situation discussed in the previous section it turned out to be enough to understand a single basic example, namely the A$_1$ BPS structure, all other cases being obtained by superposing copies of this one. In the geometric set-up discussed in this section a similar role is played by the BPS structure defined by the DT theory of compactly-supported coherent sheaves on the resolved conifold. This non-compact Calabi-Yau threefold $X$ is the total space of the vector bundle $\O_{\bP^1}(-1)^{\oplus 2}$. Contracting the zero-section $C\subset X$ gives the threefold ordinary double point $(x_1x_2-x_3x_4=0)\subset \bC^4$, also known as the conifold singulatity. 

The variation of BPS structures  we shall consider lives on the complex manifold
\[M_+=\big\{(v,w)\in \bC^2: w\neq 0 \text{ and }\Im(v/w)>0\}\subset \bC^2.\]
In fact, this variation extends to a larger space, which has a natural interpretation as the quotient of the space of stability conditions on the  compactly-supported derived category of coherent sheaves of $X$ by a certain subgroup of the auto-equivalence group. This extended variation will not be important in what follows however, so we omit the details and instead refer the interested reader to \cite{Con}.

The BPS structure $(\Gamma_{\leq 1},Z_{\leq 1},\Omega)$ corresponding to a point $(v,w)\in M_+$ is determined  by the following data:

\begin{itemize}
\item[(i)]
the lattice $\Gamma_{\leq 1}=\bZ\beta\oplus \bZ\delta$ with the form $\<-,-\>=0$;\smallskip

\item[(ii)]the central charge $Z_{\leq 1}\colon \Gamma_{\leq 1}\to \bC$ defined by
\[Z_{\leq 1}(a\beta+b\delta)=av+bw;\]

\item[(iii)]
the non-zero BPS invariants
\[\Omega(\gamma)=\begin{cases} 1 &\text{if }\gamma=\pm \beta+n\delta \text{ with } n\in \bZ,\\ 
-2 &\text{if }\gamma=k\delta\text{ with }k\in \bZ\setminus\{0\}.\end{cases}\]
\end{itemize}

 In  terms of the DT theory of the resolved conifold $X$, the relevant sheaves are the line bundles $\O_C(n)$ supported on the unique compact curve $C\subset X$, and the sheaves supported in dimension zero. The lattice $\Gamma_{\leq 1}$ can be identified with the group $H_2(X,\bZ)\oplus H_0(X,\bZ)$ which is the natural receptacle for Chern characters of compactly-supported sheaves on $X$. The Chern character of $\O_C(n)$ then corresponds to $\beta+n\delta$, and that of a skyscraper sheaf $\O_x$ to $\delta$.

The above BPS structures  form a miniversal variation of uncoupled BPS structures over $M_+$. The BPS invariants are constant because  the form $\<-,-\>$ vanishes. Let us define the rays
\[\ell_\infty=\bR_{>0}\cdot w,\qquad \ell_n=\bR_{>0}\cdot (v+nw)\subset \bC^*.\]
Then the active rays for the  BPS structure  $(\Gamma_{\leq 1},Z_{\leq 1},\Omega)$  are precisely the rays $\pm \ell_\infty$ and $\pm \ell_n$ for $n\in \bZ$, and the associated ray diagram is illustrated in Figure \ref{fig}. We denote by $\Sigma(n)\subset \bC^*$  the convex sector with boundary rays $\ell_{n-1}$ and $\ell_n$. 

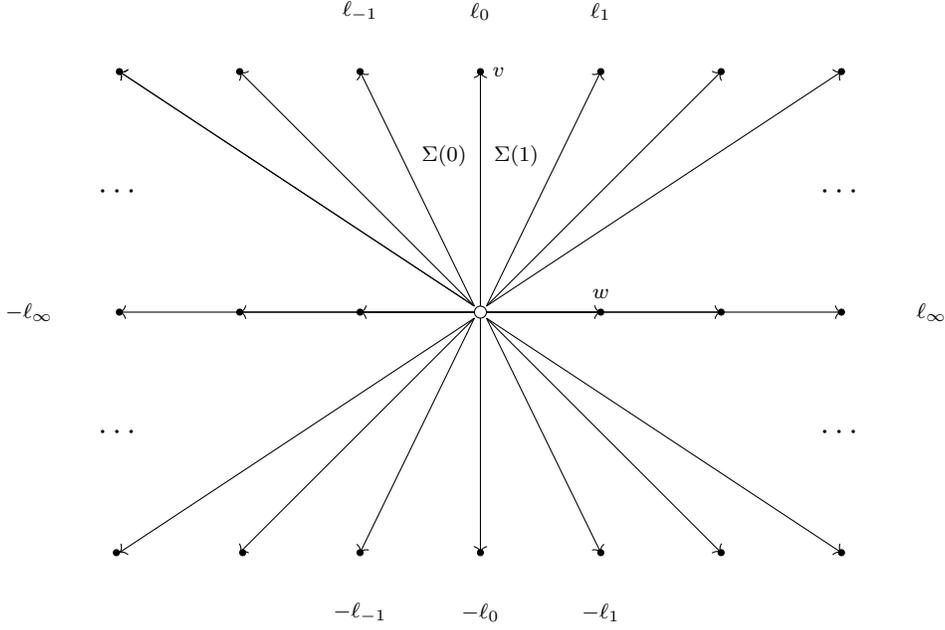
\begin{figure}
\begin{tikzpicture}[scale=0.8]
\draw (5,0) circle [radius=0.1];
\draw[->] (4.9,0.1) -- (-1,4);
\draw[->] (4.9,0.1) -- (1,4);
\draw[->] (4.9,0.1) -- (3,4);
\draw[->] (5,0.1) -- (5,4);
\draw[->] (5.1,0.1) -- (7,4);
\draw[->] (5.1,0.1) -- (9,4);
\draw[->] (5.1,0.1) -- (11,4);
\draw[->] (4.9,0.1) -- (-1,4);
\draw[fill] (-1,4) circle [radius=0.05];
\draw[fill] (1,4) circle [radius=0.05];
\draw[fill] (3,4) circle [radius=0.05];
\draw[fill] (5,4) circle [radius=0.05];
\draw[fill] (7,4) circle [radius=0.05];
\draw[fill] (9,4) circle [radius=0.05];
\draw[fill] (11,4) circle [radius=0.05];
\draw[->] (4.9,-0.1) -- (-1,-4);
\draw[->] (4.9,-0.1) -- (1,-4);
\draw[->] (4.9,-0.1) -- (3,-4);
\draw[->] (5,-0.1) -- (5,-4);
\draw[->] (5.1,-0.1) -- (7,-4);
\draw[->] (5.1,-0.1) -- (9,-4);
\draw[->] (5.1,-0.1) -- (11,-4);
\draw[fill] (-1.05,-4) circle [radius=0.05];
\draw[fill] (1.05,-4) circle [radius=0.05];
\draw[fill] (3,-4) circle [radius=0.05];
\draw[fill] (5,-4) circle [radius=0.05];
\draw[fill] (7,-4) circle [radius=0.05];
\draw[fill] (9,-4) circle [radius=0.05];
\draw[fill] (11,-4) circle [radius=0.05];
\draw[->] (4.9,0) -- (3,0); \draw[->] (4.9,0) -- (1,0); \draw[->] (4.9,0) -- (-1,0);
\draw[->] (5.1,0) -- (7,0); \draw[->] (5.1,0) -- (9,0); \draw[->] (5.1,0) -- (11,0);
\draw[fill] (3,0) circle [radius=0.05];
\draw[fill] (1,0) circle [radius=0.05];
\draw[fill] (-1,0) circle [radius=0.05];
\draw[fill] (7,0) circle [radius=0.05];
\draw[fill] (9,0) circle [radius=0.05];
\draw[fill] (11,0) circle [radius=0.05];
\draw (11,2) node { $\cdots$};
\draw (-1,2) node { $\cdots$};
\draw (11,-2) node { $\cdots$};
\draw (-1,-2) node { $\cdots$};
\draw (5,5) node {$\scriptstyle \ell_0$};
\draw (3,5) node {$\scriptstyle \ell_{-1}$};
\draw (7,5) node {$\scriptstyle \ell_{1}$};
\draw (-2.5,0) node {$\scriptstyle -\ell_{\infty}$};
\draw (12.5,0) node {$\scriptstyle \ell_{\infty}$};
\draw (5,-5) node {$\scriptstyle -\ell_0$};
\draw (3,-5) node {$\scriptstyle -\ell_{-1}$};
\draw (7,-5) node {$\scriptstyle -\ell_{1}$};
\draw (4.4,2.6) node {$\scriptstyle \Sigma(0)$};
\draw (5.6,2.6) node {$\scriptstyle \Sigma(1)$};
\draw (7,0.3) node {$\scriptstyle w$};
\draw (5.3,4) node {$\scriptstyle v$};
\end{tikzpicture}
\caption{The ray diagram associated to a point $(v,w)\in M_+$.
\label{fig}}
\end{figure}

As in Section \ref{uncoupled} we shall also consider the  doubles of these  BPS structures. We denote these by $(\Gamma,Z,\Omega)$ and use the notation
\[\Gamma=\Gamma_{\leq 1} \oplus \Gamma_{\geq 2}, \qquad \Gamma_{\geq 2}:=\Gamma_{\leq 1}^\vee=\Hom_\bZ(\Gamma_{\leq 1},\bZ).\]
We denote by $(\beta^\vee,\delta^\vee)\subset \Gamma_{\geq 2}$ the dual basis to $(\beta,\delta)\subset \Gamma_{\leq 1}$. 
The map $\Omega\colon \Gamma\to \bQ$ satisfies $\Omega(\gamma)=0$ unless $\gamma\in \Gamma_{\leq 1}$. The central charge takes the form \[Z=Z_{\leq 1}\oplus Z_{\geq 2}\colon \Gamma \to \bC,\] where the group homomorphism $Z_{\geq 2}\colon \Gamma_{\geq 2}\to \bC$
is arbitrary. We set
\[Z(\beta^\vee)=v^\vee, \qquad Z(\delta^\vee)=w^\vee.\]
In this way we obtain a  miniversal variation of  uncoupled  and integral BPS structures  over the cotangent bundle $\cT^*M_+$. These BPS structures are not finite, but are easily seen to be  convergent in the sense of Definition \ref{blahblah}.

The RH problem associated to the BPS structure $(Z,\Gamma,\Omega)$ depends on a point of the twisted torus $\bT_-$ associated to the doubled lattice $\Gamma$.  We set
\[\xi(\beta)=\exp(\vartheta), \qquad \xi(\delta)=\exp(\varphi), \qquad \xi(\beta^\vee)=\exp(\vartheta^\vee),\qquad \xi(\delta^\vee)=\exp(\varphi^\vee).\]
The solution consists of meromorphic functions
$X_{r,\gamma}\colon \bH_r \to \bC^*$
for all non-active rays $r\subset \bC^*$ and all classes $\gamma\in \Gamma$. The solutions corresponding to different non-active rays contained in the same sector $\Sigma(n)$ are analytic continuations of each other, and hence define a function $X_{n,\gamma}(\hbar)$ on the region 
\[\cV(n)=\bH_{\ell_{n-1}}\cup\bH_{\ell_n}\subset \bC^*.\]
We will not introduce notation for the corresponding functions  for non-active rays $r\subset \bC^*$  lying in the opposite sectors $-\Sigma(n)$ because we will impose the natural symmetry condition
 \begin{equation}
 \label{tor}X^{\sigma(\xi)}_{-r,-\gamma}(-\hbar)=X_{r,\gamma}^{\xi}(\hbar),\end{equation}
 where $\sigma\colon \bT_-\to \bT_-$ is the twisted inverse map of Section \ref{twto}.

As in Section \ref{thatone} it is natural to set
\begin{equation}
\label{trevor}X_{n,\beta}(\hbar)=\exp\left( \vartheta-\frac{v}{\hbar}\right), \qquad X_{n,\delta}(\hbar)=\exp\left( \varphi-\frac{w}{\hbar}\right),\end{equation}
for all integers $n\in \bZ$.
We then write the non-trivial part of the solution in the form
\begin{equation}
\label{nontrev}X_{n,\beta^\vee}(\hbar)=\exp\left( \vartheta^\vee-\frac{v^\vee}{\hbar}\right)\cdot B_n(\hbar), \qquad X_{n,\delta^\vee}(\hbar)=\exp\left( \varphi^\vee-\frac{w^\vee}{\hbar}\right)\cdot D_n(\hbar),\end{equation}
where we drop from the notation the dependence of the functions $B_n$ and $D_n$ on the parameters $(z,w)$ and $(\vartheta,\varphi)$.
Working out the conditions imposed on the functions $B_n$ and $D_n$ we obtain the following explicit version of the RH problem:

\begin{problem}
\label{tra}
 For each integer $n\in \bZ$ find meromorphic functions $B_n(\hbar)$ and $D_n(\hbar)$  defined on the region $\cV(n)$
with the following properties:
\begin{itemize}
\item[(i)] as $\hbar\to 0$ in any closed subsector of $\cV(n)$ one has \[ B_n(\hbar)\to  1, \qquad  D_n(\hbar)\to  1;\]
\item[(ii)] for each $n\in \bZ$ there exists $k>0$ such that for  any closed subsector of $ \cV(n)$
\[ |\hbar| ^{-k} <  |B_n(\hbar)|,  |D_n(\hbar)| < |\hbar|^k, \qquad  |\hbar| \gg 0;\]
\item[(iii)] on the intersection $\cH(n)=\cV(n)\cap \cV(n+1)$ there are relations
\[B_{n+1}(\hbar)=B_{n}(\hbar)\cdot (1-x q^{n} )^{-1}, \quad D_{n+1}(\hbar)=D_{n}(\hbar)\cdot (1-x q^{n})^{-n};\]

\item[(iv)] note that $\cV(0)\cap -\cV(0)=i \cdot \Sigma(0)\sqcup -i \cdot \Sigma(0)$; in the region $-i\cdot \Sigma(0)$  there are relations
\[B_0(\hbar)\cdot B_0(-\hbar)=\prod_{n\geq 0}\big(1-x q^{n}\big)\cdot \prod_{n\geq 1}\big(1-x^{-1} q^{n})^{-1},\]\begin{equation*} D_0(\hbar)\cdot D_0(-\hbar)=\prod_{n\geq 0}\big(1-xq^{n}\big)^{n}\cdot \prod_{n\geq 1}\big(1-x^{-1}  q^{n}\big)^{n}\cdot \prod_{k\geq 1} \big(1-q^{k}\big)^{-2k},\end{equation*}
\end{itemize}
where we used the notation \begin{equation}
\label{notat}x=\exp\left(\vartheta-\frac{v}{ \hbar}\right), \qquad q=\exp\left(\varphi- \frac{w}{\hbar}\right).\end{equation}
\end{problem}

Note that condition (iv) arises by comparing the functions $X_r(\hbar)$ and $X_{-r}(-\hbar)$ for non-active rays $r\subset \Sigma(0)$, using the symmetry relation \eqref{tor}.
In the case when $\vartheta=\varphi=0$, Problem \ref{tra}   coincides with \cite[Problem 3.3]{Con} except that we have rescaled the parameters $(v,w)$ by a factor of $2\pi i$, and more importantly, we are now allowing our solutions $B_n(\hbar)$ and $D_n(\hbar)$ to have zeroes and poles on the regions $\cV(n)$. The argument explained in Remark \ref{thth} proves the necessity of this change: for general choices of $\vartheta$ and $\varphi$ there can be no solution to Problem \ref{tra} in which $B_n(\hbar)$ and $D_n(\hbar)$ are holomorphic and non-vanishing. 

\subsection{Difference equations}

Our variation of BPS structures carries  a 
 natural action of the group $\bZ$, which geometrically corresponds to tensoring with line bundles on $X$. This symmetry allows us to restate the  RH problem of the last subsection as a  system of coupled difference equations. We refer the reader to \cite{Con} for the details of this, the basic upshot being that it is natural to seek a solution 
%
 to Problem \ref{tra} satisfying
\[B_n(v,w,\vartheta,\varphi,\hbar)= B_0(v+nw,w,\vartheta+n\varphi,\varphi,\hbar),\]
\[D_n(v,w,\vartheta,\varphi,v)=D_0(v+nw,w,\vartheta+n\varphi,\varphi,v) \cdot B_0(v+nw,w,\vartheta+n\phi,\varphi,\hbar)^{n}.\]
Assuming this ansatz, we can then restate Problem \ref{tra}  in terms of just two functions $B=B_0$ and $D=D_0$.

\begin{problem}
\label{tre}
Find meromorphic functions  $B(v,w,\vartheta,\varphi,\hbar)$ and $D(v,w,\vartheta,\varphi,\hbar)$  defined for $(v,w)\in M_+$, $(\vartheta,\varphi)\in \bC^2$ and $\hbar\in \cV(0)$, such that for  fixed $(v,w,\vartheta,\varphi)$ the following properties hold:
\begin{itemize}
\item[(i)] one has
\[B(v,w,\vartheta,\varphi,\hbar)\to 1, \qquad  D(v,w,\vartheta,\varphi,\hbar)\to  1,\]
as $\hbar\to 0$ in any closed subsector of $\cV(0)$;\smallskip

\item[(ii)] there  exists $k>0$  such that \[ |\hbar| ^{-k} <  |B(v,w,\vartheta,\varphi,\hbar)|,  |D(v,w,\vartheta,\varphi,\hbar)| < |\hbar|^k,\]
as $\hbar\to \infty$ in  any closed subsector of $ \cV(0)$;\smallskip

\item[(iii)] there are relations
\[{B(v+w,w,\vartheta+\varphi,\varphi,\hbar)}={B(v,w,\vartheta,\varphi,\hbar)}
\cdot(1-x)^{-1}, \]\[{D(v+w,w,\vartheta+\varphi,\varphi,\hbar)}={D(v,w,\vartheta,\varphi,\hbar)}\cdot  B(v+w,w,\vartheta+\varphi,\varphi,\hbar)^{-1},\]
for  $\hbar\in \bC^*$ lying in the  intersection $\cH(0)=\cV(0)\cap\cV(1)$;\smallskip 

\item[(iv)] for $\hbar\in -i\cdot \Sigma(0)$  there are relations
\[B(v,w,\vartheta,\varphi,\hbar)\cdot B(v,w,-\vartheta,-\varphi,-\hbar)=\prod_{n\geq 0}\big(1-x q^{n}\big)\cdot \prod_{n\geq 1}\big(1-x^{-1} q^{n})^{-1},\]\begin{equation*} D(v,w,\vartheta,\varphi,\hbar)\cdot D(v,w,-\vartheta,-\varphi,-\hbar)=\prod_{n\geq 0}\big(1-xq^{n}\big)^{n}\cdot {\prod_{n\geq 1}\big(1-x^{-1}  q^{n}\big)^{n}}\cdot {\prod_{k\geq 1} \big(1-q^{k}\big)^{-2k}},\end{equation*}
\end{itemize}
where we used the notation \eqref{notat} as before.
\end{problem}

\subsection{Special functions}
Recall from \cite[Section 4]{Con} the following two functions defined by the integral representations
\[F(z\b\omega_1,\omega_2)=\exp\bigg(\int_C \frac{e^{zs}}{(e^{\omega_1 s}-1)(e^{\omega_2 s}-1)}\frac{ds}{s}\bigg),\]
 \[G(z\b\omega_1,\omega_2)=\exp\bigg(\int_C \frac{-e^{(z+\omega_1)s}}{(e^{\omega_1 s}-1)^2(e^{\omega_2 s}-1)}\frac{ds}{s}\bigg),\]
valid when $\Re(\omega_i)>0$ and $0<\Re(z)<\Re(\omega_1+\omega_2)$.
 The contour $C$ follows the real axis from $-\infty$ to $+\infty$ avoiding the origin by a small detour in the upper half-plane.

Now consider the combinations
\begin{equation}
\label{b}F^*(v,w,\vartheta,\varphi,\hbar)=F(v-\hbar\vartheta\b w-\hbar\varphi,-2\pi i \hbar)\cdot \exp Q_F(v,w,\vartheta,\varphi,\hbar),\end{equation}
\begin{equation}
\label{c}G^*(v,w,\vartheta,\varphi,\hbar)=G(v-\hbar\vartheta\b w-\hbar\varphi,-2\pi i \hbar)\cdot \exp Q_G(v,w,\vartheta,\varphi,\hbar),\end{equation}
where the exponential factors are 
\begin{equation}
\label{expf}Q_F(v,w,\vartheta,\varphi,\hbar)=\frac{1}{\hbar}\cdot \frac{w-\hbar\varphi}{(2\pi i)^2 } \cdot \Li_2(e^{2\pi i v/w}) + \Big( \frac{v\varphi-w\vartheta}{2\pi i w}+\frac{1}{2}\Big) \Li_1(e^{2\pi i v/w}),\end{equation}
\begin{equation}
\label{expg}Q_G(v,w,\vartheta,\varphi,\hbar)=\frac{2}{\hbar}\cdot \frac{w-\hbar\varphi}{(2\pi i)^3 } \cdot \Li_3(e^{2\pi i v/w}) -\frac{1}{\hbar}\cdot \frac{v-\hbar\vartheta}{(2\pi i)^2} \cdot \Li_2(e^{2\pi i v/w}) \end{equation}\[+\frac{1}{\pi i}  \Big(\frac{v\varphi-w\vartheta}{2\pi i w}+\frac{1}{4}\Big)\Li_2(e^{2\pi i v/w})-\frac{v}{w} \Big(\frac{v\varphi-w\vartheta}{2\pi i w}+\frac{1}{2}\Big)\Li_1(e^{2\pi i v/w}).\]
These expressions are uniquely determined by the following  result.
\begin{lemma}
\label{xmas2}
Consider fixed parameters $(v,w)\in M_+$ and $(\vartheta,\varphi)\in \bC^2$, and assume that $0<\Re(v)<\Re(w)$. Then as $\hbar\to 0$ in the half-plane $\Im(\hbar)>0$  we have
\[F^*(v,w,\vartheta,\varphi,\hbar)\to 1, \qquad G^*(v,w,\vartheta,\varphi,\hbar)\to 1.\]
Moreover, if  $\vartheta$ and $\varphi$ are generic, these functions are bounded as $\hbar\to \infty$ in the same half-plane.
\end{lemma}

\begin{proof}
The first statement  follows from the asymptotic expansions given in \cite[Proposition 4.6]{Con}. As $\hbar\to 0$ in the upper half-plane these  give\[\log F(v\b w,-2\pi i \hbar)\sim -\frac{1}{\hbar}\cdot \frac{w}{(2\pi i)^2} \cdot \Li_2(e^{2\pi i v/w})-\frac{1}{2}\Li_1(e^{2\pi i v/w}) + O(\hbar).\]
\[\log G(v\b w,-2\pi i \hbar)\sim  -\frac{2}{\hbar}\cdot \frac{w}{(2\pi i)^3 }\cdot  \Li_3(e^{2\pi i v/w}) +\frac{1}{\hbar}\cdot \frac{v}{(2\pi i)^2}\cdot \Li_2(e^{2\pi i v/w}) \]\[-\frac{1}{4\pi i} \cdot \Li_2(e^{2\pi i v/w}) +\frac{v}{2w} \cdot \Li_1(e^{2\pi i v/w})+ O(\hbar).\]
Substituting $v\mapsto v-\hbar\vartheta$ and $w\mapsto w-\hbar\varphi$ and noting that
\[\Li_k\Big(e^{ \frac{2\pi i (v-\hbar\vartheta)}{w-\hbar\varphi}}\Big)=\Li_k\left(e^{2\pi i v/w}\right)+ 2\pi i \hbar \Big(\frac{v\varphi-w\vartheta}{w^2}\Big) \Li_{k-1}(e^{2\pi i v/w})+O(\hbar^2)\]
 gives the required relations
\[\log F(v-\hbar\vartheta\b w-\hbar\varphi,-2\pi i \hbar)\sim  -Q_F(v,w,\vartheta,\varphi,\hbar) +O(\hbar),\]
\[\log G(v-\hbar\vartheta\b w-\hbar\varphi,-2\pi i \hbar)\sim  -Q_G(v,w,\vartheta,\varphi,\hbar) +O(\hbar).\]

For the second statement, note that the functions $F$ and $G$ are homogeneous under rescaling all variables, so as $\hbar\to \infty$
\[F(v-\hbar\vartheta\b w-\hbar\varphi,-2\pi i \hbar)\sim F(\vartheta\b \varphi,2\pi i),\]
and similarly for $G$. The claim then follows from the form of the expressions $Q_F$ and $Q_G$.
\end{proof}

\subsection{Solution}
Using a mild extension of the techniques of \cite{Con} we can solve the RH problem for the conifold using the  special functions introduced above. 

\begin{thm}
\label{sol}
Consider fixed parameters $(v,w)\in M_+$ and take $(\vartheta,\varphi)\in \bC^2$ to be generic. Then the functions
\[B(v,w,\vartheta,\varphi,\hbar) =F^*(v,w,\vartheta,\varphi,\hbar),\]
\begin{equation}\label{ozbound} D(v,w,\vartheta,\varphi,\hbar)= G^*(v,w,\vartheta,\varphi,\hbar)\cdot G^*(0,w,0,\varphi,\hbar)^{-1}\end{equation}give a solution to Problem \ref{tre}.\end{thm}

\begin{proof}
After Lemma \ref{xmas2} we must just check the jumping conditions, parts (iii) and (iv). These follow by the same argument in the proof of \cite[Proposition 5.2]{Con} which is the case $\vartheta=\varphi=0$. Note that the exponential factors $Q_F$ and $Q_G$ make no difference in part (iii) because
\begin{equation}
\label{rel1}Q_F(v+w,w,\vartheta+\varphi,\varphi,\hbar)=Q_F(v,w,\vartheta,\varphi,\hbar)\end{equation}
\begin{equation}
\label{rel2}Q_G(v+w,w,\vartheta+\varphi,\varphi,\hbar)-Q_G(v,w,\vartheta,\varphi,\hbar)=-Q_F(v,w,\vartheta,\varphi,\hbar).\end{equation}
The required identities then follow from the corresponding properties of the functions $F$ and $G$ given in \cite[Section 4]{Con}. 
For part (iv) note that the combinations
\[Q_F(v,w,\vartheta,\varphi,\hbar)+Q_F(v,w,-\vartheta,-\varphi,-\hbar)=\Li_1(e^{2\pi i v/w})=-\log(1-e^{2\pi i v/w}),\]
\[Q_G(v,w,\vartheta,\varphi,\hbar)+Q_G(v,w,-\vartheta,-\varphi,-\hbar)=\frac{d}{dw} \left(\frac{w}{2\pi i}\cdot \Li_2(e^{2\pi i v/w})\right),\]
are independent of $(\vartheta,\varphi)$. 
The required relation then  follows from the reflection identities of \cite[Proposition 4.3]{Con} by the same argument as \cite[Proposition 5.2]{Con}.
\end{proof}

Similar remarks to those of Section \ref{unique} apply in the setting of Theorem \ref{sol}. In particular, the given solution  is not unique with the stated properties. 

\subsection{Joyce function}
We choose the quadratic refinement $\sigma\in \bT_-$ defined by
\[\sigma(\beta)=-1, \qquad \sigma(\delta)=1, \qquad \sigma(\beta^\vee)=1, \qquad \sigma(\delta^\vee)=1,\]
and consider the corresponding isomorphism $\rho_\sigma\colon \bT_+\to \bT_-$. 
Setting $\zeta=\rho_{\sigma}^{-1}(\xi)$
we have
\[y_{\beta}(\zeta)=\exp(\theta), \qquad y_{\delta}(\zeta)=\exp(\phi), \qquad y_{\beta^\vee}(\zeta)=\exp(\theta^\vee),\qquad y_{\delta^\vee}(\zeta)=\exp(\phi^\vee),\]
where $\theta=\vartheta-\pi i$ but the other variables are unshifted: $\phi=\varphi$, $\theta^\vee=\vartheta^\vee$ and $\phi^\vee=\varphi^\vee$. Note that these substitutions  produce a simplification in the expressions \eqref{expf} and \eqref{expg}.

 To compute the Joyce function  corresponding to the solution of Theorem \ref{sol} we again use the formula
\begin{equation}
\label{master3}\frac{\partial x_j}{\partial z_i}+\frac{1}{\hbar} \cdot \frac{\partial x_j}{\partial \theta_i}+\sum_{p,q} \eta_{pq}   \cdot \frac{\partial^2 J}{\partial \theta_i\, \partial \theta_p}\cdot \frac{\partial x_j}{\partial \theta_q}=0\end{equation}
derived in Section \ref{recon}. Arguing as in Section \ref{cjf} using the trivial part of the solution \eqref{trevor} shows that the Joyce function can be assumed to be independent of the dual variables $\vartheta^\vee,\varphi^\vee$. 
The non-trivial part \eqref{nontrev} of the solution is
\[x_{n,\beta^\vee}(\hbar)=\vartheta^\vee-\frac{v^\vee}{\hbar}+\log B(v+nw,w,\vartheta+n\varphi,\varphi,\hbar),\]\[ x_{n,\delta^\vee}(\hbar)=\varphi^\vee-\frac{w^\vee}{\hbar}+\log D(v+nw,w,\vartheta+n\varphi,\varphi,\hbar)+n \log B(v+nw,w,\vartheta+n\varphi,\varphi,\hbar).\]

Plugging these solutions into \eqref{master3}  involves applying operators
\[\frac{\partial}{\partial v}+\frac{1}{\hbar}\cdot \frac{\partial}{\partial \theta}, \qquad \frac{\partial}{\partial w}+\frac{1}{\hbar}\cdot \frac{\partial}{\partial \phi}.\]
which kill any functions of the variables $v-\hbar \theta$ , $w-\hbar\phi$ and $\hbar$. It therefore follows from the form of the definitions \eqref{b} and \eqref{c} that we  need only deal with the exponential factors $Q_F$ and $Q_G$.
The second factor in \eqref{ozbound}  makes no contribution because
\[Q_G(0,w,0,\varphi,\hbar)=\frac{2}{\hbar}\cdot \frac{w-\hbar \varphi}{(2\pi i)^3}\cdot \zeta(3)  -\frac{\pi i}{24}.\] Moreover, the relations \eqref{rel1}-\eqref{rel2} show that the result is independent of $n$. Differentiating the expressions \eqref{expf} and \eqref{expg}  now gives
\begin{eqnarray*}\frac{\partial^2 J}{\partial \theta^2}&=&\bigg(\frac{\partial}{\partial v}+\frac{1}{\hbar}\cdot\frac{\partial}{\partial \theta}\bigg) Q_F(v,w,\theta,\phi,\hbar)= \frac{1}{w^2}\cdot (v\phi-w\theta)\cdot \Li_0(e^{2\pi i v/w}),\\
\frac{\partial^2 J}{\partial \theta \partial \phi}&=&\bigg(\frac{\partial}{\partial w}+\frac{1}{\hbar}\cdot\frac{\partial}{\partial \phi}\bigg) Q_F(v,w,\theta,\phi,\hbar)=-\frac{v}{w^3}\cdot (v\phi-w\theta)\cdot \Li_0(e^{2\pi i v/w}),\\
\frac{\partial^2 J}{\partial \phi^2}&=&\bigg(\frac{\partial}{\partial w}+\frac{1}{\hbar}\cdot\frac{\partial}{\partial \phi}\bigg) Q_G(v,w,\theta,\phi,\hbar)=\frac{v^2}{w^4}\cdot (v\phi-w\theta)\cdot \Li_0(e^{2\pi i v/w}).\end{eqnarray*}
It follows that up to linear terms in $\theta, \phi$ the Joyce function is
\[J(v,w,\theta,\phi)=\frac{1}{6w^4} \cdot (v\phi-w\theta)^3\cdot \Li_0(e^{2\pi i v/w})=\frac{1}{6w^4} \cdot (v\phi-w\theta)^3\cdot (1-e^{-2\pi i v/w})^{-1}.\]
Note that  the Joyce form vanishes since
\[H(v,w,\theta,\phi)=v\cdot \frac{\partial J}{\partial \theta}+w \cdot \frac{\partial J}{\partial \phi}=0.\]
Thus we cannot define a diamond product. The prepotential $\cF\colon M\to \bC$  in the sense of Section \ref{prepot}  is given by the formula  
\[\cF(v,w)=-\frac{w^2}{(2\pi i)^3}\cdot \Li_3(e^{2\pi i v/w}),\]
up to quadratic terms in $v,w$. It satisfies the relations
\[\frac{\partial^3 \cF}{\partial v^3}=\frac{\partial^3 J}{\partial \theta^3}, \qquad \frac{\partial^3 \cF}{\partial v^2\, \partial w}=\frac{\partial^3 J}{\partial \theta^2\,\partial \phi}, \qquad \frac{\partial^3 \cF}{\partial v\, \partial w^2}=\frac{\partial^3 J}{\partial \theta\,\partial \phi^2},\qquad \frac{\partial^3 \cF}{\partial w^3}=\frac{\partial^3 J}{\partial \phi^3}.\]
Note that on the locus $w=1$ the prepotential coincides, up to a constant factor, with the genus 0 Gromov-Witten generating function for the resolved conifold.

\subsection{Threefolds without compact divisors}

Let $X$ be a non-compact Calabi-Yau threefold which contains no compact divisors, and introduce the complexified K{\"a}hler cone
\[\mathcal{K}_\bC(X)=\{\omega_\bC=B+i\omega:\text{ $\omega$ ample}\}\subset H^2(X,\bC).\]
The base of our variation of BPS structures  will be  the complex manifold
\[M=\mathcal{K}_{\bC}(X)\times \bC^*.\]
We consider a framed variation of BPS structures whose lattice \[\Gamma_{\leq 1}=H_2(X,\bZ)\oplus\bZ\] is the receptacle for Chern characters of compactly-supported coherent sheaves on $X$. As before, we equip this lattice with the form $\<-,-\>=0$.
The central charge corresponding to a point $(\omega_\bC,w)\in M$ is the map
\[Z_{\leq 1}\colon \Gamma_{\leq 1}\to \bC, \qquad Z_{\leq 1}(\beta,n)=v(\beta)+nw, \qquad v(\beta)=\beta\cdot\omega_{\bC}.\]

The relevant BPS invariants
$\Omega(\gamma)\in \bQ$ were first constructed by Joyce and Song (\cite{JS}, see particularly Sections 6.3--6.4). These invariants  are defined using moduli stacks of semistable objects supported in dimension $\leq 1$, and should not be confused with the ideal sheaf curve-counting invariants appearing in the famous MNOP conjectures \cite{MNOP}.  Joyce and Song prove that the numbers $\Omega(\gamma)$  are independent of the point of $M$. This is to be expected, since wall-crossing is trivial when the form $\<-,-\>$ vanishes.
A direct calculation \cite[Section 6.3]{JS} shows that
 \[\Omega(0,n)=-\chi(X),\qquad n\in \bZ\setminus\{0\},\]
 where $\chi(X)$ denotes the topological Euler characteristic of the manifold $X$.
It is expected \cite[Conjecture 6.20]{JS} that when $\beta>0$ is a positive curve class
\begin{equation}
\label{lausanne}
\Omega(\beta,n)=\GV(0,\beta),\end{equation}
and in particular, is independent of $n$.  Here $\GV(0,\beta)$ is the genus 0 Gopakumar-Vafa invariant for the class $\beta\in H_2(X,\bZ)$.
 We emphasise that the higher genus Gopakumar-Vafa invariants are invisible from the point-of-view of the  torsion sheaf invariants $\Omega(\gamma)$.
 
 \begin{assumption}
We assume the identity \eqref{lausanne}, and that moreover the invariants $\GV(0,\beta)$ vanish for all but finitely many classes $\beta\in H_2(X,\bZ)$.
\end{assumption}  
 
We can formulate a RH problem in exactly the same way as we did for the special case of the refined conifold above. It depends on a point $\xi\in \bT$ and we set
\[\xi(\beta)=\exp (\vartheta(\beta)), \qquad \xi(\delta)=\exp(\varphi).\]
We refrain from writing the problem out in detail since the formulae are a little cumbersome. In any case, it is clear that a solution can be obtained by superposing copies of the functions $B_n$ and $D_n$.
As in Section \ref{gengen} there is now no canonical quadratic refinement, but since these only effect the Joyce function by the addition of quadratic functions of the variables $\theta, \phi$ we can conclude that the Joyce function must be \[J(v,w,\theta,\phi)=\frac{1}{6w^4} \cdot \sum_{\beta\in H_2(X,\bZ)}\GV(0,\beta) \cdot (v(\beta)\phi-w\theta(\beta))^3\cdot\big(1-e^{-2\pi i v(\beta)/w}\big)^{-1},\]
up to the addition of linear terms in $\theta, \phi$. The 
prepotential  is then
\[F(v,w)=-\frac{w^2}{(2\pi i)^3}\cdot \sum_{\beta\in H_2(X,\bZ)} \GV(0,\beta)\cdot \Li_3\big(e^{2\pi i v(\beta)/w}\big),\]
up to the addition of quadratic functions of the linear variables $v(\beta)$. Note that by the Aspinwall-Morrison covering formula \cite{AM}, on the locus $w=1$ the prepotential coincides up to a constant factor with the  genus 0  Gromov-Witten  generating function.


\section{The case of the A$_2$ quiver}
\label{sten}

In this final section we summarise the results of \cite{A2} on the BPS structures arising from DT theory applied to the A$_2$ quiver. At present these are the only examples of  coupled BPS structures where solutions to Problem \ref{dtsect} can be constructed. The method we use owes much to the paper of Gaiotto, Moore and Neitzke \cite{GMN2}. The resulting Joyce structure can be described relatively explicitly, and in particular, the Joyce form can be calculated. A remarkable fact is that this Joyce structure is compatible, in the sense of Section \ref{compatible}, with the polynomial Frobenius structure associated to the A$_2$ root system. For further details on the material of this section we refer the reader to  \cite{A2}.

\subsection{Quadratic differentials}

Let us consider a meromorphic quadratic differential \[\phi(x)=\varphi(x) dx^{\tensor 2}\]
 on the Riemann surface $\bP^1$ having a single pole of order $7$ at the point $x=\infty$, and three simple zeroes.
It is easy to see  \cite[Section 12.1]{BS} that any such differential can be put in the  form
\begin{equation}\label{phir}\phi(x)=(x^3 + ax +b) dx^{\tensor 2}\end{equation}
 by applying an automorphism of $\bP^1$. 
 Away from the zeroes and poles of $\phi(x)$ there is a distinguished local co-ordinate on $\bP^1$
\[w(x)=\pm \int_*^{x} \sqrt{\varphi(u)} \, du\]
 in terms of which $\phi(x)$ takes the form $dw^{\tensor 2}$. Such a co-ordinate is uniquely determined up to transformations of the form $w\mapsto \pm w + c$.
The \emph{horizontal foliation} determined by $\phi(x)$   then consists of the  arcs on which $\Im(w)$ is constant. This foliation has singularities at the zeroes and poles of $\phi(x)$.
 Local computations summarised in \cite[Section 3.3]{BS} show that:
 \begin{itemize}
\item[(i)] there are three horizontal arcs emanating from each of the three simple zeroes;\smallskip

\item[(ii)] there are five tangent distinguished directions at the pole $x=\infty$, and  an open  neighbourhood $\infty\in U\subset \bP^1$ such that all horizontal trajectories entering $U$  approach $\infty$  along one of the distinguished directions.
\end{itemize}

Following \cite[Section 6]{BS} we take the 
 real oriented blow-up  of the surface $\bP^1$ at the point $\infty$ corresponding to the unique pole of  $\phi(x)$. Topologically the  resulting surface $\bS$ is a disc. The distinguished directions  at the pole determine a subset of five  points $\bM\subset \partial \bS$ of the boundary of this disc. The pair $(\bS,\bM)$ is an example of a \emph{marked bordered surface}. The horizontal foliation of $\bP^1$ lifts to a foliation on the surface $\bS$, with singularities at the points $\bM\subset \partial \bS$ and the zeroes of $\phi(x)$.

The quadratic differential \eqref{phir} determines a  double cover
\begin{equation}
\label{peep}\pi\colon X\to \bP^1,\end{equation}
branched at the zeroes and pole of $\phi(x)$, on which there is a well-defined global choice of square-root of $\phi(x)$. This is nothing but the projectivisation of the affine elliptic curve
\[X^\circ=\big\{(x,y)\in \bC^2: y^2=x^3+ax+b\big\}.\]
Taking the periods of the square-root of $\phi(x)$ defines a group homomorphism
\begin{equation}
\label{stuffy}Z\colon H_1(X,\bZ)\to \bC, \qquad Z(\gamma)=\int_{\gamma} \sqrt{\phi(x)}\in \bC.\end{equation}
The differential $\phi$ is called \emph{generic}   if the image of $Z$ is not contained in a one-dimensional real subspace of $\bC$.

A horizontal trajectory of $\phi(x)$  is said to be \emph{finite-length} if it never approaches the pole $x=\infty$. In our situation any such trajectory necessarily connects two distinct simple zeroes of $\phi(x)$, and is known as a \emph{saddle connection}. The inverse image of  a saddle connection under the double cover
\eqref{peep}
is a cycle $\gamma$, which can be canonically oriented by insisting  that
$Z(\gamma)\in \bR_{>0}$.
This then defines a  homology class in the group $H_1(X,\bZ)$. See \cite[Section 3.2]{BS} for more details.\footnote{For the purposes of comparison with the general situation of \cite{BS} involving the hat-homology group $H_1(X_s^\circ,\bZ)^-$, note that the group $H_1(X_s,\bZ)$ coincides with its $-1$ eigenspace under the action of the covering involution of \eqref{peep}; indeed the $+1$ eigenspace can be identified with the homology of the quotient $\bP^1$, which vanishes; moreover, puncturing $X_s$ at the inverse image of the pole $\infty\in \bP^1$ also leaves the first homology  group unchanged.}

More generally we can consider trajectories of the differential $\phi(x)$ of some phase $\theta\in \bR$. By definition these are arcs which make a constant angle $\pi \theta$ with the horizontal foliation. Alternatively one can view them as horizontal trajectories for the rescaled quadratic differential $e^{-2\pi i \theta}\cdot \phi(x)$. Once again, these finite-length trajectories $\gamma\colon [a,b]\to \bC$ define homology classes in $ H_1(X,\bZ)$, with the orientation convention being that
$Z(\gamma)\in \bR_{>0}\cdot e^{\pi i \theta}$.

The quadratic differential \eqref{phir} is said to be \emph{saddle-free} if it has no finite-length horizontal trajectories. This is an open condition. The horizontal foliation of a saddle-free differential splits the surface $\bP^1$ into a union of domains called  \emph{horizontal strips} and \emph{half-planes}. In the present case we obtain five half-planes and two horizontal strips. The resulting trajectory structure  on the blown-up surface $\bS$ is  illustrated in Figure \ref{bel2}. The crosses denote zeroes of the differential, and the black dots are the points of $\bM$.

\begin{figure}[ht]
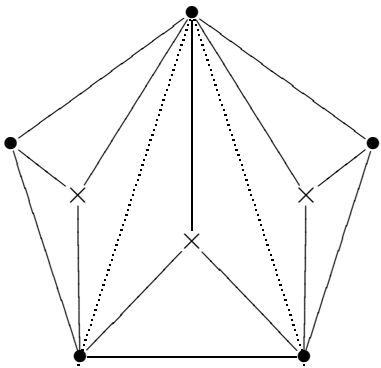

\begin{center}
  \begin{tabular}{c}
\xy /l1.2pc/:
(0,-5)*{\bullet}="1";
(4.76,-1.55)*{\bullet}="2";
(2.94,4.05)*{\bullet}="3";
(-2.94,4.05)*{\bullet}="4";
(-4.76,-1.55)*{\bullet}="5";
(0,1)*{\times}="6";
(3,-.2)*{\times}="7";
(-3,-.2)*{\times}="8";
{"1"\PATH~={**@{-}}'"2"},
{"2"\PATH~={**@{-}}'"3"},
{"3"\PATH~={**@{-}}'"4"},
{"4"\PATH~={**@{-}}'"5"},
{"5"\PATH~={**@{-}}'"1"},
{"2"\PATH~={**@{-}}'"7"},
{"3"\PATH~={**@{-}}'"6"},
{"4"\PATH~={**@{-}}'"6"},
{"5"\PATH~={**@{-}}'"8"},
{"4"\PATH~={**@{-}}'"8"},
{"1"\PATH~={**@{-}}'"8"},
{"1"\PATH~={**@{-}}'"6"},
{"1"\PATH~={**@{-}}'"7"},
{"3"\PATH~={**@{-}}'"7"},
{"1"\PATH~={**@{.}}'"3"},
{"1"\PATH~={**@{.}}'"4"},
\endxy
\end{tabular}
\end{center}
\caption{The separating trajectories of a saddle-free differential   \eqref{phir}. \label{bel2}}
\end{figure}

Taking one trajectory from the interior of each horizontal strip defines a triangulation of the marked bordered surface $(\bS,\bM)$ called the \emph{WKB triangulation} (see \cite[Section 10.1]{BS} for details). In our case there are exactly two  edges, which are the  two dashed edges  in Figure \ref{bel2}. 
As explained in \cite[Section 3.6]{BS}, each of the two horizontal strips contains a unique finite-length trajectory with  phase in the interval $(0,1)\subset \bR$, and the corresponding classes $\gamma_i\in H_1(X_s,\bZ)$ determine a basis $(\gamma_1,\gamma_2)$, whose elements are naturally  indexed by the edges of the WKB triangulation.

\subsection{Variation of BPS structures}
\label{matrix6c}

The base of our variation will be the complex manifold
\[M=\big\{(a,b)\in \bC^2: 4a^3+27b^2\neq 0\big\}.\]
It is proved in \cite{BQS} that this coincides with  the quotient of the space of stability conditions on the bounded derived category of the Ginzburg algebra of the A$_2$ quiver, by the group of autoequivalences generated by spherical twists in the two vertex simples. This can also be viewed as a special case of the main result of \cite{BS}. 

For each point $p=(a,b)\in  M$ there is a  meromorphic  quadratic differential $\phi_p(x)$ on $\bP^1$ given by the formula \eqref{phir}, and a corresponding double cover $\pi\colon X_p\to \bP^1$.
The Gauss-Manin connection on the family of homology groups \begin{equation}
\label{nov}\Gamma_p=H_1(X_p,\bZ)\isom \bZ^{\oplus 2}\end{equation}
gives them the structure of a local system of lattices over $M$. In concrete terms, we can represent classes in $H_1(X_p,\bZ)$  by the inverse images under the double cover \eqref{peep} of paths in $\bC$ connecting the zeroes of $\phi_p(x)$, and the Gauss-Manin connection is then obtained by holding these paths locally constant as $\phi_p(x)$ varies.

There is a natural variation of BPS structures over the space $M$ determined by the family of quadratic differentials \eqref{phir}. At a point $p\in M$ corresponding to a generic differential  the corresponding BPS structure $(\Gamma_p,Z_p,\Omega_p)$ is defined  as follows:
  
\begin{itemize}
\item[(i)]  the charge lattice is $\Gamma_p=H_1(X_p,\bZ)$ with its intersection form $\<-,-\>$;\smallskip

\item[(ii)] the central charge  $Z_p\colon \Gamma_p\to \bC$ is defined by
\begin{equation}
\label{cycling}
Z_p(\gamma)=\int_{\gamma} \sqrt{\phi_p(x)}\, dx\in \bC;\end{equation}
\item[(iii)]  the BPS invariant  $\Omega_p(\gamma)\in \{0,1\}$ is nonzero precisely if the differential $\phi_p(x)$ has a finite-length trajectory defining the given class $\gamma\in \Gamma_p$.
\end{itemize}

Condition (c) needs modification when the differential $\phi_p(x)$ is non-generic,  but  in fact there is no need to give an explicit description of the BPS invariants $\Omega_p(\gamma)$ in this  case, since what is important is the BPS automorphisms $\bS_p(\ell)$, and using the wall-crossing formula, these are determined by the BPS invariants at nearby generic points of $M$.

Suppose that $p\in M$ corresponds to a  saddle-free and generic differential $\phi_p(x)$, and let $(\Gamma,Z,\Omega)$ be the associated BPS structure. As explained above, the lattice \eqref{nov} then has a distinguished basis $(\gamma_1,\gamma_2)$ indexed by the edges of the WKB triangulation. We can canonically order these edges by insisting that $\<\gamma_1,\gamma_2\>=1$. 
Define  $z_i=Z_p(\gamma_i)\in \bC^*$. The orientation conventions discussed above ensure that these points lie in the upper half-plane. One can then show that the BPS invariants are as follows:
 \smallskip
 \begin{itemize}
 \item[(a)]  if $\Im (z_2/z_1)<0$ then $\Omega_p(\pm \gamma_1)=\Omega_p(\pm \gamma_2)=1$ with all others zero; 
 \smallskip
 \item[(b)]  if $\Im (z_2/z_1)>0$ then $\Omega_p(\pm \gamma_1)=\Omega_p(\pm(\gamma_1+\gamma_2))=\Omega_p(\pm \gamma_2)=1$ with all others zero.
 \end{itemize}
 The two cases are illustrated in Figure \ref{autumn}. In both cases there is a distinguished  quadratic refinement
 $\sigma_p\in \bT_{p,-}$, uniquely defined by the property that
  $\sigma_p(\gamma)=-1$ for every active class $\gamma\in \Gamma_p$. It is  easy to check that this choice  extends to  a continuous family  of quadratic refinements $\sigma_p\in \bT_{p,-}$ for all points $p\in M$. 
 
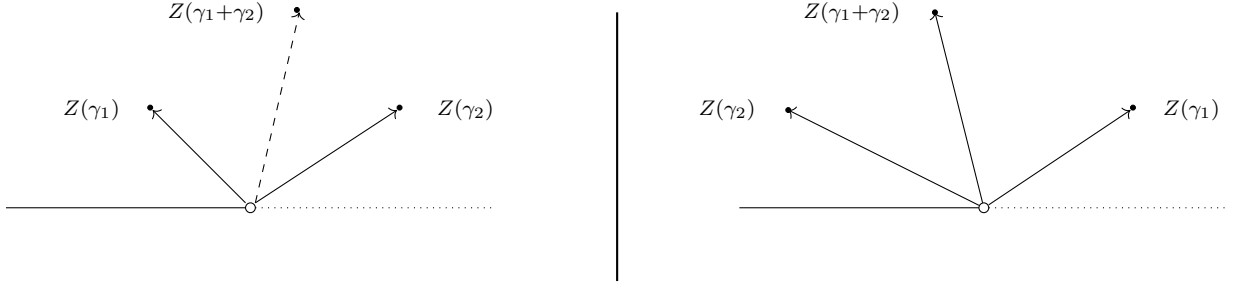
\begin{figure}
\begin{tikzpicture}[scale=.65]
\draw (0,0)--(4.9,0);
\draw[dotted] (5.1,0)--(10,0);
\draw (5,0) circle [radius=0.1];
\draw[->] (5.1,0.1) -- (8,2);
\draw[->] (4.9,0.1) -- (3,2);
\draw[->,dashed](5.1,0.1)--(6,4);
\draw (9.4,2) node { $\scriptstyle Z(\gamma_2)$};
\draw (1.75,2) node {$\scriptstyle Z(\gamma_1)$};
\draw (4.3,4) node {$\scriptstyle Z(\gamma_1+\gamma_2)$};
\draw[fill] (8.05,2.05) circle [radius=0.05];
\draw[fill] (2.95,2.05) circle [radius=0.05];
\draw[fill] (5.95,4.05) circle [radius=0.05];
\draw[thick](12.5,-1.5)--(12.5,4);
\draw (19.9,0)--(15,0);
\draw[dotted] (20.1,0)--(25,0);
\draw (20,0) circle [radius=0.1];
\draw[->] (20.1,0.06) -- (23,2);
\draw[->] (19.9,0.05) -- (16,2);
\draw[->] (19.975,0.1) -- (19,4);
\draw (24.25,2) node { $\scriptstyle Z(\gamma_1)$};
\draw (14.75,2) node { $\scriptstyle Z(\gamma_2)$};
\draw (17.3,4) node { $\scriptstyle Z(\gamma_1+\gamma_2)$};
\draw[fill] (23.05,2.05) circle [radius=0.05];
\draw[fill] (16,2) circle [radius=0.05];
\draw[fill] (19,4) circle [radius=0.05];
\end{tikzpicture}
\caption{The ray diagrams for the BPS structures associated to the A$_2$ quiver.\label{autumn} }
\end{figure}

We can use the quadratic refinement $\sigma_p$  and the map \eqref{bloxy} to identify the twisted torus $\bT_{p,-}$ with the standard torus $\bT_{p,+}$. Under this identification the birational automorphism \eqref{bir} becomes the birational automorphism of $\bT_{p,+}$ defined by
 \[\stokes_p(\ell)^*(y_\beta)=y_\beta\cdot \prod_{Z(\gamma)\in \ell}   (1+y_{\gamma})^{\Omega(\gamma)\<\gamma,\beta\>}.\]

The fact that the above BPS structures fit together to make a variation of BPS structures boils down to the wall-crossing formula
\begin{equation}
\label{pent}C_{\gamma_1} \circ C_{\gamma_2} = C_{\gamma_2} \circ C_{\gamma_1+\gamma_2} \circ C_{\gamma_1},\end{equation}
where for each class $\gamma\in \Gamma$  we defined a birational automorphism $C_\gamma\colon \bT_{p,+}\dashrightarrow \bT_{p,+}$ by
\begin{equation}
\label{missed}C_\gamma^*(y_\beta)=y_\beta\cdot (1+y_\gamma)^{\<\gamma,\beta\>}.\end{equation}
The equation \eqref{pent} is the  pentagon identity familiar from cluster theory.

\subsection{Associated family of opers}

Consider the second-order linear differential equation
\begin{equation}\label{de} y''(x)=Q(x,\hbar) \cdot y(x), \qquad Q(x,\hbar)=\hbar^{-2}\cdot {Q_0(x)}+\hbar^{-1}\cdot Q_1(x) +  Q_2(x),\end{equation}
where primes denote differentiation with respect to the complex variable $x\in \bC$, and the terms in the potential $Q(x,\hbar)$ are
\begin{equation}
\label{pot}Q_0(x)=x^3+ax+b, \qquad Q_1(x)=\frac{p}{x-q}+r, \qquad \end{equation}\begin{equation}
\label{pot2}Q_2(x)=\frac{3}{4(x-q)^2}+\frac{r}{2p(x-q)}+\frac{r^2}{4p^2}.\end{equation}
We view the equation \eqref{de}  as being specified by a point of the complex manifold \[W=\Big\{(a,b,q,p,r)\in \bC^5: p^2=q^3+aq+b\text{ and } 4a^3+27b^2 \neq 0, \ p\neq 0\Big\},\]
together with a nonzero complex number $\hbar\in \bC^*$.

The form of \eqref{pot2} is chosen to ensure that the  point $x=q$ is an apparent singularity of the equation \eqref{de}: transporting any solution around this point changes its sign. Thus the generalised monodromy  of the equation consists only of the Stokes data at the irregular singularity $x=\infty$. The Stokes sectors  are the sectors bounded by the rays through the fifth roots of $-\hbar^2$. In each such sector there is a distinguished subdominant solution to \eqref{de} up to scale, with the property that it decays exponentially as $x\to \infty$ in the sector. After choosing an initial sector, the collection of these subdominant solutions defines a point of the space
\[V=\Big\{\psi\colon \bZ/5\bZ\to \bP^1: \psi(i+1)\neq \psi(i)\text{ for all }i\in \bZ/5\bZ\Big\} \Big / \PGL_2,\]
which is easily seen to be a two-dimensional complex manifold.

For each $\hbar\in \bC^*$ there is then  a holomorphic monodromy map
\[F(\hbar)\colon W \to V,\]
sending the equation \eqref{pot} to its Stokes data. 
This map depends on a labelling of the Stokes sectors for the equation \eqref{de}, which in concrete terms amounts to a choice of fifth root of $\hbar^2$.
We prove in \cite{A2} that this map is invariant under the  isomonodromy flows:
\begin{equation}
\label{fir1}-\frac{1}{\hbar}\frac{\partial}{\partial r}+\bigg(\frac{\partial}{\partial b}+\frac{1}{2p}\frac{\partial}{\partial p} +\frac{r}{2p^2} \frac{\partial}{\partial r}\bigg),\end{equation}
\begin{equation} \label{sec1}-\frac{2p}{\hbar} \frac{\partial}{\partial q}-\frac{3q^2+a}{\hbar}\frac{\partial}{\partial p}+\bigg(\frac{\partial}{\partial a}-q\frac{\partial}{\partial b}-\frac{r}{p}\frac{\partial}{\partial q}- \frac{r(3q^2+a)}{2p^2}\frac{\partial}{\partial p}-\frac{r^2}{2p^3} (3q^2+a) \frac{\partial}{\partial r}\bigg).\end{equation}
In fact the  flow \eqref{fir1} leaves the equation \eqref{de} itself invariant. In the special case $r=0$ the second flow \eqref{sec1} reduces to the standard isomonodromy flow often studied in connection with the Painlev{\'e} I equation. 
 
\subsection{Riemann-Hilbert problem}
It is well known that the monodromy manifold $V$ has birational co-ordinate systems
\begin{equation}
\label{coord}X_T\colon V\dashrightarrow (\bC^*)^2\end{equation}
usually referred to as Fock-Goncharov co-ordinates, and indexed by the triangulations $T$ of a regular pentagon. The maps for different triangulations differ by post-composition by explicit birational automorphisms of $(\bC^*)^2$.

On the other hand, for generic $\lambda\in \bC^*$, as we explained above, the horizontal trajectory structure of the quadratic differential \begin{equation}\label{hq}\lambda^{-2}\cdot Q_0(x) \, dx^{\tensor 2}= \lambda^{-2}\cdot (x^3+ax+b) \, dx^{\tensor 2}\end{equation} naturally determines such a triangulation, namely the WKB triangulation. This triangulation is well-defined for all $\lambda\in \bC^*$ not lying on finitely many rays, but jumps discontinuously when these rays are crossed. 

Fix a point $(a,b,q,p,r)\in W$, and let $(\Gamma,Z,\Omega)$ be the  BPS structure defined by the quadratic differential $Q_0(x) dx^{\tensor 2}$ as above. Define an element $\xi\colon \Gamma\to \bC^*$ of the corresponding algebraic torus $\bT_+$ by setting 
\begin{equation}
\label{xi}\xi(\gamma)= \exp\Bigg(\int_{\gamma}\frac{Q_1(x)\, dx}{2\sqrt{Q_0(x)}}\Bigg)=\exp\Bigg(\int_{\gamma} \bigg(\frac{p}{x-q} + r\bigg) \frac{dx}{2y}\Bigg)\in \bC^*.\end{equation}
The basic claim is that the map which sends $\hbar\in \bC^*$ to the Fock-Goncharov co-ordinates   corresponding  to the WKB triangulation of the differential \eqref{hq}, applied to the monodromy of the equation \eqref{de}, gives a solution to the RH problem associated to the BPS structure $(\Gamma,Z,\Omega)$, with constant term given by \eqref{xi}.

To be completely precise, let us  fix  a non-active ray $r=\bR_{>0}\cdot \lambda$ spanned by an element $\lambda\in \bC^*$. The non-active condition  is equivalent to the statement that the differential \eqref{hq} is saddle-free. The marked bordered surface $(\bS(\lambda),\bM(\lambda))$ for this differential can be identified with the disc in $\bC$, with the marked points being the directions of the fifth roots of $\lambda^2$. The saddle-free condition ensures that the differential \eqref{hq} defines a WKB triangulation $T(\lambda)$ of this surface.

Consider now a point $\hbar\in \bH_r$. Since $\Re(\hbar/\lambda)>0$, there is a natural bijection between the fifth roots of $\lambda^2$ and $\hbar^2$. Thus we can identify the Stokes sectors of the equation \eqref{de} with the marked points $\bM(\lambda)$. Note that the Fock-Goncharov co-ordinates $(X_1,X_2)\in (\bC^*)^2$ for the triangulation $T(\lambda)$ are indexed by the edges of the triangulation, which also index a natural basis for the lattice $\Gamma$. Thus we can view the map \eqref{coord} as taking values in the torus $\bT_+$, and consider the map
\begin{equation}\label{map}X_r \colon \bH_r\to \bT_+,\end{equation}
which sends a point $\hbar\in \bH_r$ to the Fock-Goncharov co-ordinates of the equation \eqref{de} with respect to the WKB triangulation determined by the differential \eqref{hq}. In \cite{A2} we prove

\begin{thm}
\label{two}
The maps  \eqref{map} give a meromorphic solution to the RH problem associated to the BPS structure $(\Gamma,Z,\Omega)$, with constant term given by \eqref{xi}. \qed
\end{thm}

This proof of Theorem \ref{two} boils down to the following  three statements:
\begin{itemize}
\item[(i)] the WKB approximation shows that as $\hbar\to 0$ in the half-plane $\bH_r$, the Fock-Goncharov co-ordinates satisfy
\[X_{r,\gamma}(\hbar)\cdot \exp(Z(\gamma)/\hbar)\to \xi(\gamma);\]
\item[(ii)] the homogeneity of the potential \eqref{pot} under the rescaling of all variables with various weights shows  that as $\hbar\to \infty$ in the half-plane $\bH_r$, with  $\xi\in \bT_+$ held fixed
\[X_{r}(\hbar)\to {\rm constant};\]
\item[(iiii)]
given an active ray $\ell=\bR_{>0}\cdot \lambda$ for the BPS structure $(\Gamma,Z,\Omega)$, corresponding to a differential \eqref{hq} with a horizontal saddle trajectory with class $\gamma\in \Gamma$, the two systems of Fock-Goncharov co-ordinates defined by small clockwise and anti-clockwise perturbations of $\lambda\in \bC^*$ differ by post-composition by the birational transformation $C_\gamma$ of \eqref{missed}.
\end{itemize}

We unfortunately have no results on the uniqueness of the solution of Theorem \ref{two}.

\subsection{Joyce structure}

Consider the  bundle $\pi \colon \bT\to M$ whose fibres are the tori
\[\bT_{p,+}=\Hom_{\bZ}(\Gamma_p,\bC^*).\]
There are obvious  local co-ordinates on the fibres $\bT_{p,+}$  given by
\[\theta_i=\log \xi(\gamma_i)\in \bC, \qquad \xi\in \bT_{p,+},\]
and we therefore obtain local co-ordinates $(z_1,z_2,\theta_2,\theta_2)$ on the total space $\bT$.
Let $\pi\colon W\to M$ be the obvious projection map. The expression \eqref{xi} defines a holomorphic map $\Theta$ fitting into the diagram
 \begin{equation}\label{diagre}
\xymatrix@C=1.5em{
W \ar^{\Theta}[rr] \ar_{\pi}[dr] && \bT\ar^{\pi}[dl] \\
&M
} \end{equation}
This map is in fact  an open embedding, and so we can push forward the isomonodromy flows \eqref{fir1} and \eqref{sec1}   to obtain a non-linear connection on the bundle of tori $\pi \colon \bT\to M$. The following result is proved in \cite{A2}.

\begin{thm}
\label{one}
In the co-ordinates $(z_1,z_2,\theta_1,\theta_2)$ the push-forward of the isomonodromy flows  via the map $\Theta$ take the Hamiltonian form
\begin{equation}
\label{floww}\frac{\partial}{\partial z_i} + \frac{1}{\hbar} \cdot \frac{\partial}{\partial \theta_i} + \frac{\partial^2 J} {\partial \theta_i \partial \theta_1}\cdot\frac{\partial}{\partial \theta_2}-\frac{\partial^2 J}{\partial \theta_i\partial \theta_2}\cdot \frac{\partial}{\partial \theta_1},\end{equation}
where $J\colon \bT\to \bC$ is a meromorphic function with no poles on  the locus $\theta_1=\theta_2=0$. In fact\[\frac{1}{2\pi i} \cdot \Theta^*(J)=-\frac{1}{4\Delta p} \big( 2ap^2+3p(3b-2aq)r+(6aq^2-9bq+4a^2)r^2 - 2apr^3\big),\]
where we set $\Delta=4a^3+27b^2$. 
\end{thm}

We can now describe the Joyce structure on $M$ corresponding to the miniversal variation of BPS structures of Section \ref{matrix6c}. As explained in Section \ref{recon}, most of the data carries over easily from the variation of BPS structures. In terms of the local co-ordinates $(z_1,z_2)$ on $M$ considered above, the connection $\nabla$ is the one in which the $z_i$ are flat, and the bundle of lattices $\Gamma_M\subset \cT_{M}^*$ is spanned by $dz_1$ and $dz_2$. Note that by the formula \eqref{cycling}, rescaling $(a,b)$ with weights $(4,6)$ has the effect of rescaling the co-ordinates  $(z_1,z_2)$ with weight $5$, so the Euler vector field is
\[E=z_1\cdot \frac{\partial}{\partial z_1}+z_2\cdot \frac{\partial}{\partial z_2} =\frac{4a}{5}  \cdot \frac{\partial}{\partial a} + \frac{6b}{5} \cdot \frac{\partial}{\partial b}.\]

The only non-trivial step is the construction of the Joyce function $J\colon \cT_M\to \bC$. Recall that the solutions to the RH problem are   flat sections of the deformed Joyce connection. Since these solutions are functions of the monodromy of the equation \eqref{de},  the deformed Joyce connection must coincide with the isomonodromy connection. Thus the Joyce function is precisely the function $J$ of Theorem \ref{one}, and the differential equation \eqref{master} coincides with \eqref{floww}. 
The isomonodromy flows define Ehresmann connections on the two bundles appearing in the commutative diagram \eqref{diagre} which  are flat by construction, since they are the pullback via the monodromy map of the trivial connection on the bundle $\pi\colon V\times M\to M$. It follows from this that the function $J$ of Theorem \ref{one} satisfies the differential equation \eqref{fl}, at least up to the addition of terms independent of $\theta_i$, which are shown to vanish by the more careful analysis in \cite{A2}.

\subsection{Linearization}

Consider the affine space $\bC^2$ with linear co-ordinates $(a,b)$. Viewing this as the unfolding space of the A$_2$ singularity $x^3=0$ gives rise to 
a Frobenius structure  \cite[Example 1.4]{D2} with
\[e=\frac{\partial}{\partial b}, \qquad \frac{\partial}{\partial a}*\frac{\partial}{\partial a}=-\frac{a}{3}\cdot \frac{\partial}{\partial b},\]
\[g=\frac{1}{3} \cdot ( da\tensor db + db\tensor da),\qquad E= \frac{2a}{3}\cdot \frac{\partial}{\partial a} +  b \cdot \frac{\partial}{\partial b}.\]
We refer to this as the A$_2$ Frobenius structure. The conformal dimension is $d=\tfrac{1}{3}$. The discriminant locus is the submanifold cut out by the equation $\Delta= 4a^3+27b^2=0$. Thus by restriction we obtain a discriminant-free Frobenius structure on the open submanifold $M\subset \bC^2$.  
The diamond product $X\diamond Y=E^{-1} * X * Y$ is given by
\[\Delta\cdot \frac{\partial}{\partial a}\diamond \frac{\partial}{\partial a}=6a^2\frac{\partial}{\partial a}-9ab\frac{\partial}{\partial b}, \qquad \Delta\cdot \frac{\partial}{\partial b}\diamond \frac{\partial}{\partial b}=-18a\frac{\partial}{\partial a}+27b\frac{\partial}{\partial b},\]
\[\Delta\cdot \frac{\partial}{\partial a}\diamond \frac{\partial}{\partial b}=27b\frac{\partial}{\partial a}+6a^2\frac{\partial}{\partial b}.\]
The following result is quite striking, although until  further examples have been calculated it is perhaps too early to say whether it is just a coincidence.

\begin{thm}
The  Joyce structure on $M$ considered above is compatible, in the sense of Definition \ref{blog}, with the A$_2$ Frobenius structure  restricted to the open subset $M\subset \bC^2$.
\end{thm}

\begin{proof}
It was proved in \cite{A2} that the Joyce form of the above Joyce structure is
\[g=\frac{2\pi i}{5} \cdot (da\tensor db+ db\tensor da),\]
which agrees with the metric of the Frobenius structure up to scale. The Euler vector fields of the two structures also agree up to scale.
Dubrovin showed  \cite[Proposition 5.1]{D3} that the odd periods of the A$_2$ Frobenius structure are given by the periods \eqref{stuffy}. Reversing the argument of Proposition \ref{coco}(e) it follows  that the diamond product of the Joyce structure agrees with the twisted product of the Frobenius structure up to scale. Thus the two structures are compatible.
\end{proof}

\bibliographystyle{amsplain}

\end{document}